\documentclass{amsart}
\usepackage{amssymb,amsmath}
\usepackage{euscript}
\usepackage{graphicx}
\usepackage{tikz}
\usepackage{tikz-cd}
\usepackage[pdftex]{hyperref}
\usetikzlibrary{matrix}
\usepackage{mathrsfs}

\newcommand{\bB}{{\mathbb B}}
\newcommand{\bC}{{\mathbb C}}
\newcommand{\bE}{{\mathbb E}}
\newcommand{\bI}{{\mathbb I}}
\newcommand{\bP}{{\mathbb P}}

\newcommand{\bR}{{\mathbb R}}
\newcommand{\bT}{{\mathbb T}}
\newcommand{\bZ}{{\mathbb Z}}
\newcommand{\bN}{{\mathbb N}}
\newcommand{\bS}{{\mathbb S}}
\newcommand{\bF}{{\mathbb F}}
\newcommand{\bG}{{\mathbb G}}
\newcommand{\bH}{{\mathbb H}}

\newcommand{\bX}{{\mathbb X}}
\newcommand{\bY}{{\mathbb Y}}

\newcommand{\bW}{{\mathbb W}}
\newcommand{\bM}{{\mathbb M}}

\newcommand{\cF}{\mathcal F}

\newcommand{\cA}{\mathcal A}
\newcommand{\cB}{\mathcal B}
\newcommand{\cC}{\mathcal C}
\newcommand{\cD}{\mathcal D}

\newcommand{\cS}{\mathcal S}

\newcommand{\cM}{\mathcal M}
\newcommand{\cO}{\mathcal O}
\newcommand{\cP}{\mathcal P}

\newcommand{\cR}{\mathcal R}
\newcommand{\cT}{\mathcal T}

\newcommand{\scrC}{\EuScript C}

\newcommand{\scrV}{\EuScript V}

\newcommand{\Rbar}{\overline{\cR}}
\newcommand{\Face}{\operatorname{Face}}

\newcommand{\Flow}{\operatorname{Flow}}
\newcommand{\Bord}{\operatorname{Bord}}

\newcommand{\Ob}{\operatorname{Ob}}
\newcommand{\id}{\operatorname{id}}

\newcommand{\colim}{\operatorname*{colim}}
\newcommand{\hocolim}{\operatorname{hocolim}}

\newcommand{\codim}{\operatorname{codim}}
\newcommand{\fr}{\operatorname{fr}}

\newcommand{\Map}{\operatorname{Map}}
\newcommand{\Set}{\operatorname{Set}}
\newcommand{\op}{\operatorname{op}}

\newcommand{\Fun}{\operatorname{Fun}}

\newcommand{\Mbar}{\overline{\cM}}
\newcommand{\Sbar}{\overline{\cS}}
\newcommand{\Fbar}{\overline{\cF}}

\newcommand{\Stab}{\operatorname{Stab}}

\newcommand{\coker}{\operatorname{coker}}
\newcommand{\Cone}{\operatorname{C}}
\def\co{\colon\thinspace}

\renewcommand{\th}{\textsuperscript{th} \hspace{.01pt}}

\newcommand{\htp}{\simeq}
\newcommand{\ob}{\textrm{ob}}

\newcommand{\Ho}{\textrm{Ho}}

\newcommand{\Orb}{\textrm{Orb}}

\newcommand{\strict}{\textrm{strict}}

\newcommand\noloc{%
  \nobreak
  \mspace{6mu plus 1mu}
  {:}
  \nonscript\mkern-\thinmuskip
  \mathpunct{}
  \mspace{2mu}
}

\newtheorem{thm}{Theorem}[section]
\newtheorem{cor}[thm]{Corollary}
\newtheorem{lem}[thm]{Lemma}

\newtheorem{prop}[thm]{Proposition}
\newtheorem{defin}[thm]{Definition}
\newtheorem{def-lem}[thm]{Definition-Lemma}

\theoremstyle{remark}

\newtheorem{rem}[thm]{Remark}

\newtheorem{example}[thm]{Example}

\newtheorem{assu}{Assumption}

\numberwithin{equation}{section}

\title{Foundation of Floer homotopy theory I: Flow categories}
\author{Mohammed Abouzaid}
\author{Andrew J. Blumberg}

\begin{document}

\begin{abstract}
We construct a stable infinity category with objects flow categories
and morphisms flow bimodules; our construction has many flavors,
related to a choice of bordism theory, and we discuss in particular
framed bordism and the bordism theory of complex oriented derived
orbifolds. In this setup, the construction of homotopy types
associated to Floer-theoretic data is immediate: the moduli spaces of
solutions to Floer's equation assemble into a flow category with
respect to the appropriate bordism theory, and the associated Floer
homotopy types arise as suitable mapping spectra in this category.
The definition of these mapping spectra is sufficiently explicit to
allow a direct interpretation of the Floer homotopy groups as Floer
bordism groups. In the setting of framed bordism, we show that the
category we construct is a model for the category of spectra.  We
implement the construction of Floer homotopy types in this new
formalism for the case of Hamiltonian Floer theory.
\end{abstract}

\maketitle
\tableofcontents

\section{Introduction}
\label{sec:introduction}

With very few exceptions, the applications of Gromov and Floer's theories of pseudo-holomorphic curves to symplectic and low dimensional topology have come from extracting \emph{virtual counts} (either integral, rational, or with $\bF_2$-coefficients) from those curves which have expected dimension $0$. Floer \cite{Floer1989} envisioned early on  that, whenever the moduli spaces of pseudo-holomorphic curves can be perturbed to be manifolds with corners, the virtual counts can be refined to take values in appropriate generalized cohomology theories, with the usual counts corresponding to ordinary cohomology. The first concrete proposal along these lines was given by Cohen, Jones, and Segal~\cite{Cohen1995}, who introduced the notion of a framed flow category that encapsulates the basic geometric output of Floer theory in the simplest situations, and associated a (stable) homotopy type to such a datum, which, in geometric situations, is called the \emph{Floer homotopy type}.

The appearance of significant applications of the Floer homotopy type
to problems in symplectic topology~\cite{Abouzaid2021a,Abouzaid2021b}
has revealed that Cohen, Jones, and Segal's formulation, while yielding the correct answer, is too cumbersome to
support the further development of the subject. This is the first of a
series of papers whose purpose is to provide new foundations for the
study of the Floer homotopy type which will both support the study of
(algebraic) operations and extend to more general geometric contexts,
eventually incorporating Floer theory for arbitrary immersed
Lagrangians. The fundamental point of view we take is encapsulated by
the slogan that the construction of the Floer homotopy type from
geometric data is a tautology.

The main task of this paper is to prove that suitable collections of
flow categories are the objects of a category enriched in spectra.
The Floer homotopy type then arises as maps from the unit object in
this category.  To carry this out in a fashion which implements our
slogan, the work of the paper is to construct a bordism-theoretic
model for the category of spectra, and more generally for categories
of modules over certain bordism rings that appear naturally in
pseudo-holomorphic curve theory. A surprising consequence of this
reworking is that it becomes possible to define the Floer homotopy
groups, reinterpreted as Floer bordism groups, without needing to
appeal to any homotopy-theoretic machinery, which we expect to be
particularly useful for applications.

Subsequent papers in this series will construct multiplicative
structures on these categories, for the purpose of formulating
algebraic structures at the spectral level.

\subsection{Derived orbifold bordism}
\label{sec:bordism-theories}

Our main constructions are formulated in terms 
of \emph{derived orbifold bordism}, since this will allow maximal applicability to geometric problems where the moduli spaces arising from Floer theory are not manifolds. Before introducing this notion, recall that an (effective) orbifold is a topological space equipped with charts given by the quotient of an open subset of an effective representation of a finite group, and transition maps given by embeddings of such open subsets. 
This notion goes back to Satake~\cite{Satake1956}, and we will use this point of view for all definitions, though some of our proofs will appeal to results that require the more sophisticated stack-theoretic approach to orbifolds.

A \emph{derived orbifold} $\bX$ consists of an orbifold vector bundle $T^- \bX$ on an orbifold $X$ which we refer to as the \emph{obstruction bundle}, together with a section $s$ of this vector bundle. We write $T^+ \bX$ for the tangent bundle of $X$, and define the \emph{tangent bundle of $\bX$} to be the virtual vector bundle
\begin{equation}
  T \bX \equiv (T^+ \bX, T^- \bX).
\end{equation}
We say that such a derived orbifold is {\em proper} when the zero-locus of $s$ is compact. More generally, we say that $\bX$ is proper over a base $B$ if the zero locus is proper over $B$; this notion will be essential in formulating Floer homotopy in a setting adequate for application to general symplectic manifolds, because moduli spaces of pseudo-holomorphic curves are proper over the positive real axis $[0,\infty)$ via the map which records their energy.
\begin{rem}
The reader familiar with Floer theory may reasonably object at this point that the outcome of Floer theory is a derived orbifold which is (i) not necessarily smooth, and (ii) not necessarily effective. Regarding the first point, the techniques of equivariant stable smoothing (due to Lashof \cite{Lashof1979}) 
  yield smooth structures which are unique after stabilization. For
  the second point, as noted by Bai and Xu~\cite{Bai2022}, Pardon's
  result on the existence of enough vector bundles~\cite{Pardon2019}
  implies that every derived orbifold is equivalent to one whose
  ambient orbifold is effective, so that little is lost by restricting
  to this class. Moreover, in geometric applications, one need not
  appeal to Pardon's abstract result, as the methods of
  \cite{Abouzaid2021b} naturally produce vector bundles whose total
  space is effective. 
\end{rem}

There are two fundamental notions that are essential to our use of derived orbifolds in the study of Floer theory:
\begin{enumerate}
\item {\em Equivalence of derived orbifolds}:  the total space of an
orbifold vector bundle $E$ over $\bX$ underlies a derived orbifold
$\bE$ with associated vector bundle given by the direct sum of the
pullback of $T^- \bX$ and of $E$, and with section given by the direct
sum of $s$ with the identity section of $E$.  We declare $\bE$ to be
equivalent to $\bX$.
\item {\em Null-bordism of a proper derived orbifold}: this is specified by a bounding
orbifold, equipped with an extension of the vector bundle and of the
section, again with the assumption that the $0$-locus is proper.
\end{enumerate}

The derived orbifold bordism groups are then given by the quotient of the
set of equivalence classes of proper derived orbifolds by those which
are null-bordant; addition in the group is given by disjoint union.

Unfortunately, the theory described in the previous paragraph is
trivial, since the product of an orbifold with the quotient of the
interval $[-1,1]$ by its nontrivial involution defines a
null-bordism. Geometric applications thus rely on recording two
additional pieces of data: (i) a choice of family of \emph{isotropy groups,}
which prescribes the possible isomorphism classes of groups that appear in
charts, and (ii) \emph{tangential structures} which record orientation
data on the tangent space. For example, one defines \emph{oriented
derived orbifolds} by recording the additional datum of an orientation
on $ T \bX$, and requiring equivalences and cobordisms to preserve
this structure. In a different direction, requiring that all isotropy
groups be trivial recovers ordinary bordism groups~\cite{Spivak2010}.  

\begin{rem}
The first appearance of bordism groups of derived orbifolds in the
literature is Joyce's formulation of Kuranishi bordism~\cite{Joyce2019}, which was identified as the coefficients of a
generalized homology theory by Pardon~\cite{Pardon2019}.
\end{rem}

We shall focus our discussion on the following two examples:
\begin{enumerate}
\item Framed manifold bordism: this is the theory where all isotropy groups are trivial, and the tangent space is assumed to be stably trivial. Under the Pontryagin-Thom construction, this corresponds to the initial cohomology theory given by stable cohomotopy.
\item Periodic complex-oriented derived orbifold bordism: in this theory, finite isotropy groups are allowed, and the tangent bundle is assumed to be stably complex. The importance of this example is that we expect that every closed symplectic manifold admits a \emph{Fukaya category} with such coefficients (though not all Lagrangians will support nontrivial objects). As a preliminary result, we prove in Appendix \ref{sec:hamilt-floer-theory-1} that every Hamiltonian dynamical system on a closed symplectic manifold determines a flow category equipped with this structure, which should be considered as the \emph{Hamiltonian Floer bordism type}.
\end{enumerate}

\begin{rem} \label{rem:complex_manifold_bordism}
  The case of complex cobordism (in its ordinary or periodic flavor), it particularly easy to extract from our construction: in the formulation of derived orbifold bordism, one has to specify which isomorphism classes of stabilisers are allowed (see Definition \ref{def:family_of_groups}). By admitting only the trivial group, one obtains the bordism theory of derived (complex oriented) manifolds which, because of transversality, is equivalent to the usual bordism theory of (complex oriented) manifolds.
\end{rem}

\subsection{Flow categories and bimodules}
\label{sec:flow-categories-1}

The definition of the bordism groups in
Section~\ref{sec:bordism-theories} uses orbifolds as generators and
orbifolds with boundary as relations; in the case of framed manifolds,
this recovers the stable homotopy groups of spheres. The starting
point of our study of flow categories is the idea that, by considering
framed manifolds with corners, we can extend this description of
homotopy groups to a description of stable homotopy types and beyond that to 
the spaces of stable maps.  In the case of other classical bordism
theories, we expect to obtain a model for the category of modules over
the corresponding bordism ring spectrum. In the bulk of the paper, we
handle homotopy coherence questions by implementing these ideas using
quasicategories, but we shall postpone discussing such technicalities
for now.

In order to explain the previous paragraph, recall that Cohen, Jones, and Segal~\cite{Cohen1995} defined a \emph{framed flow category} to be a (non-unital) category $\bF$ in which the morphism spaces are smooth manifold with corners, equipped with stable framings, so that the facets of each morphism space are given by the images of the composition maps
\begin{equation} 
  \bF(p,q) \times \bF(q,r) \to \bF(p,r)
\end{equation}
which are assumed to be compatible with the stable framings.  This
notion extends to any of the bordism theories that we study, since
derived orbifolds with corners are closed under products.  We refer to
these generalized flow categories as \emph{structured flow
categories}, with the understanding that we fix a structure of
interest (e.g., framings or stable complex structures), and all the
maps that we consider implicitly respect the chosen structure. Note
that irrespective of the structure considered, there is always a
distinguished flow category which we denote $\ast$, which has a single
object and no morphisms, and which we refer to as the \emph{unit}.

\begin{rem} \label{rem:properness_flow_category}
It follows formally from the definition of a framed flow category that
the condition that $ \bF(p,q)$ be nonempty defines a partial order
on the set of objects; if we assume that $\bF(p,q)$ and $\bF(q,p)$
are both nonempty, we can conclude that $\bF(p,p) $ is non-empty, but since
the product $ \bF(p,p) \times \bF(p,p)$ must be a boundary stratum of
$\bF(p,p)$, the dimension would have to be unbounded, which
contradicts the hypothesis of properness. 

In practice, our definition of a flow category will be more
complicated to account for the fact that moduli spaces in Floer theory
are compact only in very restricted contexts
(c.f. Section~\ref{sec:flow-categories}). Instead, as alluded to
earlier, the morphism spaces are equipped with a proper map to the set
of non-negative real numbers, and all operations respect this map. One
can specialize to the notion discussed in this introduction by
requiring that this morphism vanish.
\end{rem}

In order to construct a category whose objects are flow categories, we
have to describe the morphisms. The starting point is that a morphism from a flow category
$\bF_1$ to a flow category $\bF_2$ is given by a generalized functor
encoded as what we call a {\em flow bimodule}: this is specified by an
assignment to each object $p_1$ of $\bF_1$ and to each object $p_2$ of
$\bF_2$ a derived orbifold with corners $\bF_{12}(p_1, q_2)$, whose
facets are enumerated by objects of $\bF_1$ and $\bF_2$, and which are
respectively equipped with equivalences
\begin{align} \label{eq:composition_maps_Flow}
  \bF_1(p_1,q_1) \times \bF_{12}(q_1,q_2) & \to \partial^{q_1}
  \bF_{12}(p_1,q_2) \\
 \label{eq:composition_maps_Flow-2}
  \bF_{12}(p_1,p_2) \times \bF_{2}(p_2,q_2) & \to \partial^{p_2}
  \bF_{12}(p_1,q_2)
\end{align}
that satisfy natural associativity conditions. The reader with
experience in Floer theory will recognise this structure as the one
appearing when studying continuation maps; in fact, a similar
construction, in the setting of Kuranishi spaces, was considered
in~\cite{Fukaya2020}, as a formal framework for formulating the
invariance of Floer homology under auxiliary choices.

\begin{example}
A proper Morse-Smale $f$ function on a complete Riemannian manifold
$M$ defines a framed flow category $\Mbar(f)$ whose objects are the
critical points of $f$ and whose morphisms are the (compactified)
moduli spaces of gradient flow lines connecting these critical
point. The work of Cohen, Jones, and Segal~\cite{Cohen1995} implies
that the stable homotopy type of $M$ can be reconstructed from this
framed flow category. The formalism of flow bimodules can be used to
recover maps of stable homotopy types from similar data.  Indeed,
after a $C^\infty$-small perturbation, we may assume that any smooth
map from a manifold equipped with such a function $f_1$ to a manifold
equipped with a function $f_2$ has the property that the images of the
descending manifold of a critical point $p_1$ of $f_1$ are transverse
to the ascending manifolds of the critical points of each critical
point $q_2$ of $f_2$. These transverse fibre products have natural
Morse-theoretic compactifications $\Mbar(p_1,q_2)$, which are the
underlying spaces of the flow bimodule associated by Morse theory to
the given map of spaces.  
\end{example}

We now state our main result, which is proved in
Section~\ref{sec:structured-case}:

\begin{thm} \label{thm:main-thm-Flow}
Structured flow categories are the objects of a stable
$(\infty,1)$-category 
whose
morphisms are flow bimodules.
\end{thm}

Stable $(\infty,1)$-categories are one of many possible models for the
notion of spectrally enriched categories, i.e., categories whose
morphisms are spectra and whose compositions are maps of spectra.
(See e.g.,~\cite[4.23]{Blumberg2013} for a precise statement.)  Peeling away the
$\infty$-categorical language, the theorem above therefore implies
that for flow categories $\bF_1$ and $\bF_2$ there is a {\em mapping
  spectrum} $\Flow(\bF_1, \bF_2)$.  Even more concretely, the homotopy
groups of this mapping spectrum (i.e., the groups of underlying
homotopy classes of morphisms between objects), which we write
$\Flow_*(\bF_1, \bF_2)$, can be explicitly described as \emph{bordism
classes of flow bimodules}.  Here a null bordism of $\bF_{12}$
consists of an assignment $\bW(p_1, q_2)$ of a derived orbifold with
corners for each pair of objects of $\bF_1$ and $\bF_2$,
together with maps
\begin{align}
\bF_{12}(p_1, q_2) & \to  \bW(p_1, q_2) \\
  \bF_1(p_1, q_1) \times \bW(q_1, q_2) & \to \bW(p_1, q_2) \\
  \bW(p_1, p_2) \times \bF_{2}(p_2, q_2) & \to \bW(p_1, q_2)
\end{align}
that are equivalences onto boundary strata (and which enumerate
them). We require associativity as before, as well as the
compatibility of the second and third sets of maps with
Equations~\eqref{eq:composition_maps_Flow}
and~\eqref{eq:composition_maps_Flow-2}.

\begin{example} \label{ex:bordism_groups}
Specializing to the case that $\bF_1 = \ast$ is the unit, the groups $
\Flow_*(\ast, \bF)$ are the \emph{bordism groups of $\bF$}, which are 
generated by \emph{right flow modules} modulo bordism. This is the
explicit model for the Floer homotopy groups of a flow category which
we mentioned earlier. In the case of Morse theory, the bordism groups
of the flow category associated to a Morse function $f$ on a manifold
$M$ are naturally isomorphic to the classical bordism groups of
$M$. The map in one direction is given by assigning to a manifold over
$M$ its fibre product with the Morse-theoretic closure of the
ascending manifold of each critical point of $f$ (c.f.~\cite[Appendix
  D]{Abouzaid2021a}). 
\end{example}

Returning to our discussion of Theorem \ref{thm:main-thm-Flow}, its next consequence is that bordism classes of flow bimodules are equipped with natural composition maps
\begin{equation}
   \Flow_*(\bF_1, \bF_2) \times  \Flow_*(\bF_2, \bF_3) \to \Flow_*(\bF_1, \bF_3)
\end{equation}
which have units. The construction of this composition is essentially a gluing argument, which is easiest to describe when the obstruction bundles are trivial: assuming that $\bF_{12}$ and $\bF_{23}$ are such flow bimodules, we may define
\begin{equation}
  \bF_{12} \circ \bF_{23} (p, r) \equiv \coprod_{q \in \Ob(\bF_2)}
  \bF_{12}(p, q) \times  \bF_{23}(q, r) / \simeq,
\end{equation}
where the equivalence relation identifies the images of the two maps
\begin{equation}
    \bF_{12}(p, q) \times  \bF_{23}(q, r) \leftarrow  \bF_{12}(p, q)
    \times \bF_{2}(q, q') \times \bF_{23}(q', r) \to   \bF_{12}(p, q')
    \times  \bF_{23}(q', r).
\end{equation}
Because the combinatorics of the gluing defining $\bF_{12} \circ
\bF_{23}$ are relatively simple, it is not too difficult to check that
the result of this gluing procedure is a topological manifold with the
desired boundary stratification, which moreover can be equipped with a
smooth structure in a canonical way up to contractible choice, as
described in Section~\ref{sec:kan-condition-flow}.  The description of
the unit in $\Flow_*(\bF,\bF)$ is slightly more elaborate, arising
from a nontrivial manifold with corner structure on the product of
morphism spaces $\bF(p,r) \times [0,1]$; see
Section~\ref{sec:quasi-units} for the construction. 

\begin{example}
Returning to Example \ref{ex:bordism_groups}, and specializing further
to the case where both flow categories are the unit category, we find
that $\Flow_*(\ast, \ast)$, which is generated by bordism classes of
derived orbifolds, is a ring. It is easy to see from the above
description of composition that the ring structure is given by
products of derived orbifolds, so that it recovers the usual ring
structure on bordism groups.
\end{example}

At the level of bordisms, the next consequence of
Theorem~\ref{thm:main-thm-Flow} is that the homotopy category
$\Ho(\Flow)$, i.e., the category obtained by passing to bordism
classes of flow bimodules, is \emph{triangulated}.  This means in
particular that there is an invertible shift endofunctor  $\Sigma$ and
an object $\Cone(\bF_{12})$ associated to each flow bimodule, giving
rise to a long exact sequence of morphism groups 
\begin{multline}
   \cdots \to \Flow_*(\bF, \Sigma^{-1} \Cone(\bF_{12})) \to  \Flow_*(\bF, \bF_1) \to \Flow_*(\bF, \bF_2) \to  \\ \Flow_*(\bF, \Cone(\bF_{12})) \to \Flow_*(\bF, \Sigma \bF_1) \to \cdots
\end{multline}
for each flow category $\bF$.

We end this section by highlighting an additional consequence of
Theorem \ref{thm:main-thm-Flow}, which, unlike those discussed above,
cannot be extracted by passing to bordism groups: 

\begin{cor}
The endomorphism spectrum of the unit flow category is an
(associative) ring spectrum, whose homotopy groups compute the bordism
ring of a point. \qed 
\end{cor}

In the case of derived orbifold bordism, there is currently no other model for this ring spectrum. On the other hand, for bordism theories of manifolds, the
Pontryagin-Thom construction implies that Thom spectra give models for
these ring spectra, which in the case of framed cobordism agrees with
the sphere spectrum. Comparing the constructions in this case, we shall
prove the following result in Section~\ref{sec:comp-with-sphere}:

\begin{prop} \label{prop:framed_flow=spectra}
The stable $\infty$-category of framed flow categories is equivalent
to the stable $\infty$-category of spectra.
\end{prop}

In addition to the Pontryagin-Thom construction, the proof of the
above result entails showing that the unit object generates the
category of framed flow categories in the sense that any nontrivial
flow category must admit a nontrivial map from a shift of the unit.
As will be apparent from the discussion in
Section~\ref{sec:generation}, our proof of this result applies more
generally to manifold flow categories.

\subsection{Hamiltonian Floer theory}
\label{sec:hamilt-floer-theory}

We now briefly indicate how Theorem~\ref{thm:main-thm-Flow} can be
applied to study Hamiltonian Floer theory. A time-dependent function
$H \co M \times S^1 \to \bR$ on a symplectic manifold determines a
time-dependent vector field which on a closed manifold can be
integrated to a diffeomorphism. Those diffeomorphisms arising from
this construction are called \emph{Hamiltonian diffeomorphisms}, and
we say that they are nondegenerate if the associated graph is
transverse to the diagonal. A precise version of the following result
is stated in Proposition~\ref{prop:Hamiltonian_Flow_category}:

\begin{prop} \label{prop:complex_oriented_flow_category}
   Floer theory associates to each nondegenerate Hamiltonian $H$ a
   flow category whose objects are the time-$1$ periodic orbits of $H$
   and whose morphisms are complex-oriented derived orbifolds. The
   construction of this flow category depends on choices, but its
   equivalence class in the homotopy category of flow categories does not. 
 \end{prop}

The approach we take follows the insight of~\cite{Abouzaid2021b},
 where moduli spaces of genus-$0$ pseudo-holomorphic curves of a given
 degree were shown to be realized as \emph{global Kuranishi charts},
 i.e., as quotients under a compact Lie group action of zero-loci of an
 equivariant vector bundle on a smooth manifold equipped with a finite
 isotropy action (a global Kuranishi chart determines a derived
 orbifold by passing to the quotient). The key point was to realize
 the domain of all such curves as stable curves in projective space,
 so that the unitary group acts via its natural action on the space
 of domains. This was extended to Hamiltonian Floer theory in work of
 Bai-Xu and Rezchikov~\cite{BaiXu2022,Rezchikov2022}, and Appendix~\ref{sec:hamilt-floer-theory-1} is really a combination of their
 results with the discussion of complex orientations from
 \cite{Abouzaid2021a}, with minor modifications to account for the
 formalism developed in this paper. 

 Returning to Remark~\ref{rem:complex_manifold_bordism}, we may specialize to the situation in which sphere bubbles are precluded, so the moduli spaces of pseudo-holomorphic curves do not have any point with nontrivial isotropy. The most natural way to ensure this is to require that the symplectic form on $M$ be symplectically aspherical in the sense that the integral of $\omega$ vanishes on any $2$-sphere in $M$. 
 \begin{cor}
   If $M$ is symplectically aspherical, the flow category associated by Floer theory to a Hamiltonian $H$ has morphisms which are complex-oriented derived manifolds.  \qed
 \end{cor}
\begin{rem}
While we expect that the category of flow categories which are
structured with respect to complex-oriented derived manifolds may be
described in terms of modules over the periodic $MU$ spectrum, we do
not prove such a result in this paper.
\end{rem}
 
We are also interested in understanding when Hamiltonian Floer theory
gives rise to a framed flow category, i.e., to a spectrum (module over
the sphere spectrum), according to
Proposition~\ref{prop:framed_flow=spectra}.  The starting point is the
fact that a complex-linear connection on $TM$ induces a monodromy map
on the based loop space
\begin{equation}
  \Omega M \to U(n)  
\end{equation}
given by parallel transport, after chosing an identification of the
tangent space of $M$ at the basepoint with $\bC^n$. The following
result is proved in Appendix \ref{sec:fram-flow-categ}: 

\begin{prop} \label{prop:framed-ham-flow}
If the $M$ is symplectically aspherical, 
then a lift of the
monodromy map to the orthogonal group, through the complexification map $O(n) \to U(n)$,  determines  a lift of the
unstructured flow category associated to a Hamiltonian on $M$ to a
framed flow category.
\end{prop}

In general, the monodromy map does not lift through the
complexifications; for instance, this implies that the first Chern
class of $M$ vanishes on every sphere.

\begin{rem}
The proof that we shall give immediately applies to the setting where
the lift is defined after stabilizing by a trivial complex vector
space. More generally, one may assume that the lift is defined only
after passing to the limit $ O(\infty) \to U(\infty)$. 

We note as well that the origin of this statement goes back to Cohen,
Jones, and Segal~\cite{Cohen1995}, who arrived at the conclusion that
a stable isomorphism of the tangent bundle of $M$ with the
complexification of a real vector bundle yields compatible framings of
moduli spaces of solutions to Floer's equation. Such an isomorphism
can be stated as a factorization of the map $M \to BU$ through the
complexification map $BO \to BU$, which yields our  assumption after
passing to based loop spaces. Closed Riemann surfaces of positive
genus show that the assumption on monodromy is strictly weaker. 
\end{rem}
\begin{rem}
 One delicate part of the statement of Proposition \ref{prop:framed-ham-flow} is that it is the \emph{unstructured} flow category that lifts to the framed setting, not the complex oriented one from Proposition \ref{prop:complex_oriented_flow_category}. This is because our result is based on trivializing the real bundle underlying the index bundle on the complex-linear Cauchy-Riemann operator on pseudo-holomorphic spheres, rather than the corresponding virtual complex vector bundle.
\end{rem}
\subsection{What is contained in this paper?}
\label{sec:what-contained-this}

This paper is fundamentally about extracting algebraic structures from
manifolds, and more generally orbifolds and derived orbifolds, with
corners. To this end, we begin in Section~\ref{sec:preliminaries} by
introducing a categorical framework for labelling corner strata,
defining a symmetric monoidal category of derived orbifolds with such
labelled strata, and giving a precise definition of the tangential
structures of framings and complex orientations, which we package as
$2$-categories.  In Section~\ref{sec:flow-categories} we define flow
categories as certain categories enriched in derived orbifolds, and
structured flow categories as lifts of the enrichment structure to the
appropriate $2$-category.  In Section~\ref{sec:quas-struct-floe} we
show that structured flow categories and flow bimodules assemble into
a semisimplicial set, which we prove satisfies the weak Kan condition
(i.e., has fillers for inner horns) in
Section~\ref{sec:kan-condition-flow}.  In
Section~\ref{sec:quasi-units}, we combine geometric and abstract
arguments to show that the semisimplicial set of structured flow
categories admits weak units and thus can be extended to a
quasicategory.  In Section~\ref{sec:prestable-structure}, we show that
this quasicategory is stable, completing the proof of
Theorem~\ref{thm:main-thm-Flow}, while
Section~\ref{sec:comp-with-sphere} proves
Proposition~\ref{prop:framed_flow=spectra}, comparing the stable
$\infty$-category of framed flow categories to the stable
$\infty$-category of spectra. We include two appendices:
Appendix~\ref{sec:orbif-orbib} collects some elementary facts and
terminology about orbifolds, while
Appendix~\ref{sec:hamilt-floer-theory-1} contains our results on
Hamiltonian Floer theory.

\subsection{What is missing from this paper?}
\label{sec:what-missing}
While there is a long list of results that need to be developed to
flesh out the foundations of Floer homotopy theory, we mention two
that will appear in subsequent work:

\begin{enumerate}

\item The category of flow categories is equipped with the
  quasicategorical analogue of a symmetric monoidal structure, which
  makes it possible to formulate algebraic structures in this
  context. This is the subject of the next paper~\cite{Abouzaidc}. 

\item The study of Floer theory for general Lagrangian submanifolds
  naturally leads to the notion of a \emph{flow multicategory}, which
  lifts curved $A_\infty$ algebras~\cite{Fukaya2009} to the level of
  bordism. This is the subject of the work in
  progress~\cite{Abouzaide,Abouzaidf}. 

\end{enumerate}

\subsection{Why did we do it this way?}

A natural question to ask is why did we make the specific choices in
this paper to set up the foundations.  There are various other
possible models for the stable $\infty$-category $\Flow^{\cS}$, some
of which make the comparison with the category of spectra in the framed case easier to prove or even essentially a
tautology; for example, the on-going program of Lurie-Tanaka proposes
another approach to foundations~\cite{Lurie2018}.  The central
motivation that guides our approach to the foundations is to remain as
close to the geometry of the situation as possible, while still
supporting the constructions made possible by spectral algebra.
Specifically, we want the output of the analysis of the moduli spaces
of pseudoholomorphic curves  to plug directly into the
foundations in a way that allows geometric arguments to be used to
study the resulting algebraic structures. The first applications of this point of view have already appeared in the work of Porcelli and Smith~\cite{Porcelli2024}, who use a simplification of our framework to establish constraints on the topology of exact Lagrangian embeddings. As discussed earlier, we take a further step towards our goal in 
Appendix~\ref{sec:hamilt-floer-theory-1}, in the setting of Hamiltonian Floer theory.

As a final aside, we recall that there is a long history of results
relating geometric to homotopical constructions in algebraic topology,
starting with Thom's identification of Thom spaces as a model for
bordism~\cite{Thom1954}.  We would like to highlight Quinn's
construction of bordism spectra~\cite{Quinn1995} and Laures and
McClure's~\cite{Laures2014} subsequent comparison of the multiplicative
structure on these spectra to those obtained from homotopy theory.
While this paper uses only the most basic version of the
Pontryagin-Thom construction, and thus does not rely on this
literature, we expect that future work will fruitfully engage with it.

\subsection*{Acknowledgments}

MA would like to start by thanking Madhav Nori for starting the Fall
2002 first year graduate topology course at the University of Chicago
with a geometric construction of oriented cobordism groups of smooth
manifolds.  Both authors would like to thank Shaoyun Bai, Mike Hill,
Mike Hopkins, Thomas Kragh, Tyler Lawson, Mike Mandell, Ciprian
Manolescu, and Semon Rezchikov for helpful conversations over the
course of this work, as well as comments on a draft.  They are
grateful to Lucy Yang and Wolfgang Steimle, as well as to Amanda
Hirschi and the reading group which she led in April 2024 for identifying multiple issues in the
first publicaly released version of the paper which needed to be
clarified or corrected.  We are particularly grateful to Colin Fourel's comments, which led us to correct the proof of the horn-filling axiom. MA would like to end by thanking Mark McLean
and Ivan Smith for the ongoing collaboration to study the global chart
approach to Floer homotopy theory; discussions with them, as well as
their feedback on an early draft of
Appendix~\ref{sec:hamilt-floer-theory-1}, were essential in ensuring
that the abstract framework described in this paper is adapted to
geometric applications.

MA was partially supported by NSF award DMS-2103805.  AJB was
partially supported by NSF awards DMS-2104420 and DMS-2052970.

\section{Derived effective orbifolds with corners}
\label{sec:preliminaries}

As discussed in the introduction, the constructions of this paper are
all formulated in elementary terms using effective orbifolds, which we
review in Appendix~\ref{sec:classical-orbifolds}.

\subsection{Categories with corners}
\label{sec:categ-with-corn}

We shall be performing many constructions with orbifolds with corners, so it is useful to begin by introducing some notation.

A manifold with corners $X$ determines a category $\cP_X $ with objects given by components of the corner strata, and morphisms given by (local) adjacency: explicitly,  arrows from $\partial^\sigma X$ to $\partial^\tau X$ correspond bijectively to components of the intersections of (the interior of) $ \partial^\sigma X$ with any sufficiently small tubular neighbourhood of $ \partial^\tau X$. Our conventions are such that the interior of $X$ is the initial element, and that we have a functor:
\begin{equation}
\codim \co   \cP_X  \to \bN
\end{equation}
which records the codimension of each stratum. This discussion applies without modification to orbifolds with corners, since we impose the condition, in Definition \ref{def:orbifold-chart-with-corners}, that the group acts trivially on the normal directions to each stratum.

The labelling of strata by components will be too fine for our
purposes, so we introduce an abstract notion that will lead to coarser
labellings. Consider a category $\cP$ which equipped with a functor
  \begin{equation}
    \codim \co \cP \to \bN.
  \end{equation}
 
  Denote by
  $\cP^{p}$ the over category of $p \in \ob(\cP)$, i.e., the category
  whose objects are arrows in $\cP$ with target $p$. The following
  definition reflects the local combinatorial structure of manifolds
  with corners:

\begin{defin}\label{def:model-manifolds-with-corners} 
The category $\cP$ is a \emph{model for manifolds with corners} if for each object $p$,  the overcategory $\cP^p$ is isomorphic to the power set of $\{1, \ldots, \codim p\}  $ (the set of subsets, partially ordered by inclusion).
\end{defin}
Concretely, this means that we can assign to each finite subset $K$ of $\{1, \ldots, \codim p\}$ an arrow $\alpha^K$ in $\cP$ with target $p$, and for each inclusion $J \subset K$ there is a unique arrow $\alpha_{J}^{K}$ in $\cP$ with $\alpha_{J}^{K} \circ \alpha^{J} = \alpha^K$. It will be useful to note that a choice of identification between the minimal elements of $\cP^{p} \setminus \{\min \cP^p\}$ and $\{1, \ldots, \codim p\}  $ determines the isomorphism in the definition, so we introduce the notation
\begin{equation} \label{eq:qp-normal_directions}
  Q_{\cP}(p) \equiv \min \left( \cP^{p} \setminus \{\min \cP^p\} \right).
\end{equation}
We have the following basic result where $\partial^{p} \cP$ denotes the undercategory of $p$ (i.e., the category whose objects are arrows with source $p$): 
\begin{lem} \label{lem:decomp_next_to_minimal_elts}
  Every arrow $\alpha \co p \to q$ functorially induces a decomposition
  \begin{equation} \label{eq:decomposition_normal_directions}
       Q_{\cP}(q)\cong Q_{\cP}(p) \amalg Q_{\partial^{p} \cP}(\alpha) .
  \end{equation}
 \qed
\end{lem}

\begin{example}
The face poset of the simplex $\Face^{op} \Delta^{n}$, equipped with the functor that records the codimension of every face is the standard example of a category with corners. Specializing to this case will recover the theory of $\langle n \rangle$-ads of Quinn~\cite{Quinn1995}.
\end{example}

\begin{rem}
Every construction that we will consider in this paper can be phrased in terms of 
categories which are in fact partially ordered sets (i.e., so that
there is at most one arrow between any pair of objects), as discussed in Remark \ref{rem:imposing_poset} below, but follow-up work will require the more general notion discussed above. 
\end{rem}

We define a morphism $f$ of models to be an embedding $\cP_0 \to \cP_1$, preserving  codimension up to an overall shift, and inducing an isomorphism on the categories of factorisations of each arrow, i.e., so that if $\alpha$ is an arrow in $\cP_0$, the functor induces a bijection between factorisations of $\alpha$ in $\cP_0$ and of its image in $\cP_1$.    This class of morphisms is closed under compositions. Given an object $p$ of $\cP_0$, denote by $Q_f(p)$ the set $Q_{\cP_1}(f(p_0))$ where $p_0$ is the unique minimal element of $\cP_0$ which admits an arrow to $p$. As a consequence of Lemma \ref{lem:decomp_next_to_minimal_elts}, we have
\begin{lem}
  If $f$ and $g$ are maps of models, we have a natural decomposition
  \begin{equation}
    Q_{g \circ f}(p) \cong Q_f(p) \amalg Q_g(f(p)).
  \end{equation} \qed
\end{lem}

Returning to the geometry of orbifolds, we have the following:
\begin{defin}
  A  \emph{stratification of an orbifold with corners} $X$ is a functor $\cP_X \to \cP$ with target a model for manifolds with corners. 
  We require that this functor preserve codimension, and that the induced functor on overcategories be an isomorphism for each corner stratum.
\end{defin}
What the second condition means is that, for each stratum $\tau$ of $X$, the functor $\cP_X \to \cP$ induces an isomorphism between strata $\sigma$ whose closure contains $\tau$ and the set of arrows whose target is the image of $\tau$ in $\cP$.
\begin{rem}
In the case that $\cP$ is itself a partially ordered set, the
labelling of the corners by $\cP$ can be formulated simply as an order preserving map
from the set of corners to $\cP$, which when restricted to those
adjacent to a corner labelled by $p$ induces an isomorphism with the subset of elements lying over $p$.
\end{rem}

For $p \in \cP$, recall that $\partial^p \cP$ denotes the undercategory of $p$.

\begin{defin}
Given an object $p$ of $ \cP$, the {\em corner stratum} $\partial^p X$ of a
$\cP$-stratified orbifold $X$, is the 
$\partial^p \cP$-stratified orbifold consisting of pairs $(\alpha, x)$ where $\alpha$ is an object of $\partial^p \cP$ and $x$ is an element of $X$ labelled by the image of $\alpha$.
\end{defin}

Note that $\partial^p X$ is a suborbifold of $X$ exactly when objects of $\cP$ admit at most one morphism to $p$. Nonetheless, we abuse terminology below in writing that we \emph{restrict} various data, for example vector bundles, from $X$ to $\partial^p X$.

For the next definition, it is convenient to note that, if $   (X, \cP_X \to \cP) $ is a stratified orbifold, then we have a decomposion
\begin{equation}
  X \equiv \coprod_{p} \partial^p X  
\end{equation}
indexed by the minimal elements $p$ of $\cP$.

\begin{defin}
   A morphism of stratified (effective) orbifolds 
  \begin{equation}
   (X, \cP_X \to \cP) \to  (X', \cP_{X'} \to \cP') 
 \end{equation}
 consists of a functor $ \cP \to \cP'$, and a map of orbifolds
 \begin{equation} \label{eq:morphism_to_boundary_stratum}
\partial^p X \to \partial^{f(p)} X
 \end{equation}
 for each minimal element $p$ of $\cP$, inducing a commutative diagram
 \begin{equation}
   \begin{tikzcd}
     \cP_X  \ar[r] \ar[d] & \cP_{X'}\ar[d] \\
     \cP \ar[r] &  \cP.
   \end{tikzcd}
 \end{equation}
\end{defin}

\subsection{Derived orbifolds}
\label{sec:derived_orbifold}
We need a preliminary definition before introducing the notion of a derived orbifold:
\begin{defin}
A virtual vector bundle on an orbifold $X$ is a pair $(V^+, V^-)$ of vector
bundles on $X$.  An equivalence $(V^+, V^-) \cong (W^+, W^-)$ of
virtual vector bundles is an isomorphism of vector bundles
\begin{equation}
V^+ \oplus W^{-} \cong V^{-} \oplus W^{+}.
\end{equation}
\end{defin}

We do not assume that the vector bundles that we consider are smooth, though it is a standard fact that such a unique such structure (up to isomorphism) exists. Subsequently, the following definition does not refer to any smoothness on the datum of the vector bundle:
\begin{defin} \label{def:derived_orbifold}
  A \emph{derived orbifold with corners} $\bX$ consists of a triple $(X, T^- \bX, s)$ where $X$ is an (effective and smooth) orbifold with corners, $T^- \bX$ is a vector bundle on $X$, and $s$ is a  section of $T^- \bX$. We say that $\bX$ is proper if the zero-locus $s^{-1}(0)$ is compact.
\end{defin}
We write $T^+ \bX$ for the tangent bundle of $X$, and $T \bX$ for the pair
\begin{equation}
  T \bX \equiv (T^+ \bX, T^- \bX),
\end{equation}
which we call the \emph{tangent bundle of $\bX$}, considered as a virtual vector bundle on $X$. More generally, we abuse notation by refering to data defined on $X$ as data defined on $\bX$, e.g., by a vector bundle on $\bX$, we mean a vector bundle on $X$.
\begin{rem}
  Definition \ref{def:derived_orbifold} is a compromise between the concrete output of pseudo-holomorphic curve theory, and the desire to have a relatively simple structure with which to perform formal constructions. In the setting of \cite{Abouzaid2021b} which we use for applications, the actual output is a topological orbifold, equipped with a topological submersion over a smooth manifold, a fibrewise smooth structure, and a vector bundle with a section which is smooth in the directions of the fibre. Moreover, the total space has a natural stratification (finer than the stratification by corner strata because of the existence of even codimension strata in the interior), so that each stratum is smooth,  and the restriction of the section is smooth with respect to this structure. We choose to throw away the information contained in the smooth structures of strata, keeping only the smooth structure on the base and the fibre, from which we extract a (stable) smooth structure on the total space by appealing to $G$-smoothing theory. We could further choose to smooth the section, which would allow us to simplify the definition of equivalence of derived orbifolds in Definition \eqref{def:strong_equivalence}, but that would further complicate the arguments in Appendix \ref{sec:hamilt-floer-theory-1}.
\end{rem}
For the next definition, the reader should recall that vector bundles
on effective orbifolds have functorially defined pullbacks with
respect to good morphisms (c.f.~Proposition
\ref{prop:pullback_good_maps}).
\begin{defin}
  A morphism $f \co \bX \to \bX'$ of derived orbifolds consists of a
  good map $f\co X \to X'$ of orbifolds, and a map $\tilde{f} \colon
  T^- \bX \to f^* T^- \bX'$ of vector bundles so that the following
  diagram commutes:
\begin{equation} \label{eq:section-compatibility}
\begin{tikzcd}
T^- \bX \ar[r,"\tilde{f}"] & f^* T^- \bX'. \\
X \ar[u,"s"] \ar[ur,"f^*s'"'] \\
\end{tikzcd}
\end{equation}
\end{defin}

Derived orbifolds form a category under this notion of morphism; given
maps $f \colon \bX \to \bX'$ and $g \colon \bX' \to \bX''$, the
composite map on vector bundles is defined as
\begin{equation}
\begin{tikzcd}
T^- \bX \ar[r,"\tilde{f}"] & f^* T^- \bX' \ar[r, "f^* \tilde{g}"] & f^* (g^*
T^- \bX'') \ar[r,"\cong"] & (g \circ f)^* T^- \bX'',
\end{tikzcd}
\end{equation}
where the last map is the canonical isomorphism.  Coherence for the
canonical comparisons $f^* g^* \cong (g \circ f)^*$ provides the
required associativity property. 

We want to distinguish a special class of morphism (equipped with
additional data) which are associated to total spaces of vector
bundles, which we denote $N$ below because we have in mind normal bundles:

\begin{defin}
Let $\bX$ be a derived orbifold and $p \colon N \to X$ a vector bundle
on its underlying orbifold $X$.  We define the derived orbifold $\bN$
to have underlying orbifold the total space $N$, with the vector
bundle $T^- \bN$ on $N$ given by the pullback $p^* (T^- \bX) \oplus N$
and section defined by $p^* s \oplus \id_N$.
\end{defin}

Note that the identity map of $\bX$ factors through $\bN$: 
\begin{equation}
\begin{tikzcd}
\bX \ar[rr,bend right=30, swap, "\id"] \ar[r,"\iota"] & \bN \ar[r,"p"] & \bX,  
\end{tikzcd}
\end{equation}
where $\iota$ denotes the inclusion of the zero section.

\begin{defin} \label{def:strong_equivalence}
A \emph{strong equivalence} between derived orbifolds $\bX$ and $\bX'$ is a map $f \co \bX \to \bX'$ of derived orbifolds, together with a vector bundle structure on $X'$ over $X$ identifying this map with the map induced by the inclusion of $0$-section.
\end{defin}

Since the total space of a vector bundle on $\bX$ is again a vector
bundle on $\bX$, whose $0$-section is given by composing $0$-sections,
we can see that strong equivalences are closed under composition.

\begin{rem}
The above limited definition is a replacement for the general definition of an equivalence of derived orbifolds, which is easiest to state when the sections $s$ and $s'$ are smooth, namely that the following sequence of vector bundles is exact along $s^{-1}(0)$
    \begin{equation}
      T^+ \bX \to T^- \bX \oplus T^+ \bX' \to   T^- \bX',
    \end{equation}
    and that the map of $0$-loci be an isomorphism of orbispaces.
\end{rem}

We now incorporate the discussion of Section \ref{sec:categ-with-corn}.
\begin{defin}
  A \emph{$\cP_{\bX}$-statified derived orbifold} is a derived orbifold $\bX$ whose underlying derived orbifold with corners is equipped with a stratification $\cP_X \to \cP_{\bX}$. The {\em corner stratum} $\partial^p \bX$  associated to an element  $p$ of $ \cP_{\bX}$  consists of the orbifold with corners  $\partial^p X$, stratified by $\partial^p \cP_{\bX} $, and equipped with the pullback of $ T^{-} \bX$, and the pullback of the section $s$.
\end{defin}

We can now define the category of stratified derived orbifolds we
will use.

\begin{defin}
  The category $d\Orb$ of (stratified effective) derived orbifolds has objects pairs $(\cP_{\bX}, \bX)$ where
      $\bX$ is a $\cP_{\bX}$-stratified derived orbifold, and morphisms
  \begin{equation}
   f \co (\bX, \cP_X \to \cP_{\bX}) \to  (\bX', \cP_{X'} \to \cP_{\bX'}) 
 \end{equation}
 given by a morphism of stratified orbifolds, and a lift (for each minimal element of $\cP_{\bX})$ of the map $ \partial^p X \to \partial^{f(p)} X$ to a strong equivalence
 \begin{equation} \label{eq:strong_equivalence_to_boundary_stratum}
\partial^p \bX \to \partial^{f(p)} \bX.
 \end{equation}
\end{defin}

\begin{rem}
Note that in contrast to the usual framework for derived orbifolds
(e.g., via stacks), $d\Orb$ is an ordinary category and not a
$2$-category.  This suffices for our purposes because we only need the
very restricted class of morphisms arising as boundary inclusions
along strong equivalences.  However, $d\Orb$ does not have colimits;
in our proof that the quasicategory of flow categories that we construct satisfies the horn-filling condition, we
will appeal to Pardon's work on vector bundles in order to glue, which implicitly uses the full $2$-category of stacks.
\end{rem}

\begin{defin}
We define the product of charts $\bX$ and $\bY$ to be given by: 
\begin{enumerate}
\item the product of models $\cP_{\bX} \times \cP_{\bY}$,
\item the product of orbispaces $\bX \times \bY$,
\item the sum of vector bundles $T^{-} \bX \oplus T^{-} \bY$, 
\item and the sum of sections $s_0 \oplus s_1$.
\end{enumerate}
\end{defin}

It is immediate from the monoidal structures on the constituents that
this specifies a monoidal structure on $d\Orb$:

\begin{prop}
The product of charts endows $d\Orb$ with a symmetric monoidal
structure.  The unit is specified by the derived orbifold with
underlying orbifold a point, stratified by a point, equipped with the rank $0$ obstruction
bundle.\qed
\end{prop}

We will work with categories enriched in $d\Orb$, lifting specified
stratifications.  For clarity, we spell out the associated
stratification information implied by such an enrichment.

\begin{lem}\label{lem:stratifying_2category}
Let $\cC$ be a category enriched in $d\Orb$.  Associated to the
derived orbifold $\cC(x,y)$ of morphisms from an object $x$ to an
object $y$ is the category $\cP_{\cC(x,y)}$, which we will denote
$\cP_{x,y}$.  The composition
\begin{equation}
\cC(x,y) \times \cC(y,z) \to \cC(x,z)
\end{equation}
specified by the enrichment implies that we have functors
\begin{equation}
\cP_{x,y} \times \cP_{y,z} \to \cP_{x,z}
\end{equation}
and therefore we obtain a category $\cP_{\cC}$ enriched in categories
with objects $\ob(\cC)$ and morphism categories $\cP_{\cC}(x,y) =
\cP_{x,y}$.
\end{lem}

\subsection{Structured derived orbifolds}
\label{sec:tang-spher-struct}

As noted in the introduction, the bordism theory of orbifolds is trivial unless we fix some structure either at the level of  stabiliser groups or tangent spaces. For the former, we use the following notion:
\begin{defin} \label{def:family_of_groups}
  A \emph{multiplicative family of groups} is a collection of isomorphism classes of finite groups which is closed under direct sums. 
\end{defin}
For applications to Floer theory, the most important classes of examples are the trivial family (consisting only of the group with one element), and the family of all finite groups. We shall assume that such a family of groups is fixed, and not mention it unless required for specificity.

A concrete way to impose structure on tangent spaces is to have a chosen isomorphism
between the tangent space and a prescribed virtual vector bundle. We
recall that $T^+ \bX$ is our chosen notation for the tangent bundle of
$X$, so that $T \bX$ stands for the formal difference of this tangent
space with the vector bundle $T^- \bX$, i.e., $(T^+ \bX, T^- \bX)$.

\begin{defin} \label{def:relative_framing}
Let $\bX$ be a derived orbifold.  A {\em framing of $\bX$ relative
  $E$}, for a virtual vector bundle $E = (E^+, E^-)$ on $\bX$,
consists of an equivalence $T\bX \cong E $ of virtual vector bundles
along $s^{-1}(0)$.
\end{defin}

We remind the reader that such an equivalence consists of an
isomorphism, over the $0$-locus, of vector bundles
\begin{equation}
T^+ \bX \oplus E^- \cong T^- \bX \oplus E^+.
\end{equation} 

The main disadvantage of the notion of framing in
Definition~\ref{def:relative_framing} is that it is not automatically
inherited by strata.  Using the notation $Q_{\cP}(p)$ for the normal direction to a stratum, as in Equation \eqref{eq:qp-normal_directions}, we have:
\begin{defin}
  A \emph{consistent normal framing} of the corner strata of a derived
  orbifold $\bX$ is a choice for each $p \in \cP_{\bX}$ of an
  isomorphism 
  \begin{equation} \label{eq:splitting-normal}
         T^+ \partial^p \bX \oplus \bR^{Q_{\cP}(p)} \cong  T^+  \bX|_{\partial^p X}
       \end{equation}
       extending the inclusion of the tangent space of $\partial^p X$, with the property that the vectors corresponding to each element of $ Q_{\cP}(p)$ point inwards. For each arrow from $p$ to $q$ in $\cP$, we require that the following diagram commute
       \begin{equation} \label{eq:consistency_normal_vector_fields}
         \begin{tikzcd}
           T \partial^q \bX \oplus \bR^{Q_{\cP}(p)} \ar[r] \ar[d] &   T \partial^p \bX|_{\partial^q X} \oplus \bR^{Q_{\cP}(p)} \ar[d] \\
          T \partial^q \bX \oplus \bR^{Q_{\cP}(q)} \ar[r] &   T \bX|_{\partial^q X}.
         \end{tikzcd}
       \end{equation}
\end{defin}
The commutativity of Diagram \eqref{eq:consistency_normal_vector_fields} should be interpreted as the standard condition that the normal vector field to $ \partial^q \bX$ associated to a stratum which also contains $\partial^p \bX$ is obtained by restriction from the larger stratum. Proceeding inductively, this implies that the normal vector fields to the codimension $1$ strata determine the normal framing of all strata, from which we immediately conclude:
\begin{lem}
  The space of consistent normal framings is contractible. \qed
\end{lem}

We shall henceforth assume that such a choice is fixed for each derived orbifold.
\begin{rem}
  An alternative approach is to modify all constructions to refer only to the short exact sequence
 \begin{equation}
T \partial^p \bX \to   T \bX|_{\partial^p \bX} \to  \bR^{Q_{\cP}(p)}
\end{equation}
which is canonical up to a choice of positive real dilation of each factor in $\bR^{Q_{\cP}(p)}$. 
\end{rem}

The choice of splitting in Equation \eqref{eq:splitting-normal} leads to:
\begin{lem} \label{lem:inherit_framings}
Let $\bX$ be a derived orbifold and $E$ a virtual vector bundle on
$\bX$.  A framing of $\bX$ relative $E$ induces a framing of
$\partial^p \bX$ relative $E \ominus \bR^{Q_\cP(p)}$ for each $p \in
\cP_{\bX}$.  This assignment is functorial in the sense that for each
arrow $\alpha \co p \to q$, the isomorphism in Equation
\eqref{eq:decomposition_normal_directions} intertwines the two
framings of $\partial^q \bX$ relative $E \ominus \bR^{Q_\cP(q)}$
obtained by (i) restricting from $\bX$, and (ii) restricting from the
induced framing on $\partial^p \bX $. \qed
\end{lem}
In some situations, we shall require more control over the induced framing at the boundary: if $E^+$ has a distinguished trivial sub-bundle which is identified with $\bR^{Q_\cP(p)}$, we may ask that the framing of $\bX$ relative $E$ respect this identification at the boundary. This is generally not possible, but can be achieved by a stabilisation:
\begin{lem} 
 Given a framing of $\bX$ relative $E$, a stratum  $p \in \cP_{\bX}$, and a sub-bundle of $E^+$ which is identified with $\bR^{Q_\cP(p)}$, there exists a vector bundle $W$ on $X$, and a framing of $\bX$ relative $E \oplus (W,W)$, which agrees with the direct sum of the given framing with the identity on $W$ away from a neighbourhood of the image of $\partial^p \bX$, and whose restriction to this stratum identifies the normal directions with the given sub-bundle of $E^+$.
\end{lem}
\begin{proof}
  This is an immediate consequence of the high connectivity of the inclusion of the orthogonal groups $O(N) \to O(N+1)$ as $N \to \infty$.
\end{proof}

\subsection{A bicategory of relatively framed derived orbifolds}
\label{sec:bicat-relat-fram}

The structures that we shall introduce will require stably comparing
the tangent spaces of derived orbifolds to structured vector bundles,
relative fixed virtual vector spaces. In Morse theory, we can restrict attention to ordinary vector spaces, which correspond to the negative-definite subspace for the Hessian on the tangent space at a critical point (one may switch to the positive-definite subspace depending on conventions).

In order to adequately formulate
the functoriality and multiplicativity of these constructions, it is
most convenient to introduce bicategories which are labelled by the structure groups of the vector bundles. We start with the case of the orthogonal group $O$.

\begin{defin}
The bicategory $d\Orb^{O}$ has
\begin{enumerate}
\item $0$-cells given by virtual vector spaces $V$, i.e., pairs
  $V = (V^+,V^-)$ of real vector spaces.

\item $1$-cells between virtual vector spaces $V_0$ and $V_1$
  (i.e., objects of $d\Orb^O(V_0,V_1) $) given by an object $\bX$  of $d\Orb$
 together a vector bundle $E_{\bX}$, and an equivalence of virtual bundles
\begin{equation} \label{eq:relative_isomorphism_TX-vbunlde}
  T \bX \oplus V_1 \cong E_{\bX} \oplus V_0.
\end{equation}
We will abuse notation and write $\bX$ for a $1$-cell.

\item $2$-cells, i.e., morphisms in $d\Orb^O(V_0,V_1)$, consisting of a morphism $f \co \bX \to \bX'$ in $d\Orb$, a decomposition
  \begin{equation} \label{eq:normal_decomposition}
  Q_f = Q_f^+ \amalg Q_f^-      
  \end{equation}
   of the finite set associated by the map of
stratifying categories $\cP_{\bX} \to \cP_{\bX'}$ to each minimal
element of $\cP_{\bX}$, split surjections of vector bundles
  \begin{align} 
 \label{eq:add_normal_directions_to_E+_sur}
\rho_f^{+} \colon E^+_{\bX'}   & \rightarrow E^+_{\bX} \oplus \bR^{Q_f^+}  \\
    E^-_{\bX}  & \leftarrow  E^-_{\bX'} \oplus \bR^{Q_f^-} \noloc \rho^{-}_f \label{eq:stabilise_E-_sur}
  \end{align}
  and a chosen isomorphism $\gamma \colon \ker \rho_f^+ \to \ker
  \rho_f^{-}$ between the kernels of these maps.  In mild abuse of
  notation, we write $E_f$ for the vector bundle over $\bX$ given by
  this kernel (and $E_{f^-}$ and $E_{f^+}$ when necessary to
  disambiguate), so that we obtain isomorphisms:
   \begin{align}  \label{eq:add_normal_directions_to_E+} 
     E^+_{\bX} \oplus \bR^{Q_f^+} \oplus E_{f^+}  & \cong E^+_{\bX'} \\
     \label{eq:stabilise_E-}
    E^-_{\bX'} \oplus \bR^{Q_f^-} & \cong  E^-_{\bX} \oplus E_{f^-}.
  \end{align} 
  
 We require the commutativity of the following diagram
  \begin{equation} \label{eq:commutative_diagram_morphism_Kuranishi-stable}
    \begin{tikzcd}
     T^+ \bX  \oplus V^+_1 \oplus E_{\bX}^- \oplus V^-_0 
        \oplus  E_{f^-} \oplus \bR^{Q^+_f}
        \ar[r] \ar[d] & T^+ \bX'  \oplus V^+_1 \oplus E_{\bX'}^- \oplus V^-_0 
      \ar[d] \\
    T^- \bX  \oplus V^-_1 \oplus E_{\bX}^+ \oplus V^+_0  
        \oplus  E_{f^+} \oplus \bR^{Q^+_f} \ar[r] &  T^- \bX'  \oplus V^-_1 \oplus E_{\bX'}^+ \oplus V^+_0,
    \end{tikzcd}
  \end{equation}
  where the top horizontal map is induced by Equation \eqref{eq:stabilise_E-} and the normal framing of $\bX$ in $\bX'$, and the bottom horizontal map by Equation \eqref{eq:add_normal_directions_to_E+}.
\end{enumerate}
\end{defin}

\begin{rem}
The appearance of the decomposition in Equation
\eqref{eq:normal_decomposition} is due to the fact that codimension
$1$ strata of moduli spaces in Morse theory arise from two different
types of phenomena:
  \begin{enumerate}
    \item breaking of Morse trajectories at an intermediate critical
      point, and
    \item boundaries of families of continuation equations.
      \end{enumerate}

  In the first case, the moduli space of solutions is the free
  $\bR$-quotient of a manifold whose tangent space is naturally equivalent to the virtual bundle formed by the negative definite subspaces of the Hessian at the critical points associated to the ends. This implies that we have an isomorphism
  \begin{equation}
    T \cM(p,r) \oplus V_{r} \oplus  \bR^{\{r\}} \cong V_{p},
  \end{equation}
  where we choose to label the $\bR$-translation by the input critical point. This leads us to set
  \begin{equation}
    E_{\cM(p,r)} \equiv ( 0, \bR^{\{p\}}).
  \end{equation}
  Along a boundary stratum given by a product $ \cM(p,q)
  \times\cM(q,r)$, we can identify the $\bR$-translation associated to
  the critical point $q$ with the normal direction, which corresponds
  to the case where $Q^+_f = \emptyset $ in Equation
  \eqref{eq:normal_decomposition}.

  In the second case, we have a manifold $B$ with boundary parametrising a family of continuation maps between Morse functions $F$ and $F'$. Writing $\cM_{B}(p,r')$ for the moduli space of solutions to such a family of equations, we have a natural isomorphism
  \begin{equation}
    T \cM_{B}(p,r') \oplus V_{r'}  \cong TB \oplus V_{p},
  \end{equation}
  which leads one to set
  \begin{equation}
   E_{\cM_B(p,r')} \equiv ( TB, 0).      
  \end{equation}
  It then follows from the formalism that, along the boundary of $B$,
  we are in the case $Q^-_f = \emptyset$ in Equation
  \eqref{eq:normal_decomposition}. 

The general case is a mixture of these two, as the boundary of the
parametrized moduli space $\cM_{B}(p,r')$ has some codimension $1$ boundary strata
given by breaking of Morse trajectories at one of the two ends (which
corresponds to the first case).  There are thus some strata with $ Q^-_f \neq \emptyset$ and others with $Q^+_f \neq \emptyset$; the general stratum is associated to a nontrivial decomposition of $Q_f$.
\end{rem}

\begin{rem}
We note that every derived orbifold is tautologically framed by $E = (T
\bX \oplus V_0, V_1)$, so that no interesting constraint is
imposed by assuming a lift of a derived orbifold to $d\Orb^O(V_0,V_1)$.
In Section~\ref{sec:stably-framed-stably}, we shall see such
nontrivial constraints appear when we equip $E$ with additional
structures.
\end{rem}
  
The composition of morphisms $f\colon \bX \to \bX'$ and $g\colon \bX' \to
\bX''$ in $d\Orb^O(V_0,V_1)$ is straightforward to define: we use the
decomposition $Q_{g \circ f} \cong Q_g \amalg Q_f$ 
to set 
\begin{align}
Q^+_{g \circ f} & \equiv Q^+_g \amalg Q^+_f \\
Q^-_{g \circ f} & \equiv Q^-_g \amalg Q^-_f,
\end{align}
and identify the kernels $E^+_{g \circ f}$ and $E^-_{g \circ f}$ of
the resulting composite maps 
\begin{align}
E^+_{\bX''} & \rightarrow E^+_{\bX'} \oplus \bR^{Q_g^+} \rightarrow
E^+_{\bX} \oplus \bR^{Q_f^+} \oplus \bR^{Q_g^+} \cong E^+_{\bX}
\oplus \bR^{Q_{g \circ f}^+}
\\
E^-_{\bX}  & \leftarrow  E^-_{\bX'} \oplus \bR^{Q_f^-} \leftarrow E^-_{\bX''}
\oplus \bR^{Q^-_g} \oplus \bR^{Q^-_f} \cong E^-_{\bX''}
\oplus \bR^{Q_{g \circ f}^-}
\end{align}
with each other using their canonical identifications as $E_{g}^{+}
\oplus (\rho^+_g)^* E_f^{+}$ and $E_{g}^{-}
\oplus (\rho^-_g)^* E_f^{-}$ and the given isomorphisms.

With these data, the commutativity of
Diagram~\eqref{eq:commutative_diagram_morphism_Kuranishi-stable} is
easy to check.  To see that this composition law is associative, the
checks are immediate except for the comparison of the isomorphisms of
kernels.  To see this, we use the coherence of pullback and the direct
sum.  To be more precise, given morphisms $f \colon \bX \to \bX'$, $g
\colon \bX' \to \bX''$, and $h \colon \bX'' \to \bX'''$, we are
comparing the isomorphisms induced by the decomposition
\begin{equation}
(E_h^- \oplus (\rho_h^-)^* E^-_g) \oplus (\rho^-_{h \circ g})^* E_f^-
\end{equation}
and the decomposition
\begin{equation}
E_h^- \oplus (\rho^-_h)^* (E^-g \oplus (\rho^-_g)^* E_f^-)
\end{equation}
(and analogously for $E_*^+$.)
By the associativity coherence for direct sums, these isomorphisms are
naturally isomorphic by a unique isomorphism, and therefore induce the
same isomorphisms between the kernels $\ker(\rho^+_h \circ \rho^+_g \circ
\rho^+_f)$ and $\ker(\rho^-_h \circ \rho^-_g \circ \rho^-_f)$.

Next, we construct the horizontal composition functor
\begin{equation} \label{eq:composition_functor_dOrbO}
      d\Orb^O(V_0,V_1) \times d\Orb^O(V_1,V_2) \to d\Orb^O(V_0,V_2).    
    \end{equation}
    At the level of derived orbifolds, this is given by the product, taking $\bX_{01}$ and $\bX_{12}$ to $\bX_{01} \times \bX_{12}$. We assign to virtual vector bundles $E_{01}$ and $E_{12}$ the virtual vector bundle
    \begin{align} \label{eq:composition_virtual_bundles}
        E_{02} \equiv E_{01} \oplus E_{12} \oplus (V^-_1,V_1^-) \cong \left( E_{01}^+ \oplus E_{12}^+ \oplus V_1^-,  E_{01}^- \oplus E_{12}^- \oplus V^-_1 \right)
          \end{align}
          and thus obtain the desired equivalence  of virtual bundles
          \begin{equation}\label{eq:composition_equivalence}
                  T \left( \bX_{01} \times \bX_{12}\right)  \oplus  V_2  \cong E_{02} \oplus  V_0
          \end{equation}
          as a composition of the isomorphism of Equation \eqref{eq:relative_isomorphism_TX-vbunlde} for the two factors.
            \begin{rem}
             There is an alternative way to define horizontal composition, which replaces Equation \eqref{eq:composition_virtual_bundles} by
             \begin{equation}
                E_{01} \oplus E_{12} \oplus V_1 \ominus V_1,
              \end{equation}
              and thus does not privilege one of the two formal summands of $V_1$. We prefer the above approach because, in the applications to Floer homotopy, the Morse-theoretic setting is privileged, and the formulae we obtain in that setting are simpler with the choice we made.
           \end{rem}
At the level of morphisms, we assign to a pair of morphisms $f_{01}$ from $\bX_{01}$ to $\bX'_{01}$ and $f_{12}$  from $\bX_{12}$ to $\bX'_{12}$,  the product map of derived orbifolds, the decomposition
\begin{align}\label{eq:bicategory_composition_q}
  Q^+_{f_{01} \times f_{12}} & \equiv Q^+_{f_{01}}  \amalg Q^+_{f_{12}} \\ \notag
  Q^-_{f_{01} \times f_{12}} & \equiv Q^-_{f_{01}}  \amalg Q^-_{f_{12}}
\end{align} using the natural isomorphism $Q_{f_{01}} \amalg Q_{f_{12}} \cong Q_{f_{01} \times f_{12}}$, and the direct sum of split-surjections of vector bundles. 
It is easy to check that this data makes
Diagram~\eqref{eq:commutative_diagram_morphism_Kuranishi-stable}
commute for the product.

Finally, we need to consider the associativity of this composition to
describe the bicategory structure.  Consider the diagram of
composition functors
\begin{equation}
\begin{tikzcd}[column sep=tiny]
d\Orb^O(V_0, V_1) \times d\Orb^O(V_1, V_2) \times d\Orb^O(V_2, V_3)
\ar[r] \ar[d] & d\Orb^O(V_0, V_2) \times d\Orb^O(V_2, V_3) \ar[d] \\
d\Orb^O(V_0, V_1) \times d\Orb^O(V_1, V_3) \ar[r] & d\Orb^O(V_0, V_3).\\
\end{tikzcd}
\end{equation}

This diagram does not commute strictly; on the level of objects, we
have an associator coming from the coherence isomorphisms for the
products:
\begin{equation}
(\bX_{01} \times \bX_{12}) \times \bX_{23} \cong \bX_{01} \times
  (\bX_{12} \times \bX_{23}),
\end{equation}
the coherence isomorphisms for the direct sum of 
Equation~\eqref{eq:composition_virtual_bundles}, and the associated
coherences for the required isomorphisms of
Equation~\eqref{eq:composition_equivalence}.  On the level of
morphisms, the key point is just that the composition is coherently
compatible with the product, using the coherence of pullback.  The
decompositions of Equation~\eqref{eq:bicategory_composition_q} and the
resulting split surjections of vector bundles are
similarly coherently compatible with the composition.  The unitors are
constructed similarly.  It is tedious but straightforward to verify
that the pentagon identity and the unit identities are satisfied.

We conclude our discussion by lifting the notion of a strong
equivalence to this setting: 
\begin{defin}
  A \emph{strong equivalence} between objects $\bX$ and $\bX'$ of
  $d\Orb^O(V_0,V_1)$ consists of:
  \begin{enumerate}
    \item a strong equivalence of the underlying orbifolds (with
    underlying vector bundle $N$),
  \item a vector bundle $W$ over the orbifold underlying $\bX$, and
    \item  an isomorphism
  \begin{equation}
    E^{\pm}_{\bX'} \cong E^{\pm}_{\bX} \oplus N \oplus W,
  \end{equation}
with the property that the relative framing of $\bX'$ agrees with
the direct sum of the relative framing of $\bX$ with the identity on
$N \oplus W$.
\end{enumerate}
\end{defin}

\subsection{Stably framed and stably complex derived orbifolds}
\label{sec:stably-framed-stably}

From the point of view of the abstract theory, it would seem most
natural to define a bicategory of stably complex derived orbifolds to
consist of those derived orbifolds with the structure of a complex
structure on the virtual vector bundle $E$. However, this is not
compatible with Equation~\eqref{eq:add_normal_directions_to_E+}, so we
would need to formulate the complex structure using an additional
stabilization procedure on $E$.

A slightly different problem arises when formulating the notion of
stably framed charts: the natural condition to impose is that the
virtual vector bundle $E$ is constant, i.e., is identified with a
trivial virtual vector bundle $U$. This would imply that the derived
orbifold is framed relative $U \oplus V_1 \ominus V_0$. However, in
concrete situation, we would like to formulate that our chart is
\emph{stably framed} relative a fixed vector bundle. 

A solution to these technical issues is to introduce a decomposition
of the virtual vector bundle to which we compare the tangent space
into three factors: a stable vector bundle which we denote by $I$,
that in geometric situations will be an index bundle, a stable vector
space denoted $U$ which arises from automorphisms or parameter spaces
of the equations defining the moduli spaces, and a vector bundle
denoted $W$ that is required for the stabilisation procedure:

\begin{defin}
The bicategory $d\Orb^{U}$ of \emph{stably complex derived orbifolds}
is the bicategory obtained from $d\Orb^O$ with the same $0$-cells, with
$1$-cells given by the $1$-cells of $d\Orb^O$ together a vector bundle
$W_{\bX}$ over $\bX$ for each object, and a decomposition of virtual
bundles 
\begin{equation}
  E_{\bX} \cong (W_{\bX}, W_{\bX}) \oplus I_{\bX} \oplus U_{\bX},  
\end{equation}
where $I_{\bX}$ is a complex virtual vector bundle, and $U_{\bX}$ is a
virtual vector space. 

The $2$-cells are given by the $2$-cells of $d\Orb^O$, with the property that the split surjections of Equations~\eqref{eq:stabilise_E-_sur} and \eqref{eq:add_normal_directions_to_E+_sur} split as a direct sum of maps
\begin{align}
  W_\bX  & \leftarrow W_{\bX'} \\
     U^+_{\bX} \oplus \bR^{Q^+_f}   & \cong U^+_{\bX'} \\ 
     U^-_{\bX}  & \cong U^-_{\bX'} \oplus \bR^{Q^-_f} \\ \label{eq:maps_index_bundles}
        I^{\pm}_{\bX}  & \leftarrow I^{\pm}_{\bX'},
\end{align}
where the middle two maps are maps of vector spaces, and the last map respects the complex structure. Writing $W_f$ for the kernel of the first map and $I_f$ for the kernel of the last one, we require the identification between these kernels give rise to a decomposition  
\begin{equation}
E_{f} \cong (W_{f},W_f) \oplus  I_{f}.
\end{equation}
The composition in the $2$-cells of $d\Orb^{U}$ are then induced by composition in $d\Orb^O$. 
 
Finally, the composition functors in $d \Orb^U$ lift Equation \eqref{eq:composition_functor_dOrbO}  $ d \Orb^O$ by setting 
  \begin{align}
    W_{02} & \equiv W_{01} \oplus W_{12} \oplus (V^-_1,V_1^-) \\
    U_{02} & \equiv U_{01} \oplus U_{12} \\
    I_{02} & \equiv I_{01} \oplus I_{12} \\
    Q^{\pm}_{02} & \equiv Q^{\pm}_{01} \amalg Q^{\pm}_{12}.
  \end{align}
\end{defin}
Analogous checks to the discussion of Section \ref{sec:bicat-relat-fram} show that adding this
augmented data retains the necessary coherences to form a bicategory.

In this setting, we require additional structures on strong equivalences:

\begin{defin}
  A \emph{strong equivalence} between objects $\bX$ and $\bX'$ of $d\Orb^U(V_0,V_1)$ consists of a strong equivalence in $d \Orb^O(V_0,V_1)$ between their images under the forgetful map, which we require to give rise to decompositions
  \begin{align}
    I^\pm_{\bX'} & \cong I^\pm_{\bX} \oplus I_f \\
    W_{\bX'} & \cong W^\pm_{\bX} \oplus W_f,
  \end{align}
  that induces a complex structure on $I_f$ in the first case.
\end{defin}

There is a straightforward extension of the constructions we described
to other classical bordism groups, but we want to highlight one
particular case: 
\begin{defin}
The bicategory $d\Orb^{fr}$ of \emph{stably framed derived orbifolds}
is the bicategory obtained from $d\Orb^U$ with the same $0$-cells, with
$1$-cells given by a $1$-cell in $d\Orb^U$ for which $I_{\bX}$ is
trivial, and $2$-cells given by a $2$-cell in $d\Orb^U$ for which
$I_{f}$ is the $0$ vector space. 
\end{defin}

The notion of strong equivalence for framed derived orbifolds is
inherited from the corresponding notion for stably complex derived
orbifolds.

\begin{defin}
A \emph{strong equivalence} between objects $\bX$ and $\bX'$ of
$d\Orb^{fr}(V_0,V_1)$ consists of a strong equivalence in $d
\Orb^U(V_0,V_1)$ between the underlying objects of
$d\Orb^U(V_0,V_1)$.
\end{defin}

We refer to both the complex oriented and framed cases as
\emph{structured derived orbifolds}, and generically refer to them as
$d \Orb^{\cS}$.  In both cases, we have the following straightforward
lifting property, which is established by taking the direct sum of the
structure maps with the identity on the normal bundle:

\begin{lem} \label{lem:canonical_lift_structure_equivalence}
If $\bX \to \bX'$ is a strong equivalence in $d \Orb$, and
$\tilde{\bX}$ is a lift of $\bX$ to $d \Orb^{\cS}$, then there exists
a canonical lift $ \tilde{\bX}'$ of $\bX'$ to $d \Orb^{\cS}$, and a
lift of the strong equivalence to a strong equivalence from
$\tilde{\bX} $ to $ \tilde{\bX}' $. \qed
\end{lem}

\section{Flow categories}
\label{sec:flow-categories}

The purpose of this section is to define precisely what we mean by a
flow category and a structured flow category.  That is, we will
specify the vertices of the stable $(\infty,1)$-category of flow
categories.  Most of the work of the section is to make precise the
idea that a (structured) flow category is a stratified category
enriched in (structured) derived orbifolds.

\subsection{Stratifying categories}
\label{sec:strat-part-order}

We fix an abelian monoid $\Gamma$, equipped with a homomorphism
\begin{equation}
  \cA \co \Gamma \to [0,\infty),
\end{equation}
which we refer to as the energy.  The key examples are given by $\Gamma
= \{0\}$, and $\Gamma = [0,\infty)$.  We will generically use additive
  notation for $\Gamma$; the unit is denoted by $0$ and the operation
  by $+$.

We will work with categories graded by $\Gamma$; a $\Gamma$-graded
category $\cC$ is specified by a (small) category $\cC$ and a functor
$\cC \to \Gamma$, where $\Gamma$ is regarded as a category with a
single object.  Explicitly, this means that for each pair of objects
$x,y \in \ob(\cC)$ and $\gamma \in \Gamma$ there is a decomposition of
the morphism sets
\begin{equation}
\cC(x,y) = \coprod_{\gamma \in \Gamma} \cC^{\gamma}(x,y)
\end{equation}
such that the composition on
$\cC$ restricts to a map
\begin{equation}
\cC^{\gamma_1}(x,y) \times \cC^{\gamma_2}(y,z) \to \cC^{\gamma_1 +
  \gamma_2}(x,z).
\end{equation}
We will work with $\Gamma$-graded
categories without identities.

We can form $2$-categories whose $1$-cells are $\Gamma$-graded
categories (possibly without identities) as follows: regard $\Gamma$
as a category enriched in categories where for $\gamma, \gamma' \in
\ob(\Gamma)$, the morphism category $\Gamma(\gamma, \gamma')$ has only
identity morphisms.  Then we define a $\Gamma$-graded strict
$2$-category $\scrC$ as specified by a strict $2$-functor $\scrC \to
\Gamma$.  Note that this definition makes sense for strict non-unital
$2$-categories; i.e., categories enriched in non-unital categories.

We now turn to our main definitions for this section.  Let $\cP$ be a
set. For each pair $(p,q)$ of elements of $\cP$ and
each element $\lambda$ of $\Gamma$, we define a category as follows:

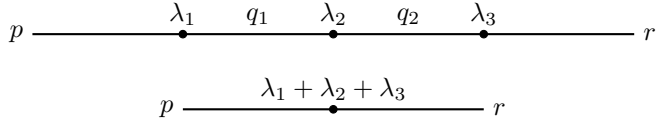
\begin{figure}
  \centering
   \begin{tikzpicture}
     \coordinate[label= left:$p$] (p) at (-4,0);
     \coordinate[label= right:$r$] (r) at (4,0);
     \draw [thick] (p) -- (r);
     \coordinate[label= above:$q_1$] (q1) at (-1,0);
     \coordinate[label= above:$\lambda_1$] (l1) at (-2,0);
     \draw [fill,color=black] (l1) circle (.05);
      \coordinate[label= above:$q_2$] (q2) at (1,0);
     \coordinate[label= above:$\lambda_2$] (l2) at (0,0);
     \draw [fill,color=black] (l2) circle (.05);
     \coordinate[label= above:$\lambda_3$] (l3) at (2,0);
     \draw [fill,color=black] (l3) circle (.05);
     \begin{scope}[shift={(0,-1)}]
      \coordinate[label= left:$p$] (p) at (-2,0);
     \coordinate[label= right:$r$] (r) at (2,0);
     \draw [thick] (p) -- (r);
     \coordinate[label= above:$\lambda_1 + \lambda_2 + \lambda_3$] (l) at (0,0);
     \draw [fill,color=black] (l) circle (.05);
   \end{scope}
   \end{tikzpicture}
   \caption{There is a morphism in $\cP^\lambda(p, r) $ from the bottom to the top labeled arc, given by collapsing the interior edges. }
  \label{fig:arcs_labeled_byP}
\end{figure}

\begin{defin}\label{def:poset-flow-categories}
For $p,q \in \cP$, the category $\cP^\lambda(p, q)$ has
\begin{enumerate}
\item objects  given by directed arcs (i.e., rooted trees with only bivalent
  vertices) with edges labeled by elements of $\cP$ and vertices
  labeled by elements of $\Gamma$ whose sum is $\lambda$, so that the incoming
  leaf is labeled by $p$ and the outgoing leaf by $q$. 
\item morphisms from $\gamma_1$ to $\gamma_2$ given by collapsing internal edges of $\gamma_2$
  (dropping the labels associated to the collapsed edges), adding
  the elements of $\Gamma$ associated to the endpoints of collapsed
  edges, such that the resulting arc is $\gamma_1$.
\end{enumerate}
\end{defin}

It is straightforward to see that the tree with a single vertex
labeled by $\lambda$ and no interior edges is the initial object of
$\cP^\lambda(p, q)$ and to prove the following result.

\begin{lem}
The category $\cP^\lambda(p, q)$ is a model for
manifolds with corners. \qed
\end{lem}

We have a natural map
\begin{equation}
    \cP^{\lambda}(p, q) \times   \cP^{\mu}(q, r) \to \cP^{\lambda + \mu} (p, r)
\end{equation}
given by concatenation of trees. This map is associative, which
justifies the following definition.

\begin{defin}
The non-unital $\Gamma$-graded strict $2$-category $\cP^{\Gamma}$ has
object set $\cP$ and morphism categories $\cP^{\Gamma}(p,q)$ given by the union of the categories $ \cP^\lambda(p, q)$.
\end{defin}

Since $\cP^{\Gamma}(p,q)$ is a disjoint union of connected categories,
a derived orbifold $\bX$ stratified by
$\cP^{\Gamma}(p,q)$ decomposes as a union of components
$\bX^{\lambda}$ where $\bX^{\lambda}$ is stratified by
$\cP^{\lambda}(p,q)$.  (Note that $\bX^\lambda$ might itself consist
of multiple connected components.)

\subsection{Unstructured Flow categories}
\label{sec:unstr-flow-categ}

We are finally in a position to define our most basic notion of an
(unstructured) flow category.  Recall from
Lemma~\ref{lem:stratifying_2category} that associated to a category
$\bX$ enriched in $d\Orb$ is a strict $2$-category $\cP_{\bX}$
encoding the fact that each morphism object $\bX(p,q)$ is a
$\cP_{\bX(p,q)}$-stratified derived orbifold.

\begin{defin}\label{def:flow_category}
A \emph{flow category} $\bX$ with object set $\cP$ is a non-unital
category enriched in the category $d\Orb$ of derived orbifolds,
stratified by $\cP^{\Gamma}$ in the sense that it is equipped with an
isomorphism of $2$-categories $\cP_{\bX} \cong \cP^{\Gamma}$; i.e.,
$\bX(p,q)$ is stratified by $\cP^{\Gamma}(p,q)$ and the stratification
is compatible with composition.

Furthermore, $\bX$ satisfies the
property that, for each element $p$ of $\cP$, the projection map  
\begin{equation}
\coprod_{q} \bX(p,q) \to [0,\infty)   
\end{equation}
is proper in the sense that, for each finite interval $[0,E]$, the zero set of
\begin{equation}
    \coprod_{q}  \coprod_{\cA(\lambda) \in [0,E]} \bX^\lambda(p,q)
\end{equation}
is compact. When we exceptionally need to drop this requirement, we shall specifically use the term \emph{non-proper flow category}.
\end{defin}

Concretely, this means that $\bX$ consists of the following data:
\begin{enumerate}
\item A derived orbifold $\bX^\lambda(p,q)$, stratified by $\cP^\lambda(p,q)$, for each pair $(p,q)$ of objects and each element $\lambda$ of $\Gamma$,
\item A morphism of derived orbifolds
\begin{equation}
\bX^{\lambda_1}(p,q) \times \bX^{\lambda_2}(q,r) \to \bX^{\lambda_1 + \lambda_2}(p,r)     
\end{equation}
lifting the morphism
\begin{equation}
\cP^{\lambda_1}(p,q) \times \cP^{\lambda_2}(q,r) \to \cP^{\lambda_1 + \lambda_2}(p,r)
\end{equation}
of stratifying categories for each triple
$(p,q,r)$, so that the following diagram (strictly) commutes for every quadruple
$(p,q,r,s)$: 
\begin{equation}
\begin{tikzcd}
\bX^{\lambda_1}(p,q) \times \bX^{\lambda_2}(q,r) \times \bX^{\lambda_3}(r,s) \ar[r] \ar[d] & \bX^{\lambda_1 + \lambda_2}(p,r)  \times \bX^{\lambda_3}(r,s) \ar[d] \\
\bX^{\lambda_1}(p,q) \times \bX^{\lambda_2 + \lambda_3}(q,s) \ar[r] & \bX^{\lambda_1 + \lambda_2 + \lambda_3}(p,s).
\end{tikzcd}
\end{equation}
\end{enumerate}
Moreover, we note that, as discussed in Remark
\ref{rem:properness_flow_category}, the set $\cP$ of objects of $\bX$
acquires a partial order from the condition that the subset  of the
morphism spaces with trivial energy be non-empty. In mild abuse of
terminology, we shall sometimes mention the partial order on $\cP$
without specifying the fact that it comes from $\bX$.

\begin{rem} \label{rem:imposing_poset}
  An alternative approach to the construction of this section would be to equip $\cP$ with a partial order, and then impose, in Definition \ref{def:poset-flow-categories}, the following condition on objects:
  \begin{equation}
\parbox{31em}{if a vertex is labelled by an element $\lambda_v \in \Gamma$ with trivial energy, then the previous and subsequent elements $(q_i, q_{i+1})$ satisfy $q_i < q_{i+1}$.}    
\end{equation}
Incorporating this restriction into the definition, one finds that $ \cP^{\lambda}(p,r)$ is a partially ordered set. At this stage, Definition \ref{def:flow_category} would be adapted to impose the requirement that $\bX^{0}(p,q) $ is empty unless $p<q$, which then implies that the partial order $\cP$ refines the partial ordering on objects induced by the non-emptiness of morphisms of energy $0$.
\end{rem}

\subsection{Structured Flow categories}
\label{sec:struct-flow-categ}

The purpose of this section is to define a notion of flow categories
with tangential structure, which amounts to lifting the
unstructured data to $d\Orb^{\cS}$. To this end, we make the following
useful observation:
\begin{lem} \label{lem:identify_normal_direction}
  The set $Q_{\iota_q}$ arising as the image of the minimal element of
  $ \cP^{\lambda_1}(p,q) \times \cP^{\lambda_2}(q,r)$ for the
  composition map
  \begin{equation}
 \iota_q \co \cP^{\lambda_1}(p,q) \times \cP^{\lambda_2}(q,r) \to
 \cP^{\lambda_1 + \lambda_2}(p,r)
  \end{equation}
   is canonically identified with the singleton $\{q\}$. \qed
\end{lem}

Next, we recall the definition of a (small) category enriched in a
bicategory.  Given a set $O$ and an enriching bicategory $B$, a
(unital) category with object set $O$ enriched in $B$ is specified by
a lax functor $O \to B$, where $O$ is regarded as the codiscrete
(chaotic) category on the underlying set.  The non-unital variant is
specified by dropping the unit conditions on the lax functor $O \to
B$.  Explicitly, a category $\cC$ enriched in $B$ is specified by
giving for each object $x \in ob(B)$ a set $\cC_x$ of objects over
$x$, for each pair of objects $a \in \cC_x$ and $b \in \cC_y$
respectively an arrow in $B(a,b)$, and for each triple of objects $a
\in \cC_x$, $b \in \cC_y$, and $c \in \cC_z$, a $2$-cell $B(a,b)
\times B(b,c) \Rightarrow B(a,c)$.  An enriched functor $F$ between
categories $\cC$ and $\cD$ enriched over $B$ is specified by a
function that takes an object $a \in \cC_x$ to $Fa \in \cD_x$ and for
a pair of objects $a \in \cC_x$ and $b \in \cC_y$ specifies a $2$-cell
$\cC(a,b) \to \cC(Fa, Fb)$.  (See for example~\cite{Street1983} for
more details.)
  
\begin{defin}
A \emph{lift} of a flow category $\bX$ with object set $\cP$ to
$d\Orb^\cS$ is a category enriched in $d\Orb^\cS$ whose image under
the forgetful pseudofunctor $d\Orb^\cS \to d\Orb$ is $\bX$.  That is,
we have a (non-unital) lax functor $\cP \to d\Orb^{\cS}$
such that the composite
\begin{equation}
\cP \to d\Orb^{\cS} \to d\Orb
\end{equation}
is equal to $\bX$.
\end{defin}

We impose some additional conditions on the lift to specify a
structured flow category.

\begin{defin} \label{def:framed_flow_category}
A \emph{structured flow category} $\bX$ with object set
$\cP$ is a lift of a flow category from $d\Orb$ to $d\Orb^{\cS}$,
satisfying the following properties: 
\begin{enumerate}
\item The decomposition of $Q_{\iota_q}$ associated to the codimension-$1$ boundary strata is
 \begin{equation}
   Q_{\iota_q}^{-} = \{q\} \qquad Q_{\iota_q}^{+}= \emptyset.
 \end{equation}
\item  The virtual vector space
    $U_{\bX(p,q)}$ is given by $(0,\bR^{\{q\}})$.
   \item The isomorphism
\begin{equation} \label{eq:isomorphism-product-compatible-decomposition}
U^{-}_{\bX(p,q) \times \bX(q,r)}  \cong  U^-_{\bX(p,r)} \oplus \bR^{\{q\}}
\end{equation}
is given by the identity map on $\{q,r\}$, where we use the fact that
\begin{equation}
U^{-}_{\bX(p,q) \times \bX(q,r)}  \cong U^{-}_{\bX(p,q)} \oplus
U^-_{\bX(q,r)}.
\end{equation}
\end{enumerate}
\end{defin}

Explicitly, since $d\Orb^{\cS}$ is a bicategory, the first datum of
the lift is a choice of virtual vector space $V_p$ for each element of
$\cP$. Next, for each pair $(p,q)$, and each $\lambda \in \Gamma$, we
pick a vector bundle $W^\lambda(p,q)$ on $\bX(p,q)$,  a structured
vector bundle $I^\lambda(p,q)$, and an equivalence 

\begin{equation}
T \bX^\lambda(p,q) \oplus V_q \oplus \bR^{q} \oplus W^\lambda(p,q)
\cong I^\lambda(p,q) \oplus V_p \oplus W^\lambda(p,q)       
\end{equation}
of virtual vector bundles over $\bX^\lambda(p,q)$.  

Omitting the subscripts which record the energy, these isomorphisms
then satisfy the following associativity property: for each triple
$(p,q,r)$ of objects, the following diagram of equivalences of virtual
vector bundles over $\bX(p,q) \times \bX(q,r) $ commutes: 
    \begin{equation}
      \begin{tikzcd}
        \begin{gathered}
          T \bX(p,q) \oplus V^-_q \oplus \bR^{q} \oplus W(p,q)  \\
           T \bX(q,r) \oplus V_r \oplus \bR^{r} \oplus W(q,r)
        \end{gathered}
        \ar[r] \ar[d] &
        \begin{gathered}
          I(p,q) \oplus V_p \oplus W(p,q)   \oplus \\
         I(q,r) \oplus V^-_q  \oplus W(q,r)
       \end{gathered} \ar[d]  \\
        T \bX(p,r) \oplus V_r \oplus \bR^{r} \oplus W(p,r) \ar[r] & I(p,r) \oplus V_p \oplus W(p,r).
      \end{tikzcd}
    \end{equation}

\section{The semisimplicial set of flow categories}
\label{sec:quas-struct-floe}

In this section we construct a semisimplicial set with vertices flow
categories and $1$-simplices flow bimodules.  In subsequent sections
we will show that this is the underlying semisimplicial set of a
quasicategory.

\subsection{Stratifying categories associated to a sequence of flow categories}
\label{sec:strat-categ-assoc}

Before providing the formal definitions that underly the construction
of the category of flow categories, we discuss the example of Morse
theory, which might help give the reader some intuition about our
goals. 

\begin{rem}
  Say that $\vec{f} = (f_0, \ldots, f_n)$ is a sequence of Morse
  functions, and that we are given continuations maps $f_{ij}$ between
  $f_i$ and $f_j$ whenever $i<j$. In the formal definition we will
  presently give, the Morse functions will label the vertices of the
  $n$-simplex, and the continuation maps its edges, and the goal is to
  understand the combinatorial data associated to choosing families of
  continuation maps relating all possible choices of compositions
  among the continuation maps $f_{ij}$. The first nontrivial case is
  that of a $2$-simplex whose edges are labeled by continuation maps
  $f_{ij}$, $f_{jk}$, and $f_{ik}$: in the parlance of Floer theory
  one can compose $f_{ij}$ and $f_{jk}$ to a \emph{broken}
  continuation map from $f_i$ to $f_k$, and ask for a family,
  parametrized by an interval, relating this broken continuation map
  to $f_{ik}$.

The full structure can be encoded by an $(n-1)$-cube (with coordinates
labeled by $\{1, \ldots, n-1\}$) of families of broken continuation
maps from $f_0$ to $f_n$. The vertices of this cube determine a
composition of continuation maps, with $f_{ij}$ appearing in the
composition (assuming $1 \leq i < j \leq n-1$) if and only if the
coordinates associated to $i$ and $j$ are both equal to $0$, and all
coordinates strictly between $i$ and $j$ equal $1$. The two facets
associated to setting a coordinate $1 \leq i \leq n$ equal to either
$0$ or $1$ are then respectively given  by (i) a family of continuation maps
for the sequence of Morse functions obtained by omitting $f_i$ or (ii)
a family of (broken) continuation maps obtained by composition a
product of families of continuation maps for the sequences of Morse
functions $(f_0, \ldots, f_i)$ and $(f_i, \ldots, f_n)$. There are
various compatibility conditions between the families associated to
these codimension-$1$ boundary strata, which are best encoded by
Definition \ref{def:Poset-flow-simplex} below.

  As a final observation on the structure of these moduli spaces, note that the codimension $1$ boundary strata associated to a coordinate $i$ behave differently according to whether we set this coordinate equal to $0$ or $1$: in the former case, there is a codimension-$1$ boundary stratum associated to each critical point of the function $f_i$, while in the second case (with $f_i$ omitted from the sequence), there is a single codimension $1$ boundary stratum. This dichotomy will lead us to have to separate some of our discussions into cases.
\end{rem}

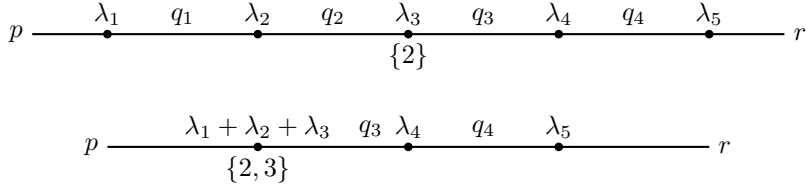
\begin{figure}[h]
  \centering
   \begin{tikzpicture}
     \coordinate[label= left:$p$] (p0) at (-5,0);
     \coordinate[label= right:$r$] (p4) at (5,0);
     \draw [thick] (p0) -- (p4);
     \coordinate[label= above:$q_1 $] (q1) at (-3,0);
     \coordinate[label= above:$\lambda_1$] (l1) at (-4,0);
     \draw [fill,color=black] (l1) circle (.05);
      \coordinate[label= above:$q_2$] (q1p) at (-1,0);
     \coordinate[label= above:$\lambda_2$] (l2) at (-2,0);
     \draw [fill,color=black] (l2) circle (.05);
     \coordinate[label= above:$\lambda_3$] (l3) at (0,0);
     \coordinate[label= below:$\{2\}$] (l3) at (0,0);
     \draw [fill,color=black] (l3) circle (.05);
     \coordinate[label= above:$q_3$] (q4) at (1,0);
     \coordinate[label= above:$\lambda_4$] (l4) at (2,0);
      \draw [fill,color=black] (l4) circle (.05);
      \coordinate[label= above:$q_4$] (q4p) at (3,0);
      \coordinate[label= above:$\lambda_5$] (l5) at (4,0);
       \draw [fill,color=black] (l5) circle (.05);
     \begin{scope}[shift={(0,-1.5)}]
      \coordinate[label= left:$p$] (p1) at (-4,0);
     \coordinate[label= right:$r$] (p4) at (4,0);
     \draw [thick] (p1) -- (p4);
     \coordinate[label= above:$\lambda_1 + \lambda_2 + \lambda_3$] (l) at (-2,0);
   \coordinate[label= below:$\protect{\{1,2,3\}}$] (l3) at (-2,0);
      \draw [fill,color=black] (l3) circle (.05);
     \coordinate[label= above:$q_3$] (q4) at (-.5,0);
     \coordinate[label= above:$\lambda_4$] (l4) at (0,0);
      \draw [fill,color=black] (l4) circle (.05);
      \coordinate[label= above:$q_4$] (q4p) at (1,0);
      \coordinate[label= above:$\lambda_5$] (l5) at (2,0);
      \draw [fill,color=black] (l5) circle (.05);
   \end{scope}
   \end{tikzpicture}
   \caption{There is a morphism in $\cP^\Gamma(p, r)$ from the bottom to the top labeled arc, with $q_1$ and $q_2$ in $ \cP_1$ and $q_3$ and $q_4$ in $\cP_4$,  given by collapsing two interior edges and adding the element $3$ to the set of labels of one of the vertices.}
  \label{fig:arcs_labeled_byvecP}
\end{figure}

\begin{defin}\label{def:Poset-flow-simplex}
Given a sequence $\vec{\cP} = (\cP_0, \ldots, \cP_n)$ of sets, and a
pair of elements $p \in \cP_j$ and $r \in \cP_\ell$ (for $0 \leq j
\leq \ell \leq n$), the category
$\vec{\cP}^{\Gamma}(p,r)$ has objects directed arcs
equipped with the following additional data: 
\begin{enumerate}
\item Each edge is labeled by an element of the sets $\cP_k$ for $j
  \leq k \leq \ell$, so that these sets appear in increasing
  order, the incoming edge is labeled by $p$, and the outgoing
  edge is labeled by $r$.
\item Each vertex lying on edges labeled by elements of $\cP_k$ and $\cP_{k'}$
  for $k \leq k'$ is labeled by a pair consisting of (i)
  a subset of $\{k+1, \ldots, k'-1\}$, and (ii) an element $\lambda$
  of the abelian monoid obtained by inverting the positive action
  elements of $\Gamma$.  If $k=k'$, we assume that this element lies in $\Gamma$. 
  \end{enumerate}

A morphism in $\vec{\cP}^{\Gamma}(p,r)$ is given as
follows:  on underlying arcs, a morphism from $\gamma$ to $\gamma'$
is obtained by collapsing a sequence of
consecutive edges of $\gamma'$.  For the additional data, $\gamma$
must have labels that satisfy the following constraints:
\begin{itemize}
\item the labels of the uncollapsed edges in $\gamma$ agree with the labels of the corresponding edges in $\gamma'$, and
\item the label of each vertex $v$ of $\gamma$ contains the union of
  the labels of the vertices $v'$ of $\gamma'$ which are collapsed to
  it, together with any element of the sequence $\vec{\cP}$ with the
  property that $\gamma'$ contains an edge labeled by this element
  and all such edges are collapsed to $v$.
\end{itemize}
\end{defin}

\begin{rem}
One can restrict further to the situation where the labels of vertices
always lie in $\Gamma$ itself. This will give rise to a theory
controlling \emph{filtered equivalences} of flow categories and is
important for quantitative applications.
\end{rem}

\begin{example}
In the case $n=0$,
Definition~\ref{def:Poset-flow-simplex} recovers the construction of
Definition~\ref{def:poset-flow-categories}; the edges are all labeled
by elements of the set $P_0$ and the conditions on the labels of the
vertices reduce to simply specifying an element of $\Gamma$.
\end{example}

Additionally, observe that the category $
\vec{\cP}^{\Gamma}(p,r)$ depends only on the subsequence of
$\vec{\cP}$ consisting of the elements between $\cP_j$ and $\cP_\ell$.
Moreover, the minimal elements are given by an arc with a single
vertex labeled by the sequence $\{1, \ldots, n-1\}$, and an arbitrary
element of the localization of $\Gamma$.

\begin{lem} \label{lem:model_manifold_corners_simplex}
The category $\vec{\cP}^{\Gamma}(p,r)$ is a model for manifolds with corners.
\end{lem}

\begin{proof}
The overcategory of any object is given by the powerset of the union
of (i) the set of internal edges of the arc, and (ii) all elements of
the sequence $\vec{\cP}$ (between $\cP_j$ and $\cP_\ell$) which do not
appear in the label of a vertex.
\end{proof}

As in the case of a singleton, the categories
$\vec{\cP}^{\Gamma}(p,r)$ assemble into a strict $2$-category which we
denote $\vec{\cP}^{\Gamma}$.

\begin{defin}
The strict $2$-category $\vec{\cP}^{\Gamma}$ has objects the
disjoint union $\coprod_i \ob(\cP_i)$ and morphism the categories
$\vec{\cP}^{\Gamma}(p,r)$, with natural strictly associative
composition map
\begin{equation}
\vec{\cP}^{\Gamma}(p,q) \times \vec{\cP}^{\Gamma}(q,r) \to
\vec{\cP}^{\Gamma}(p,r).
\end{equation}
\end{defin}

Let $\partial^{j} \vec{\cP}$ denote the sequence obtained by omitting
the $j$\th element from $\vec{\cP}$.  Observe that for each $j$ we
have natural inclusions of strict $2$-categories
\begin{equation}\label{eq:inclusion_categories_forget_element_sequence}
  \partial^{j} \vec{\cP}^{\Gamma} \to \vec{\cP}^{\Gamma}.
\end{equation}
It straightforward to check that these maps satisfy the following
simplicial identities:
 
\begin{lem}\label{lem:poset-bicategory-semisimplicial}
For $0 \leq i < j \leq n$, there is an equality of strict
$2$-categories 
\begin{equation}
\partial^{i} \partial^{j} \vec{\cP}^{\Gamma} = \partial^{j-1}
\partial^i \vec{\cP}^{\Gamma}.
\end{equation}
Moreover, these identifications are compatible with the natural inclusion
of Equation~\eqref{eq:inclusion_categories_forget_element_sequence}. \qed
\end{lem}

\subsection{The semisimplicial set of unstructured flow categories}
\label{sec:simpl-set-unstr}

We are now ready to define the most basic version of the
semisimplicial set of flow categories.

\begin{defin} 
  An \emph{elementary flow simplex} consists of
  \begin{enumerate}
    \item a sequence $\vec{\cP}$
      of  sets, and
      \item a lift $\bX$ of the category
$\vec{\cP}^{\Gamma}$ to $d\Orb$ in the sense that there is a strict
$2$-functor $\cP_{\bX} \to \vec{\cP}^{\Gamma}$, which satisfies the
      following property:
      \begin{equation} \label{eq:lower_bound_flow_simplex}
    \parbox{30em}{ the energy map on the union of the morphism space
      with source $p$ is a proper map to $\bR$, with uniform lower
      bound that is independent of $p$.}
  \end{equation}
  \end{enumerate}
\end{defin}

We recall from Section \ref{sec:derived_orbifold} that properness
refers to compactness of the $0$-locus, so that this definition
specializes to Definition \ref{def:flow_category} when the sequence of
 sets $\vec{\cP}$ consists of a singleton.

\begin{defin}
Given an elementary flow simplex $\bX$, and a stratum $\sigma$ of
$\Delta^n$ of dimension $k$, we obtain an elementary flow $k$-simplex
$\partial^\sigma \bX$ by restricting $\bX$ to the subcategory
associated to $\partial^{\sigma} \vec{\cP}$.
\end{defin}

Explicitly,
the objects of $\partial^{\sigma} \bX$ are the elements of
$\partial^{\sigma} \vec{\cP}$, and the morphisms are given by
restricting the morphism spaces to the appropriate strata.
As in the discussion preceding Lemma \ref{lem:model_manifold_corners_simplex}, this means that the
morphisms are the same for objects lying in sets $\cP_j$ and
$\cP_\ell$ whenever all elements between $j$ and $\ell$ lie in
$\sigma$, and are otherwise given by a proper stratum of codimension equal
to the number of such elements.
 
\begin{defin} \label{def:flow_simplex}
A \emph{flow $n$-simplex} consists of a sequence $\vec{\cP}$ of
 sets, and elementary flow simplices $\bX_{\sigma}$
lifting $\partial^\sigma \cP$ for each stratum $\sigma$ of
$\Delta^n$, together with a functor enriched in $d\Orb$ 
\begin{equation}
\bX_{\tau} \to \partial^\tau \bX_\sigma    
\end{equation}
whenever $\tau$ is a subset of $\sigma$, lifting the isomorphism of
stratifying categories fixed by Equation
\eqref{eq:inclusion_categories_forget_element_sequence} so that the
following diagram commutes 
    \begin{equation} \label{eq:diagram_three_nested_strata}
      \begin{tikzcd}
        \bX_{\tau} \ar[r] \ar[d] & \partial^\tau \bX_\sigma   \ar[d] \\
        \partial^{\tau} \bX_{\rho} \ar[r] & \partial^\tau \partial^{\sigma} \bX_\rho
      \end{tikzcd}
    \end{equation}
for each triple $\tau \subset \sigma \subset \rho$.
\end{defin}
As a consequence of the fact that morphisms in $d \Orb$ are equivalences onto boundary strata we have:
\begin{lem}
If $\bX$ is a flow $n$-simplex, then for any pair $\tau \subset
\sigma$ of simplices of $\Delta^n$, and every pair of objects $p$ and
$r$ of the flow categories labelled by the vertices of $\tau$, the
morphism
\begin{equation}
\bX_{\tau}(p,r) \to \partial^\tau \bX_\sigma(p,r)     
\end{equation}
is a strong equivalence, with the property that the following diagram
of strong equivalences commutes:
\begin{equation}
\begin{tikzcd}
    \bX_{\tau}(p,q) \times \bX_{\tau}(q,r) \ar[r] \ar[d] &   \partial^\tau \bX_\sigma(p,q) \times \partial^\tau \bX_\sigma(q,r)\ar[d]  \\
    \partial^{q}   \bX_{\tau}(p,r) \ar[r]  & \partial^{q}   \partial^\tau \bX_\sigma(p,r),
\end{tikzcd}
\end{equation}
      \qed
\end{lem}

\begin{rem}
In order to understand why we do not simply use elementary simplices,
we discuss the case of the $2$-simplex. In that case, we have the
following data: 
  \begin{enumerate}
  \item a triple $\partial^{<0>} \bX$, $\partial^{<1>} \bX$, $\partial^{<2>} \bX$ of flow categories,
    \item a triple of flow $1$-simplices $\partial^{<01>} \bX$, $\partial^{<12>} \bX$, and $\partial^{<02>} \bX$ whose $0$-simplices are indicated in the notation.
  \end{enumerate}
  For a pair $(p_0,p_2)$ whose first element is an object of $\partial^{<0>} \bX$, and whose second element lies in  $\partial^{<2>} \bX$, the codimension $1$ boundary strata of $\bX(p_0,p_2)$ are of two types; for each object $p_1$, of  $\partial^{<1>} \bX$ we have a stratum equivalent to 
  \begin{equation}
      \partial^{<01>} \bX(p_0,p_1) \times \partial^{<12>} \bX(p_1,p_2) ,
  \end{equation}  
  while we also have a stratum which is equal to $\partial^{<02>} \bX (p_0,p_2)$. The fact that this second stratum is equal, rather than strongly equivalent, to $\partial^{<02>} \bX (p_0,p_2) $ implies that the semisimplicial set formed by elementary simplices fails to satisfy the property that elementary simplices that differ by stabilisation are equivalent, which will then invalidate the computations of the homotopy types of flow categories, performed in Section \ref{sec:comp-with-sphere}.
\end{rem}

\begin{defin}\label{def:flow_semisimplicial}
The set $\Flow_n$ has elements flow simplices $\bX$ such that
\begin{equation}
\ob(\bX) = \vec{\cP}^{\Gamma}, \qquad |\vec{\cP}| = n+1.
\end{equation}
For $0 \leq i \leq n$, the face maps
\begin{equation} \label{eq:boundary_map}
\partial^i \co \Flow_n \to \Flow_{n-1},
\end{equation}
are given by restricting a flow simplex $\bX$ to the subcategory
associated to $\partial^{i} \vec{\cP}$.
\end{defin}

\begin{rem}
In order to handle set-theoretic concerns about the size of
$\Flow_\bullet$, we will tacitly appeal to the approach taken
in~\cite[1.2.15]{Lurie2009} to handle size issues; we assume we have
chosen a series of inaccessible cardinals and work with the
Grothendieck universes they specify.
\end{rem}

Explicitly, the objects of $\partial^i \bX$ are given by those elementary flow simplices $\bX_\tau$ with $\tau$ a subset of the facet obtained by omitting $i$. 
It is evident from the construction (and
Lemma~\ref{lem:poset-bicategory-semisimplicial}) that the boundary
maps satisfy the simplicial identities, so we conclude:

\begin{lem}\label{lem:unstructured_flow_semisimplicial}
The face maps of Equation~\eqref{eq:boundary_map} equip the 
collection of sets $\Flow_\bullet = \{\Flow_n\}$ with the structure of
a semisimplicial set. 
\end{lem}

As promised in the introduction, the morphisms of $\Flow$ (i.e., the
elements of $\Flow_1$) by definition consist of a pair of sets $\vec{\cP} = (\cP_0, \cP_1)$, flow categories $\bX_{01}$
on $\vec{\cP}$, $\bX_0$ on $\cP_0 = \partial^0 \vec{\cP}$, and $\bX_1$
on $\cP_1 = \partial^1 \vec{\cP}$, and morphisms $\bX_0 \to \partial^0
\bX_{01}$ and $\bX_1 \to \partial^1 \bX_{01}$.  Our hypotheses yield
compatible maps of derived orbifolds
\begin{align}
\bX_0(p,q) \times \bX_{01}(q,r) & \to \bX_{01}(p,r) \\ \bX_{01}(p,q)
\times \bX_1(q,r) & \to \bX_{01}(p,r)
\end{align}
that are associative with respect to the composition in $\bX_0$ and
$\bX_1$.  That is, this data specifies a (non-unital) graded bimodule
structure on $\bX_{01}$ over the flow categories $\bX_0$ and $\bX_1$;
this is what we refer to as a {\em flow bimodule}.

\subsection{The semisimplicial set of structured flow categories}
\label{sec:semisimplicial-set}

We now adapt the previous definition to the setting of complex
oriented or framed flow categories. As a starting point, we have the
following analogue of Lemma \ref{lem:identify_normal_direction}: 
\begin{lem} \label{lem:inclusion_map_forget_element_sequence}
Every codimension $1$ object of the category $\vec{\cP}^{\Gamma}(p, r)$ is given by either
    \begin{enumerate}
    \item an element  $j < k < \ell$, corresponding to the inclusion
\begin{equation} 
  \iota_k \co  \partial^{k} \vec{\cP}^{\Gamma}(p, r) \to \vec{\cP}^{\Gamma}(p, r),
\end{equation}
\item or an element $q$ of $\cP_{k}$ for $j \leq k \leq \ell$, corresponding to the inclusion
 \begin{equation} 
  \iota_{q} \co   \vec{\cP}^{\Gamma}(p, q) \times \vec{\cP}^{\Gamma}(q,r)  \to \vec{\cP}^{\Gamma}(p, r).
\end{equation} 
    \end{enumerate} \qed 
  \end{lem}
  In the first case, we set
   \begin{equation} \label{eq:decomposition_normal_direction-forget_vertex}
     Q_{\iota_k}^{-} = \emptyset \qquad Q_{\iota_k}^{+}=  \{k\},
   \end{equation}
while in the second case we set   \begin{equation}\label{eq:decomposition_normal_direction-differential}
     Q_{\iota_{q}}^{-} =  \{q\} \qquad Q_{\iota_{q}}^{+}=  \emptyset.
   \end{equation}
   
\begin{defin} \label{def:framed_flow_simplex-elementary}
A \emph{structured elementary flow simplex} $\bX$ with
object set $\vec{\cP}$ is a lift of an elementary flow simplex from
$d\Orb$ to $d\Orb^{\cS}$, in the sense that we
have a category enriched in $d\Orb^{\cS}$ such that under the forgetful
pseudofunctor $d\Orb^{\cS} \to d\Orb$ the following properties hold
whenever $p$ lies in $\cP_j$ and $r$ in $\cP_\ell$:
\begin{enumerate}
\item The virtual vector space
  $U_{\bX(p,r)}$ is given by $( \bR^{\{j+1, \ldots, \ell\}},\bR^{\{r\}})$.
\item For each subset $K$ of $\{j+1, \ldots, \ell-1\}$, the restriction of the framing of $\bX(p,r)$ to $\partial^K \bX(p,r) $ splits as the direct sum of a framing of this stratum with the identity on $\bR^{K}$ (where we use the inclusion of $ \bR^{K}$ in $U_{\bX(p,r)} $, the choice of decomposition in Equation~\eqref{eq:decomposition_normal_direction-forget_vertex}, and Lemma~\ref{lem:inherit_framings}).
\item The composition maps associated to an element $q$ of $\cP_k$ are given by
  \begin{align}
      U^+_{\bX(p,r)} &  \cong \bR^{\{j+1, \ldots, \ell\}}  \\
      & \cong  U^+_{\bX(p,q)} \oplus \bR^{\{k\}} \oplus U^+_{\bX(q,r)} \\
        U^-_{\bX(p,r)} \oplus \bR^{\{q\}} &  \cong    \bR^{\{q,r\}} \\
   &   \cong   U^-_{\bX(p,q)}   \oplus U^-_{\bX(q,r)}.
    \end{align} 
\end{enumerate}
\end{defin}

\begin{rem}
  The asymmetry in the definition of the virtual vector space $ U^+_{\bX(p,r)}$, including $\ell$ but not $j$, is ultimately due to the fact that we broke symmetry in the definition of the $U^{-}_{\bX(p,q)}$ when introducing flow categories.
\end{rem}

Structured elementary flow simplices assemble to produce structured
flow simplices.

\begin{defin} \label{def:framed_flow_simplex}
A \emph{structured flow $n$-simplex} is a flow $n$-simplex equipped
with a lift of the underlying elementary simplices to structured
elementary simplices, and a lift of each functor 
\begin{equation}
\bX_{\tau} \to \partial^\tau \bX_\sigma    
\end{equation}
whenever $\tau$ is a subset of $\sigma$  to a functor enriched in
$d\Orb^{\cS}$, consistently with the identifications of vector
spaces fixed in Definition \eqref{def:framed_flow_simplex}. We require 
 the lift of each diagram
    \begin{equation} \label{eq:diagram_three_nested_strata_framed}
      \begin{tikzcd}
        \bX_{\tau} \ar[r] \ar[d] & \partial^\tau \bX_\sigma   \ar[d] \\
        \partial^{\tau} \bX_{\rho} \ar[r] & \partial^\tau \partial^{\sigma} \bX_\rho
      \end{tikzcd}
    \end{equation}
to commute for each triple $\tau \subset \sigma \subset \rho$.
\end{defin}

We can define structured flow categories as above.

\begin{defin}\label{def:framed_flow_semisimplicial}
We define $\Flow_n^{\cS}$ as the collection of
structured flow $n$-simplices $\bX$ such that
\begin{equation}
\ob(\bX) = \vec{\cP}, \qquad |\vec{\cP}| = n+1.
\end{equation}
We define face maps
\begin{equation}\label{eq:framed_flow_face}
\partial^i \co \Flow_n^{\cS} \to  \Flow_{n-1}^{\cS} 
\end{equation}
exactly as in Equation \eqref{eq:boundary_map}, by restricting to the
strata indexed by $\partial^i \vec{\cP}^{\Gamma}$.
\end{defin}

To justify Definition~\ref{def:framed_flow_semisimplicial}, we need to
verify that the properties listed in
Definition~\ref{def:framed_flow_simplex} are inherited by this
construction.  It is straightforward to check that the conditions
descend to strata.  Furthermore, we have the following lemma, which
essentially follows from that check and the analysis of the basic case
in Lemma~\ref{lem:unstructured_flow_semisimplicial}.

\begin{lem}
The face maps of Equation~\eqref{eq:framed_flow_face} give
$\Flow_\bullet^{\cS} = \{\Flow_n^{\cS}\}$ the structure of a
semisimplicial set. \qed
\end{lem}

\section{Filling inner horns}
\label{sec:kan-condition-flow}

The goal of this section is to show that every inner horn in $\Flow$
(and in its structured variants) has a filler.  This establishes
$\Flow$ as a semisimplicial set that is a weak Kan complex;
constructing degeneracies for $\Flow$, done in Section
\ref{sec:quasi-units}, will later complete the construction of a
quasicategory.  The proof of the lifting for inner horns essentially
amounts to taking an appropriate gluing of the derived orbifolds that
underlie the simplices associated to the horn. To perform this
construction, we need two ideas.  The first is to raise the dimension
of the underlying orbifolds, by appropriate choices of orbibundles, so
that they have the same dimension.  The second is a smoothing result
analogous to the classical fact that manifolds with corners have
canonical smoothing to manifolds with boundary.

Recall that the horn $\Lambda^n_k$ is defined to be the simplicial set
given as the union of all of the faces of the standard simplex
$\Delta^n$ which contain the vertex $k$.  A horn is {\em inner} if $0
< k < n$.  There is a natural inclusion $\Lambda^n_k \to \Delta^n_k$,
and a filler for the map $\Lambda^n_k \to X$ is a lift to $\Delta^n_k
\to X$.  A Kan complex admits such fillers for all horns, and a
quasicategory admits such fillers for all inner horns.

For this section, we thus fix a length $n$ sequence $\vec{\cP}$ of
sets, an integer $0 < k < n$, a collection of elementary flow
simplices $\bY_{\sigma} $,  with underlying category
$\partial^{\sigma} \vec{\cP}^{\Gamma}$, indexed by the nondegenerate
simplices $\sigma$ of the horn $\Lambda^n_k$, together with
equivalences  
\begin{equation} \label{eq:equivalence_restriction_stratum_horn}
    \bY_{\tau} \to \partial^\tau \bY_\sigma    
  \end{equation}
making the diagram in equation~\eqref{eq:diagram_three_nested_strata}
commute for triples of nondegenerate simplices of the horn.  
   
Our goal is to define an $n$-simplex extending the above data:
  \begin{thm}
    \label{thm:horn_filling}
    There exists an elementary $n$-simplex $\bY$ with equivalences
  \begin{equation} \label{eq:data_of_horn}
\bY_{\sigma} \to \partial^{\sigma} \bY,
  \end{equation}
  whenever $\sigma$ is a nondegenerate simplex of the horn, so that diagram~\eqref{eq:diagram_three_nested_strata} commutes.
  \end{thm}

  For expository reasons, we shall begin with the following special case:
     \begin{assu} \label{assu:structure_maps_iso}
       The equivalences in Equation~\eqref{eq:equivalence_restriction_stratum_horn} are isomorphisms, as are all structure maps
     \begin{equation} \label{eq:structure-map-horn}
 \bY_{\sigma}(p,q) \times  \bY_{\sigma}(q,r) \to  \partial^{q} \bY_{\sigma}(p,r) .
\end{equation}
\end{assu}
Under this assumption, the orbifolds $Y_{\sigma}$ underlying the flow
simplices $\bY_{\sigma}$ themselves form a horn $\Lambda^n_k$ in the
category of orbifolds. 

We return to the general situation in Section~\ref{sec:stabilization}
below, after resolving this special case. 

\subsection{Constructing a flow simplex in the topological category}
\label{sec:constr-flow-simpl}

We now proceed to explain the construction of an elementary flow
$(n-1)$-simplex, but without the smooth structure, from the data of an
inner horn. This simplex will correspond to the missing facet of the
horn. For the construction, we recall that the objects of the
category $ \vec{\cP}^{\Gamma}(p,r)$ are (directed) arcs
with edges labelled by elements of the sets $\cP_i$, and with each
vertex labelled by a subset of $\{i+1, \ldots, j-1\}$, assuming that
the incoming edge is labelled by an elements of $\cP_i$ and the
outgoing edge by $\cP_j$. 

We define $\Lambda^n_k \vec{\cP}^{\Gamma}(p,r) $ to be the subset of
$\vec{\cP}^{\Gamma}(p,r) $ corresponding to the $k$-horn. Explicitly,
its elements are labelled trees $\alpha$ which 
\begin{equation} \label{eq:k-horn-labelled-trees}
\parbox{31em}{do not contain a vertex which is labelled either by the
  entire set $\{1, \ldots, n-1\}$, or by the complement of $k$ in this
  set.}    
\end{equation}

The given horn thus determines a contravariant functor
\begin{equation}
 \bY_{(\_)}(p, r) \co \Lambda^n_k \vec{\cP}^{\Gamma}(p,r) \to d\Orb,
\end{equation}
specified as follows.  Given a labelled tree $\alpha$ whose vertices
are labelled by a sequence $\{q_i\}_{i=0}^{d}$ of objects with $q_0 =
p$ and $q_d = r$, the value of the functor is the derived orbifold
$\bY_{\alpha}(p, r) $ which is the product, over all components of
$\alpha$ (i.e., complement of the vertices) of derived orbifold $
\bY_{\sigma}(q_i, q_{i+1}) $ associated by a nondegenerate simplex
$\sigma$ of the horn to successive pairs of objects labelling the
vertices of the tree. Assumption \ref{assu:structure_maps_iso} implies
that this functor factors through derived orbifolds with obstruction
bundles of constant rank.

Applying the forgetful map from derived orbifolds to orbifolds, consider the colimit   \begin{equation}
 \colim_{\alpha \in \Lambda^n_k \vec{\cP}^{\Gamma}(p,r)}    Y_{\alpha}(p, r)
  \end{equation}
  which in general is an orbispace.

The main point of the assumption in the special case that we consider is:
\begin{lem}
 If the maps in Equation \eqref{eq:structure-map-horn} are isomorphisms, then  $ \colim Y_{\alpha}(p, r) $ is a topological orbifold. \qed
\end{lem}

It is easy to see that, under Assumption \ref{assu:structure_maps_iso},  the obstruction bundles on $  Y_{\alpha}(p, r)$, underlying the Kuranishi spaces $\bY_{\alpha}(p, r)$ assemble to a vector bundle on this colimit so that we obtain a derived topological orbifold.  Equipping it with an appropriate smooth structure will  end up being the value of the horn-filling along the missing facet, which will follow from the construction in the next two section.

\subsection{$L$-blocks}
\label{sec:l-blocks}

We begin by considering a method for locally smoothing the strata of a manifold with corners. As a starting point, fix a constant $\epsilon$ between $0$ and $1$, and consider the manifold with corners
\begin{equation}
  L_{d,0} \subset [0,1]^{d}
\end{equation}
consisting of elements $x_i$ satisfying
\begin{equation}
 \prod_{i=1}^{d} x_i \leq \epsilon
\end{equation}
for some value of $\epsilon$. This subset meets all boundary strata of the cube transversely, and satisfies the following property:
\begin{lem}
  The intersection of $L_{d,0}$ with the facet $x_i = 1$ agrees with $L_{d-1,0}$. \qed
\end{lem}
A slightly more precise formulation of the above result is that the inclusion of the $(d-1)$-cube in the $d$-cube induces a diffeomorphism from $L_{d-1,0}$ to the boundary facet of $L_{d,0}$ given by its intersection with $x_i=1$. These maps are compatible with compositions in the sense that they determine a unique map
\begin{equation}
    L_{k,0} \to L_{d,0}
\end{equation}
associated to setting the coordinates labelled by any collection of $d-k$ elements of $\{0, \ldots, d\}$ equal to $1$. We conclude:
\begin{cor}
  The collection of manifolds $\{ L_{d,0} \}_{d=0}^{\infty}$ forms a semicosimplicial set in the category of smooth manifolds with corners. \qed
\end{cor}

We shall require the following elaboration of the above construction: let $L_{d,1}$ denote the subset of the $d+2$-cube with coordinates $(x_1, \ldots, x_d, y)$ satisfying the inequality
\begin{equation}
    (1-y)^2 + \frac{\prod_{i=1}^{d} x^2_i}{\epsilon^2} \leq 1.
\end{equation}
This is a manifold with corners, as can be inductively seen by computing that the tangent space of the hypersurface $   (1-y)^2 + \frac{\prod_{i=1}^{d} x^2_i}{\epsilon^2}   = 1 $ at $(x,y) = (0,0)$ is the hyperplane $y = 0$. Setting $y=1$ or $x_i = 1$ respectively gives rise to inclusions 
\begin{equation}
  L_{d,0} \to L_{d,1} \leftarrow L_{d-1,1}.
\end{equation}
Using $\mathbf{1}$ to denote the category $0 \to 1$, the functoriality of this construction can therefore be stated as follows:
\begin{lem}
  The assignment $(d,i) \mapsto L_{d,i}$ defines a functor from $\Delta_{+} \times  \mathbf{1}$. \qed 
\end{lem}

\begin{figure}
  \centering
  \begin{tikzpicture}
    \begin{scope}
       \node  at (0,1.7) {$L_{1,0}$}; 
      \draw[thin,gray] (0,1) -- (0,0) ;
      \draw[thick] (0,0) -- (0,1/2);
      \filldraw[black] (0,1/2) circle (2pt);
      \filldraw[black] (0,0) circle (2pt) node[anchor=north]{$0$};
    \end{scope}

    \begin{scope}[shift={(2,0)}]
           \node  at (0.5,1.7) {$L_{2,0}$};
      
      \draw[thin,gray] (0,0) -- (0,1) -- (1,1) -- (1,0) -- cycle ;
      \filldraw[black] (0,0) circle (2pt) node[anchor=north]{$(0,0)$};
      \filldraw[black] (0,1) circle (2pt) node[anchor=south]{$(0,1)$};
      \filldraw[black] (1,0) circle (2pt) node[anchor=north]{$(1,0)$};
      \filldraw[black] (1/2,1) circle (2pt); 
      \filldraw[black] (1,1/2) circle (2pt); 
 
      \draw[thick] (0,0) -- (0,1) -- (1/2,1) .. controls (1/2,1/2) .. (1,1/2) -- (1,0) -- cycle;
      
    \end{scope}
    
    \begin{scope}[shift={(5,0)}]
      \node  at (0.5,1.7) {$L_{3,0}$};
      
 \filldraw[black] (1,0,0) circle (2pt) node[anchor=west]{$(1,0,0)$};
 \filldraw[black] (0,1,0) circle (2pt) node[anchor=south]{$(0,1,0)$};
 \filldraw[black] (0,0,1) circle (2pt) node[anchor=north]{$(0,0,1)$};

 \filldraw[black] (1,1,0) circle (2pt);
 \filldraw[black] (0,1,1) circle (2pt);
 \filldraw[black] (1,0,1) circle (2pt);

 \filldraw[black] (1,1/2,1) circle (2pt);
 \filldraw[black] (1/2,1,1) circle (2pt);
 \filldraw[black] (1,1,1/2) circle (2pt);

 \filldraw[black] (0,0,0) circle (2pt);
 
   \draw[thick,dashed]  (0,0,0) -- (0,0,1);
   \draw[thick,dashed]  (0,0,0) -- (0,1,0);
   \draw[thick,dashed]  (0,0,0) -- (1,0,0);
   
   \draw[thin,gray]  (0,1,1) -- (1,1,1);
   \draw[thin,gray]  (1,1,0) -- (1,1,1);
   \draw[thin,gray]  (1,0,1) -- (1,1,1);

   \draw[thick] (0,0,1) -- (0,1,1) -- (1/2,1,1) .. controls (1/2,1/2,1) .. (1,1/2,1) -- (1,0,1) -- cycle;
   \draw[thick] (0,1,0) -- (0,1,1) -- (1/2,1,1) .. controls (1/2,1,1/2) .. (1,1,1/2) -- (1,1,0) -- cycle;
   \draw[thick] (1,0,0) -- (1,0,1) -- (1,1/2,1) .. controls (1,1/2,1/2) .. (1,1,1/2) -- (1,1,0) -- cycle;
 \end{scope}

 \begin{scope}[shift={(2,-3)}]
   \node  at (0.5,1.7) {$L_{1,1}$};
   
      \draw[thin,gray] (0,0) -- (0,1) -- (1,1) -- (1,0) -- cycle ;
      \filldraw[black] (0,1) circle (2pt); 
      \filldraw[black] (1/2,1) circle (2pt); 
      \filldraw[black] (0,0) circle (2pt);
      
      \draw[thick] (0,1) -- (0,0) .. controls (1/2,0) .. (1/2,1) -- cycle;
      
    \end{scope}

    \begin{scope}[shift={(5,-3)}]
      \node  at (0.5,1.7) {$L_{2,1}$};
      
      \draw[thin,gray] (0,1,0) -- (0,1,1) -- (1,1,1) -- (1,1,0) -- cycle ;
            \draw[thin,gray] (1,0,0) -- (1,0,1) -- (1,1,1) -- (1,1,0) -- cycle ;
   \draw[thin,gray]  (0,1,1) -- (1,1,1);
   \draw[thin,gray]  (1,1,0) -- (1,1,1);
   \draw[thin,gray]  (1,0,1) -- (1,1,1);

 
   
 \filldraw[black] (0,1,1) circle (2pt);
 \filldraw[black] (1,0,1) circle (2pt);
 \filldraw[black] (0,1,0) circle (2pt);
 \filldraw[black] (1,0,0) circle (2pt);
 \filldraw[black] (0,0,0) circle (2pt);
 \filldraw[black] (0,0,1) circle (2pt);
 \filldraw[black] (1,1/2,1) circle (2pt);
  \filldraw[black] (1/2,1,1) circle (2pt);

   \draw[thick] (0,0,1) -- (0,1,1) -- (1/2,1,1) .. controls (1/2,1/2,1) .. (1,1/2,1) -- (1,0,1) -- cycle;
   \draw[thick] (1,0,1) -- (1,0,0) .. controls (1,1/2,0) .. (1,1/2,1) -- cycle;
   \draw[thick] (0,1,1) -- (0,1,0) .. controls (1/2,1,0) .. (1/2,1,1) -- cycle;

   \draw[thick,dashed]  (1,0,0)--(0,0,0) -- (0,1,0);
   \draw[thick,dashed] (0,0,1)--(0,0,0);

 \end{scope}
    
  \end{tikzpicture}
  \caption{The corners and edges of the first five $L$-blocks.}
  \label{fig:L-blocks}
\end{figure}

\begin{rem}
While one can describe diffeomorphic models for the manifolds $L_{d,0}$ as polyhedra, by replacing the inequality $\prod x_i = \epsilon$ with $\epsilon \leq \sum (1- x_i) $, we do not know how to give a similar description for the manifolds $L_{d,1}$.
\end{rem}

\subsection{Horn-filling for orbifolds}
\label{sec:horn-fill-orbif}

In order to use the $L$-blocks from the previous section, we introduce
a covariant functor from $\vec{\cP}^{\Gamma}(p,r)$ to the product
$\Delta_{+} \times \mathbf{1}$. The functor to $\mathbf{1}$ is quite easy to define, and simply measures whether the tree associated to an object of $ \Lambda^n_k \vec{\cP}^{\Gamma}(p,r) $ has a label which arise from the set $\cP_k$ corresponding to the vertex of the horn, or an element of it. Any object corresponding to a tree which lacks such a label maps to $1$, and the remaining objects map to $0$. Geometrically, this exactly corresponds to assigning $1$ to those strata which correspond to the boundary of the horn, and $0$ to those associated to the interior.

In order to define the functor to $\Delta_{+} $ recall that the morphisms in
$\vec{\cP}^{\Gamma}(p,r)$ are given by
collapsing interior edges and adding elements to the sets labelling
the vertices. Assign to each object of $\vec{\cP}^{\Gamma}(p,r)$ the sum of the number of elements of $\{1, \ldots, n\} \setminus \{k\}$ which do not appear as labels of edges, with the number of elements of $(\cP_1, \ldots, \cP_n) $ that appear as label of edges. The ordering along the labelled arc that determines such an object (and the ordering of the sets $\{i+1, \ldots, j-1\}$) implies that this assignment is functorial.

\begin{rem}
 An alternate way of describing the functor to $\Delta_{+} \times \mathbf{1}$ is that it records all contributions to the codimension of an object which are not shared by objects of the missing boundary facet $\partial^k \vec{\cP}^{\Gamma}(p,r) $. 
\end{rem}

Considering the $L$-blocks as a covariant functor from
$\Delta_{+} \times \mathbf{1}$ to the category of smooth manifolds with corners,
and restricting to the horn, we thus obtain a covariant functor 
\begin{equation}
  L \co  \Lambda^n_k \vec{\cP}^{\Gamma}(p,r) \to d \Orb.
\end{equation}

Under Assumption \ref{assu:structure_maps_iso}, we are now ready to construct a horn-filling in the special case which we are considering. Our construction uses the weighted colimit, which is obtained from the disjoint union of the products of the orbispaces $  Y_{\alpha}(p, r)$ with the $L$-blocks associated to $\alpha$ under the relations which identifies, for each arrow $ \alpha \to \beta$ in $ \Lambda^n_k \vec{\cP}^{\Gamma}(p,r)$, the images of the maps
\begin{equation} \label{eq:weighted_colim_codim-1}
L_{\beta} \times  Y_{\beta}(p, r)    \leftarrow  L_{\alpha} \times  Y_{\beta}(p, r)  \to L_{\alpha} \times  Y_{\alpha}(p, r).
\end{equation}

\begin{lem} \label{lem:gluing_fixed_obstruction}
  If the maps in Equation \eqref{eq:structure-map-horn} are isomorphism, the weighted colimit
  \begin{equation}
    Y(p,r) \equiv \coprod_{\alpha \to \beta}    L_\alpha \times  Y_{\beta}(p, r) / \sim 
  \end{equation}
  is a smooth orbifold stratified by $\vec{\cP}^{\Gamma}(p,r)$, with smooth structure determined by a choice of collars on the $L$-blocks and on the orbifolds $Y_{\alpha}(p,r)$ for $\alpha$ in $\Lambda^n_k \vec{\cP}^{\Gamma}(p,r)$. Moreover, the stratum associated to $\alpha$ in $\Lambda^n_k \vec{\cP}^{\Gamma}(p,r)$ is diffeomorphic to $Y_{\alpha}(p,r)$.
\end{lem}
\begin{proof}
  To start, we note that the functor $\alpha \mapsto L_{\alpha}$ is constant on the subcategories of $ \Lambda^n_k \vec{\cP}^{\Gamma}(p,r)$ consisting of arcs whose labels from the sets $(\cP_2, \dots, \cP_{n-1})$ are fixed. This implies that the local structure of $Y(p,r) $ near a point of $ L_\alpha \times  Y_{\beta}(p, r)  $ depends only on the minimal element of each such subcategory, in which the only edge labelled by an element of $\cP_1$ is the incoming edge, and the only edge labelled by an element of $\cP_n$ is the outgoing edge. So we assume that $\beta$ (and hence $\alpha$) both satisfy this property.

  Next, we explain the type of boundary stratum of $  Y(p,r)$ to which each boundary stratum of the $L$-blocks corresponds: if $\alpha$ maps to $(d,0)$, then $L_{\alpha}$ has $2 d +1 $ boundary strata, with $d$ cubes obtained by setting $x_i = 0$ that have a prescribed bijective correspondence with the codimension $1$ objects of $ \vec{\cP}^{\Gamma}(p,r) $ containing $\alpha$, and  will be part of the corresponding boundary strata. There is as well one boundary stratum given by the hypersurface $\prod x_i = \epsilon$, which will be part of the boundary stratum associated to the missing facet of the horn. The remaining facets of $L_{d,0}$, which are $d$ copies of $L_{d-1,0}$ obtained by setting $x_i = 1$, are glued to other facets by the weighted colimit, and thus lie in the interior of $Y(p,r)$.  On the other hand, if $\alpha$ maps to $(d,1)$, then with the exception of the case $d=1$, $L_{\alpha}$ has $2 d + 2$ boundary facets, with $d$ cubes obtained by setting $x_i=0$ that again are part of boundary strata labelled by corresponding objects of $\Lambda^n_k \vec{\cP}^{\Gamma}(p,r) $,  and one boundary facet which is the hypersurface $  (1-y)^2 + \frac{\prod_{i=1}^{d} x^2_i}{\epsilon^2}   = 1 $ which as before contributes to the missing top stratum of the horn. The remaining facets, which consist of  $d$ copies of $L_{d-1,1}$ obtained by setting $x_i=1$ and one copy of $L_{d,0}$ obtained by setting $y=1$, are glued by the weighted colimit, and thus lie in the interior.  In the case $d=1$, we have one fewer boundary stratum since $L_{0,1} $ is empty.

  In order to prove that $Y(p,r) $ is a manifold (with boundary), it remains to analyse the local structure near each point of $ L_\alpha \times  Y_{\beta}(p, r) $, for a given morphism $\alpha \to \beta$. By the preceding discussion, the key point is to describe the structure at a point where multiple gluings occur, i.e. where multiple coordinates of the cube equal $1$.  Having restricted attention at the beginning of the proof to the case where the number of such coordinates on $\alpha$ agrees with its codimension, we have a distinguished identification of the category of factorizations $\alpha \to \gamma \to \beta $ with the product $\{ 0 \to 1 \}^{k} $, where $k$ is the difference between the codimension of $\alpha$ and $\beta$. This implies that a neighbourhood of $  L_\alpha \times  Y_{\beta}(p, r) $  is stratified by $\{ -1 \leftarrow 0 \to 1 \}^k$, with top strata given by $ L_\gamma \times  Y_{\gamma}(p, r)$  for $\gamma$ lying under $\alpha$ and over $\beta$. A choice of collars thus identifies this neighbourhood with the product of $  L_\alpha \times  Y_{\beta}(p, r)$ with
  \begin{equation}
\left(    (-1,0] \cup [0,1) \right)^k ,
  \end{equation}
  so that we inherit a smooth structure  from the identification $ (-1,0] \cup [0,1) \cong (-1,1)$.

Having described the gluing, we see that the stratification  by $\vec{\cP}^{\Gamma}(p,r)$ is explicitly given
as follows: the stratum associated to the tree with a unique vertex 
labelled by $\{0, \ldots, n\}$ is the interior, and the one associated
to the tree with a unique vertex labelled by the complement of $k$ is
the union (over $\alpha$) of the product of $ Y_{\alpha}(p, r)$ with
the facet of the $L$-block which lies in the interior of the cube
(i.e., given by $\prod_{i=0}^{d} x_i = \epsilon$ or  $  (1-y)^2 + \frac{\prod_{i=1}^{d} x^2_i}{\epsilon^2}   = 1 $). The other facets
are each indexed by a codimension-$0$ stratum in the horn
$Y_{\alpha}(p, r)$, and are given by the collared completion of this
component, since they are given by its union with products of its boundary strata with cubes.
\end{proof}

Note that the choice of identification of the boundary strata with the orbifolds $Y_{\alpha}(p, r)$ depends only on a choice of diffeomorphism
\begin{equation} \label{eq:diffeomorphism_intervals}
  [-1,1) \to [0,1)  
\end{equation}
which is the identity near $1$. Fixing such a choice and using a
coherent choice of collars on $Y_{\alpha}(p, r)$ (which we can always
do) gives diffeomorphisms with the property that, for each arrow
$\alpha \to \beta$, we obtain a commutative diagram
\begin{equation}
  \begin{tikzcd}
    Y_{\beta}(p, r) \ar[r] \ar[d] &  \partial^\beta Y(p,r) \ar[d] \\
    Y_{\alpha}(p, r) \ar[r] & \partial^\alpha Y(p,r).
  \end{tikzcd}
\end{equation}

Given these data, we note that the structure maps of the flow horn induce maps
\begin{equation}
    Y(p,q) \times  Y(q,r) \to  \partial^{q} Y(p,r),  
\end{equation}
and the requirement that collars be consistent with the structure maps of the horn (i.e., Equation \eqref{eq:structure-map-horn}) ensure that these maps are associative. We conclude:
\begin{lem}
Under Assumption \ref{assu:structure_maps_iso}, the $k$-horn formed by the orbifolds $Y_\sigma$ admits a filling to a flow simplex. \qed
\end{lem}

By construction, the vector bundles $T^- \bY_\alpha(p, r)$ glue to a vector bundle on $Y(p,r)$, as do the sections, yielding a derived orbifold which we denote $ \bY(p,r)$. The fact that no choices are required in the gluing of these vector bundles implies:
\begin{cor}
Under Assumption \ref{assu:structure_maps_iso}, the $k$-horn formed by the derived orbifolds $\bY_\sigma$ admits a filling to a flow simplex. \qed
\end{cor}

\subsection{Stabilization}
\label{sec:stabilization}

The key idea required to construct a horn-filler for diagrams which do
not satisfy Assumption \ref{assu:structure_maps_iso} is that one can
achieve the desired isomorphism property for a given input and output
by an appropriate choice of vector bundles over all the orbifolds
indexed by elements of $\Lambda^n_k \vec{\cP}^{\Gamma}(p,r)$. To state
the result precisely, it is convenient to introduce a new notion: 

\begin{defin}
The category of derived orbifolds with \emph{obstruction bundles of constant rank} is the 
subcategory of $d \Orb$ consisting of all objects and of those
  morphisms that satisfy the property that the strong
  equivalence in
  Equation~\eqref{eq:strong_equivalence_to_boundary_stratum} is an isomorphism.
\end{defin}

The proof of the following result is postponed to Section \ref{sec:enough-orbibundles}:

\begin{prop} \label{prop:horn_filling}
Given a flow $k$-horn $\bY$, and a pair of objects $(p,r)$ of $\cP_0$ and $\cP_n$, there exists a functor
\begin{equation} \label{eq:bimodules_extension}
\tilde{\bY}_{(\_)}(p, r) \co   \Lambda^n_k \vec{\cP}^{\Gamma}(p,r) \to d\Orb,
\end{equation}
factoring through derived orbifolds with obstruction bundles of
constant rank,  and which is equipped with a natural equivalence from
$\bY_{(\_)}(p, r)$. 

  In addition, we may choose the collection $\tilde{\bY}$ of derived orbifolds $  \tilde{\bY}_{(\_)}(p, r) $ so that they form bimodules over the flow categories with objects the elements of $\cP_0$ and $\cP_n$, and so that the natural transformation from $\bY$ is a map of bimodules  to $\tilde{\bY}$.
\end{prop}

We can now provide the proof of the main result of this section:
\begin{proof}[Proof of Theorem \ref{thm:horn_filling}]
  We associate to each pair $(p,r)$ the derived orbifold $\bY(p,r)$ with underlying orbifold
  \begin{equation}
      Y(p,r) \equiv     \coprod_{\alpha \to \beta}    L_\alpha \times  \tilde{Y}_{\beta}(p, r) / \sim,
      \end{equation}
      as in Lemma~\ref{lem:gluing_fixed_obstruction}, and whose associated vector bundles and sections are obtained by gluing pullback from $\tilde{Y}_{\beta}(p, r) $. This derived orbifold is stratified by $\vec{\cP}^{\Gamma}(p,r) $, so that the strata associated to elements of  $\Lambda^n_k \vec{\cP}^{\Gamma}(p,r)$ are equipped with distinguished diffeomorphisms to $\tilde{Y}_{\beta}(p, r) $.

      For pairs $(p,q)$ of objects of $\cP_a$ and $\cP_b$, with either $a \neq 0$ or $b \neq n$, we set
      \begin{equation}
            \bY(p,q) \equiv    \bY_{[a,b]}(p,q),  
      \end{equation}
where $[a,b]$ denotes the simplex spanned by all vertices between $a$ and $b$. The strong equivalences
      \begin{equation}
            \bY(p,q) \times  \bY(q,r) \to  \partial^{q} Y(p,r),          
          \end{equation}
          are thus determined by the initial horn whenever   either $a \neq 0$ or $b \neq n$; in the remaining case, it is induced by
          the natural equivalence from $\bY$ to $\tilde{\bY}$, together with the choice of diffeomorphism in Equation \eqref {eq:diffeomorphism_intervals}.
\end{proof}

\subsection{Filling structured horns}
\label{sec:fill-struct-horns}

Assume that the initial data we fixed at the beginning of Section \ref{sec:kan-condition-flow} consists of structured flow simplices $\bY_{\sigma} $ and that the equivalences in Equation \eqref{eq:equivalence_restriction_stratum_horn} are structured. This implies in particular that all the underlying derived orbifolds $  \bY_{\alpha}(p,r)$ for each element $\alpha$ of $  \Lambda^n_k \vec{\cP}^{\Gamma}(p,r)$ are equipped with virtual structured vector bundles $I_{\alpha}(p,r)$, vector bundles $W_{\sigma}(p,r)$, and virtual vector spaces $U_{\alpha}(p,r)$. The latter are prescribed by Definition \ref{def:framed_flow_simplex}, in the sense that $U^+_{\alpha}$ is   $ \bR^{\{1, \ldots, n\}}$ and $U^{-}_{\alpha}$ is freely generated by the sequence of objects that appear as labels of edges of the tree corresponding to $\alpha$, with the exception of $p$.

To begin, let us impose Assumption \ref{assu:structure_maps_iso}, which we now interpret as the requirement that the structure maps are isomorphisms of structured derived orbifolds. This implies in particular that we have isomorphisms of virtual vector bundles
\begin{align}
  I_{\tau}(p,r)  \cong &  I_{\sigma}(p,r)|_{\partial^{\tau} Y_{\sigma}(p,r)} \\
    I_{\sigma}(p,q) \times  I_{\sigma}(q,r) \cong &  I_{\sigma}(p,r) |_{\partial^{q} Y_{\sigma}(p,r)}
\end{align}
which respect compositions, and are compatible with the complex structure. Using the construction of Section \ref{sec:horn-fill-orbif}, this yields a complex vector bundle
\begin{equation}
  I^{\pm}(p,r) \to  Y(p,q) 
\end{equation}
given on $L_\alpha \times  Y_{\beta}(p, r) $ by the pullback of the virtual vector bundle $ I_{\beta}(p,r) $, which is the product of the virtual vector bundles $I_{\tau}(p,r)$ associated to the stratum $\beta$.

Similarly, the vector bundles $W_{\alpha}(p,r)$ assemble, under the assumption that the structure maps are isomorphisms, to a vector bundle
\begin{equation}
  W(p,r) \to   Y(p,q) 
\end{equation}
whose restriction to $ L_\alpha \times  Y_{\beta}(p, r) $ is equipped with a canonical isomorphism with the pullback $W_{\beta}(p,r) $.

We now claim that $\bY$ naturally lifts to a structured derived orbifold: on each codimension $0$ subset given by $L_\alpha \times  Y_{\alpha}(p, r)$, we use the identification
\begin{equation}
T L_{\alpha} \cong \bR^{\{q_1, \ldots, q_d\}},
\end{equation}
where $q_i$ are the orbits that appear as labels in $\alpha$ to obtain an isomorphism
\begin{align} \label{eq:structure_isomorphism_alpha_block}
 T L_{\alpha} \oplus T  Y_{\alpha}(p, r) \oplus \bR^{r} \oplus W_{\alpha}(p,r)\oplus I^-_{\alpha} &   \cong  T  Y_{\alpha}(p, r) \oplus \bR^{\{q_1, \ldots, q_d, r\}} \oplus W_{\alpha}(p,r) \oplus I^-_{\alpha} \\
  & \cong \bR^{\{1, \ldots, n\} } \oplus T^{-}  Y_{\alpha}(p, r) \oplus W_{\alpha}(p,r) \oplus I^+_{\alpha}.  
  \end{align}
  The compatibility of tangential structures for the derived orbifolds $ \bY_{\alpha}(p,r)$ yields the following result, where we set $U_{\alpha}(p,r) = (  \bR^{\{1, \ldots, n\} }, \bR^{r} )$:
  \begin{lem}
    The isomorphisms of Equation \eqref{eq:structure_isomorphism_alpha_block} assemble to a lift of $\bY$ to a structured derived orbifold, together with lifts of the maps
    \begin{equation}
            \bY_{\alpha}(p,r) \to \bY(p,r)
          \end{equation}
          for all corner strata. These maps are compatible with arrows $\alpha \to \beta$, and with the bimodule actions of the flow categories $\bX_0$ and $\bX_n$. \qed
  \end{lem}
  \begin{cor}
   Under Assumption \ref{assu:structure_maps_iso}, the flow simplex $\bY$ lifts to a structured flow simplex. 
  \end{cor}
  \begin{proof}
    The structure maps constructed in Equation
    \eqref{eq:structure_isomorphism_alpha_block} satisfy all the
    desired properties, except that they do not restrict to a framing
    of the stratum $\partial^k  \bY_{\alpha}(p,r) $. This last
    condition can be ensured by enlarging the vector bundle $W_{\bY}$,
    inductively applying a relative version of Lemma~\ref{lem:inherit_framings}. 
  \end{proof}

It remains, as in the unstructured case, to drop Assumption
\ref{assu:structure_maps_iso}. This is an immediate consequence of
Lemma~\ref{lem:canonical_lift_structure_equivalence}: 

\begin{prop}
Given a structured flow $k$-horn, the filling of the underlying flow
$k$-horn lifts to a structured horn-filling. \qed 
\end{prop}

\subsection{The proof of Proposition  \ref{prop:horn_filling}}
\label{sec:enough-orbibundles}

We begin with some preliminary results:

\begin{lem} \label{lem:large_enough_for_arbitrary_embedding}
If $Y \co \cD \to \Orb$ is a finite diagam in the category of
orbifolds, and $X$ is a orbispace under $Y$ with the property that the
map $Y_\alpha \to X$ is representable for each object $\alpha$ of $\cD$, then
for each assignment of a vector bundle $U_\alpha$  on $Y_\alpha$,
there exists a vector bundle $V$ on $X$ with the property that
$U_\alpha$ embeds in the pullback of $X$ for each $\alpha$. 
\end{lem}

\begin{proof}
This is an immediately consequence of Pardon's
result~\cite{Pardon2019} on the existence of enough vector bundles, as
we can take the direct sum of vector bundles $V_\alpha$ on $X$ whose
pullback to $Y_\alpha$ admit an embedding of $U_\alpha$. 
\end{proof}

\begin{lem}\label{lem:stably_contractible_choice_embedding}
If $f \co Y \to X$ is a representable map of orbifolds, and $U$ is a
vector bundle on $Y$, then there is a vector bundle $V$ on $X$ with
the property that the connectivity of the space of embeddings of $U$ in
$f^* V$ is arbitrarily high. Moreover this connectivity function
increases upon enlarging $V$, and passing to subbundles of $U$. 
\end{lem}

\begin{proof}
Presenting $Y$ and $X$ as $G$ quotients for some compact Lie group
$G$, this follows from standard results on the connectivity of
equivariant Stieffel manifolds.
\end{proof}

\begin{lem} \label{lem:extend_deformation}
  If $f \co Y \to X$ is a representable map of orbifolds, and $\Phi \co V_0 \to V_1$ is an embedding of vector bundles on $X$, then every family of embeddings starting at the restriction of $\Phi$ to $Y$ extends to a family of embeddings in $X$.
\end{lem}
\begin{proof}
  Presenting $Y$ and $X$ as $G$ quotients for some compact Lie group $G$, this follows from the fact that the corresponding map is a $G$-cofibration, hence satisfies the homotopy extension property.
\end{proof}

We now prove the desired result:

\begin{proof}[Proof of Proposition \ref{prop:horn_filling}]
  We proceed by induction on the value of the action. In the inductive step, we thus assume that the choice of the functor $\tilde{Y}^{\lambda'}_{(\_)}(p_0,p_n)$, and of the natural transformation from $Y^{\lambda'}_{(\_)}(p_0,p_n)$ have been fixed for all $\lambda'$ of action smaller than a given element $\lambda$ of $\Gamma$, and that the bimodule structure maps have been fixed up to this energy level.

  Next, we consider the colimit
  \begin{equation}
 Y^{\lambda}(p_0, p_n) \equiv \colim_{\alpha \in \Lambda^n_k \vec{\cP}^{\lambda}(p_0,p_n)}    Y_{\alpha}(p_0, p_n).
  \end{equation}
  The map $Y_{\alpha}(p_0, p_n) \to   Y^{\lambda}(p_0, p_n)$ is representable, so we may apply Lemma \ref{lem:large_enough_for_arbitrary_embedding} to obtain a vector bundle $V^{\lambda}(p_0,p_n)$, which admits an embedding of the colimit $U_{\alpha}$ of the restrictions to $Y_{\alpha}(p_0, p_n) $  of  (i) the obstruction bundle $T^+ \tilde{Y}_{\beta}(p_0, p_n) $ for each arrow $\alpha \to \beta$ in the image of the bimodule maps and (ii) the obstruction bundle $T^+ Y_{\beta}(p_0, p_n) $ for all other arrows. The colimit is taken with respect to the maps
  \begin{equation}
    \begin{tikzcd}
      T^+ Y_{\beta}(p_0, p_n) |_{\alpha} \ar[r] \ar[d] & T^+ Y_{\gamma}(p_0, p_n)|_{\alpha} \ar[d]   \\
       T^+ \tilde{Y}_{\beta}(p_0, p_n)|_{\alpha} \ar[r] & T^+ \tilde{Y}_{\gamma}(p_0, p_n)|_{\alpha} ,
    \end{tikzcd}   
  \end{equation}
where the bottom row is only defined if the arrow is in the image of
the bimodule structure maps.  Moreover, by Lemma
\ref{lem:stably_contractible_choice_embedding}, we may assume that the
connectivity of the space of embeddings of $U_\alpha$ in the vector bundle we construct is
arbitrarily large; 
we shall require it to be larger than the codimension of any non empty
stratum of $Y^{\lambda}(p_0, p_n)$. 

There is no assumption so far that the embeddings into $
V^{\lambda}(p_0,p_n)$ that we construct on each stratum are compatible
with restriction. We next proceed by decreasing induction on the
codimension of these strata to achieve this, i.e., to produce
embeddings 
  \begin{equation}
       U_{\alpha} \to V^{\lambda}(p_0,p_n)|_{\alpha}
  \end{equation}
for each object $\alpha$ of $ \Lambda^n_k \vec{\cP}^{\lambda}(p_0,p_n)$ so that the diagram
  \begin{equation}
    \begin{tikzcd}
 U_{\beta}|_{\alpha}  \ar[r]  \ar[dr] & U_{\alpha} \ar[d]  \\
&       V^{\lambda}(p_0,p_n)|_{\alpha} 
    \end{tikzcd}
  \end{equation}
  commutes for each map $\alpha \to \beta$.  In the inductive step, given a stratum $\delta$ of codimension $k$,
  the inductive hypothesis provides  us an embedding
  \begin{equation}
       U_{\delta}|_{\partial \delta} \to V^{\lambda}(p_0,p_n)|_{\partial \delta}
      \end{equation} and Lemma \ref{lem:extend_deformation} allows us
  to deform the embedding to the entire stratum, which completes the
  induction step, and hence the proof.
\end{proof}

\section{The simplicial structure}
\label{sec:quasi-units}

The purpose of this section is to prove the following result:

\begin{prop}
The semisimplicial set $\Flow^{\cS}$ of structured flow categories
admits the structure of a simplicial set.
\end{prop}

Our construction of the simplicial structure proceeds in two steps.
First, we give a geometric construction of the initial and terminal
degeneracies $s_0, s_n \colon \Flow_n^{\cS} \to \Flow_{n+1}^{\cS}$.
In particular, this yields an identity element in the endomorphism
space of each object. The remaining degeneracies are then constructed
by adapting an abstract result of Steimle~\cite{Steimle}, as discussed
in Section~\ref{sec:quasi-semi-categ}.

The geometric part of the construction arises from the following observation: there is a natural bimodule associated to each Morse-theoretic flow
category which assigns to each pair of critical points the moduli
space of gradient flow lines with one interior marked point. This
bimodule plays the r\^ole of the diagonal bimodule and represents the
identity map of flow categories.  From the perspective of the
semisimplicial set of flow categories, this corresponds to a choice
of a degeneracy map that will form a part of the quasicategory
structure.  Our first task will be to construct these
bimodules in the abstract setting.

The starting point of the construction is to associate to each flow
category $\bX$ a bimodule $\bX \times [0,1]$ enriched in derived
orbifolds, given by assigning to each pair of objects the product with
an interval of the derived orbifold of morphisms between them with the
following exception: for equal objects, we add a copy of the unit derived orbifold.
This does not define an edge in $\Flow$, i.e., a
flow bimodule, for the simple reason that the images of the following
two bimodule structure maps agree:
\begin{equation}
    \bX(p,q) \times \left( \bX(q,r) \times [0,1] \right) \to \bX(q,r) \times [0,1] \leftarrow  \left(\bX(p,q) \times [0,1] \right) \times \bX(q,r).
\end{equation}
This violates the requirement that the corner strata be enumerated by the objects of the category defined in Section \ref{sec:strat-categ-assoc}.  We shall resolve this problem by introducing a more refined manifold with corner structure on the topological manifold $X(p,q) \times [0,1]$, corresponding to degenerating the interval factor over the higher codimension strata of $ X(p,q)$.

\subsection{Conic degenerations}
\label{sec:conic-degenerations}

Let $\bT \bP^1$ denote the quotient of $[0,+\infty)^{2} \setminus \{( 0 , 0)\}$ by dilation with respect to positive real numbers $\lambda$:
\begin{equation}
  \bT \bP^1 \equiv \frac{\{ (x, y ) | 0 \leq x, y, \quad (x,y) \neq (0,0) \}}{ (x, y) \sim (\lambda \cdot x, \lambda \cdot y)}.
\end{equation}
Note that this smooth manifold with corners is diffeomorphic to the interval, which we equip with the fixed orientation determined by positive rescaling of $x$. The distinguished coordinates (in the complement of a point) obtained from the slices $x=1$ or $y=1$ will be convenient for our setting, because we have a natural family of embeddings
\begin{equation}
   \bT \bP^1 \subset  \bT \bP^1 \times  \bT \bP^1,
\end{equation}
parametrised by $t \in (0,\infty)$, which is given as the closure of the solution set of the equation $x_0 \cdot y_1  = t $, and which degenerates, at $t=0$ to a union of two copies of $\bT \bP^1$, meeting along the origin $x_0 = y_1 = 0$. The total space of this family is a smooth submanifold with corners of the product of $[0,\infty)$ with $ \bT \bP^1 \times  \bT \bP^1$. More generally, 
\begin{lem} \label{lem:conic_bundle_smooth}
  The closure $\cD_{n+1}$ of the solution set  in $[0,\infty)^n \times  (\bT \bP^1)^{n+1}$ of the equations
  \begin{equation} \label{eq:equation_conic}
        x_{i-1} \cdot y_i = t_i \quad 1 \leq i \leq n
      \end{equation}
      is a smooth submanifold with corners of dimension $n+1$.      Its fibre over a point $(t_1, \ldots, t_n) \in [0,\infty)^n$ is a union of intervals, indexed by one more than the number of coordinates $t_i$ which vanish. Moreover, for each subset $I$ of $ \{1, \ldots, n\}$  the inclusion of the open subset given by setting $\{t_{i} \neq 0\}_{i \in I} $ canonically lifts to a commutative pullback diagram:
      \begin{equation} \label{eq:compatible_degenerations}
        \begin{tikzcd}
          \cD_{n - |I|+1} \times (0,\infty)^{|I|} \ar[r] \ar[d] &    \cD_{n+1} \ar[d] \\
          {[}0,\infty)^{n-|I|} \times (0,\infty)^{|I|}  \ar[r] &  {[}0,\infty)^n.
        \end{tikzcd}
      \end{equation}
      \qed
\end{lem}

We refer to $\cD_{n+1} $ as the \emph{total space of the conic bundle over $[0,\infty)^n$ with discriminant along the boundary}. Note that this definition makes sense starting with $n=0$, in which case we have $ \cD_{n+1} \cong \bT \bP^1$. We separately define $\cD_{0}$ to be a point.

A useful property of this construction is that, up to canonical diffeomorphism, it does not change when the ordering of the coordinate is reversed, as can be seen by swapping $x$ and $y$ and replacing $i$ by $n-i$ in Equation \eqref{eq:equation_conic}.  The identification of each component of the fibre with $ \bT \bP^1$ is given by projection to some of the factors: explicitly, if $\{ j_1, \ldots, j_d\}$ label the vanishing coordinates of a point in $[0,\infty)^n  $, then, setting $j_0 = 0$, projecting the fibre of $\cD_{n+1} $ over this point  onto any factor of $  (\bT \bP^1)^{n+1}$ labelled by an element $i \in \{j_{r}, \ldots, j_{r+1}-1\}$ is a diffeomorphism on the $r$\th   interval, and is constant on all others.

We now formulate the compatibility of this construction with the inclusion
\begin{equation}
  [0,\infty)^{n} \times [0,\infty)^{m} \to [0,\infty)^{n+m+1}  
\end{equation}
associated to setting $t_{n+1}=0$.
\begin{lem} \label{lem:degeneration_map_boundary_strata_model}
  The fibre of $\cD_{n+m+2}$ over the boundary stratum of $ [0,\infty)^{n+m+1}$ given by $t_{n+1}=0$ is the union of two codimension $1$ boundary strata, which are the images of embeddings
  \begin{equation} \label{eq:boundary_strata_Dn-over-corner}
\cD_{n+1} \times [0,\infty)^{m}  \to  \cD_{n+m+2} \leftarrow [0,\infty)^n \times \cD_{m+1}.
\end{equation}
These embeddings are natural in the sense that the following diagram, associated to setting $t_{n+1} = t_{n+m+2} = 0$, commutes:
\begin{equation}
  \begin{tikzcd}
 {[}0,\infty)^{\ell} \times \cD_{m+1} \times {[}0,\infty)^{\ell} \ar[r] \ar[d] & {[}0,\infty)^n \times    \cD_{m+\ell+1} \ar[d]  \\
    \cD_{n+m+1} \times {[}0,\infty)^{\ell}  \ar[r] & \cD_{n+m+2}.
  \end{tikzcd}
\end{equation} \qed
\end{lem}
We note that Equation \eqref{eq:boundary_strata_Dn-over-corner} labels all the boundary strata if we include as well the maps
  \begin{equation}
\cD_{0} \times [0,\infty)^{n+m+1}  \to  \cD_{n+m+2} \leftarrow [0,\infty)^{n+m+1} \times \cD_{0}
\end{equation}
associated to the boundary of the interval.

If $X$ is an orbifold with corners, with a partial ordering on the codimension $1$ boundary strata whose restriction to a neighbourhood of any corner stratum is a complete order, then the compatiblity condition of Diagram \eqref{eq:compatible_degenerations} implies that a choice of collars determines, by pullback from the local projection maps to $[0,\infty)^k$ a manifold with corners $\cD_{X}$, equipped with a projection map
\begin{equation} \label{eq:projection_degeneration_man_corner_base}
  \cD X  \to X.  
\end{equation}

This construction is compatible with immersions $X \to Y$ which are transverse to the corner strata in the sense that these induces a transverse pullback diagram
\begin{equation} \label{eq:degeneration_compatible_submanifold}
   \begin{tikzcd}
     \cD X  \ar[r] \ar[d] & \cD Y \ar[d] \\
     X \ar[r] & Y,
  \end{tikzcd}
\end{equation}
whenever the collars on the corner strata of $X$ are obtained by pullback from $Y$.
\begin{defin}
 The \emph{conic degeneration with discriminant $\partial \bX$} over a derived orbifold with corners $\bX$, with a partial order on the codimension $1$ boundary strata forming a complete order near each corner stratum,  is the derived orbifold $\cD \bX$ with underlying orbifold $\cD X$,  whose obstruction bundle and section are given by the pullback of $T^- \bX $ under Equation \eqref{eq:projection_degeneration_man_corner_base}.
\end{defin}
The compatibility of conic degenerations with immersions, i.e., Diagram~\eqref{eq:degeneration_compatible_submanifold}, implies:
\begin{lem} \label{lem:strong_equivalence_conic}
  A strong equivalence $\bX \to \bY$ functorially induces a strong equivalence $ \cD \bX \to \cD \bY$. \qed
\end{lem}

\begin{rem}
In this construction, the discriminant locus of this conic degeneration is the entire boundary of $X$. One can more generally introduce a degeneration along only a subset of this boundary, by pulling from the corresponding collar direction; we will specify this more general data when required.  
\end{rem}

\subsection{The diagonal bimodule}
\label{sec:diagonal-bimodule}

There is a natural partial order on the set of codimension $1$ boundary strata of the derived orbifolds underlying each morphism space $\bX^{\lambda}(p,q)$ in a flow category $\bX$, given by the induced decomposition $\lambda = \lambda_1 + \lambda_2$ into two elements of the monoid $\Gamma$, and the image of $\lambda_1$ under the action map to $[0,\infty)$. Note that distinguising $\lambda_2$ instead would simply reverse the ordering, which, as discussed following Lemma \ref{lem:conic_bundle_smooth}, would not change the construction of the conic degeneration up to canonical isomorphism. 

Using this order, we associate to each pair of objects $(p,q)$ of $\bX$, and each element $\lambda$ of $\Gamma$ of strictly positive action, the derived orbifold
\begin{equation}
   s \bX^{\lambda}(p,q) \equiv \cD \bX^{\lambda}(p,q).
\end{equation}
In addition, for each object $p$ of $\bX$, we set the derived orbifold $s\bX^0(p,p)$ associated to the unit of $\Gamma$ to be a point. We write $s\bX(p,q)$ for the resulting $\Gamma$-graded derived orbifold. The next result follows from Lemma \ref{lem:degeneration_map_boundary_strata_model}:

\begin{lem}
  There are natural maps of derived orbifolds
  \begin{equation} \label{eq:bimodule_structure_maps}
    \bX(p,q) \times     s \bX(q,r) \to s \bX (p,r) \leftarrow  s \bX(p,q) \times     \bX(q,r)
  \end{equation}
  which are the structure maps of a bimodule over $\bX$. \qed
\end{lem}

The resulting bimodule is called the \emph{diagonal bimodule}, and is denoted  $s \bX$.
\begin{lem} \label{lem:diagonal_bimodule_is_morphism}
The diagonal bimodule defines an morphism in $\Flow$ with source and target $\bX$.
\end{lem}
\begin{proof}
  The key point is that $s \bX(p,q)$ is stratified by the category associated to the pair $(\cP, \cP)$ of sets. The codimension $1$ strata are the images of the two maps in Equation \eqref{eq:bimodule_structure_maps}, and this implies, by induction on the codimension of a stratum, that $s \bX $ has the desired stratification.
\end{proof}

In the next subsection, we show that the morphism represented by the
diagonal bimodule in $\Flow$ is an equivalence; this is part of the
verification that the semisimplicial structure on $\Flow$ extends to a quasicategory.

\subsection{Initial and terminal degeneracies}
\label{sec:fill-init-term}

The diagonal bimodule can be interpreted as a map
\begin{equation}
 s \co \Flow_0 \to \Flow_1.
\end{equation}

In this section, we show that $s$ extends to a system of degeneracies
for $\Flow$, so that $\Flow$ admits the structure of a quasicategory.
In particular, this shows that the diagonal bimodule realizes an
idempotent self-equivalence for each flow category.

We now to extend this map to the higher
simplices of $\Flow$, by defining maps
\begin{align}
  s_0 \co \Flow_{n-1} & \to \Flow_{n} \\
   s_n \co \Flow_{n-1} & \to \Flow_{n}
\end{align}
which are compatibles with the face maps, and which are the initial and terminal degeneracies of a simplicial set. The remaining degeneracies will be constructed in the next section. Since the arguments in the cases $s_0$ and $s_n$ are entirely analogous, we only discuss the initial degeneracy.

We thus start with an elementary $n$-simplex $\bX$, on vertices which we denote $\{1, \ldots, n\}$, in order not to have to subsequently change the notation. The basic data of such a simplex consists of derived orbifolds $\bX(p_j, q_k)$ associated to each pair $(p_j,q_k)$ of objects of $\bX_j$ and $\bX_k$ for $j \leq k$. We shall be particularly interested in the case $j=1$, and in the codimension-$1$ boundary strata associated to the images of the left bimodule structure maps
\begin{equation} \label{eq:left_bimodule_map}
  \bX^{\lambda_1}(p_1, p'_{1}) \times  \bX^{\lambda_2}(p'_1, q_{k}) \to  \bX^{\lambda_1 + \lambda_2}(p_1, q_{k}).
\end{equation}

We write $s_0 \vec{\cP}$ for the sequence of sets obtained by repeating the first element, and assign to each pair $(p_j,q_k) $ of elements of $\cP_j$ and $\cP_k$ the  derived orbifold $s_0 \bX(p_j, q_k) $ given by   \begin{enumerate}
  \item $\bX(p_j, q_k)$ if $j \neq 0$,
  \item $\bX(p_j, q_k)$ if $j = k= 0$, where we use the fact that an element of $\cP_0$ corresponds to an element of $\cP_1$, or
  \item the conic degeneration
    \begin{equation}
      \cD_{0} \bX(p_j,q_k) \to    \bX(p_j,q_k)
    \end{equation}
     with discriminant locus the images of the left bimodule structure maps (Equation \eqref{eq:left_bimodule_map}) if $j=0$ and $k \neq 0$, ordered by the value of $\lambda_1$.
  \end{enumerate}
Lifting Equation \eqref{eq:left_bimodule_map} to the total space of the degeneration using Lemma \ref{lem:degeneration_map_boundary_strata_model}, we obtain composition maps for $p_0$ and $p'_0$ in $\cP_0$, and $q_k$ in $\cP_k$
  \begin{equation}
    \begin{tikzcd}
      s_0 \bX(p_0, p'_0)  \times s_0 \bX(p'_0, q_k) \ar[d] &  s_0 \bX(p_0, q_k) \\
      \bX(p_0, p'_0)   \times \cD_0  \bX(p'_0, q_k) \ar[r] & \cD_0 (p_0, q_k) \ar[u],
    \end{tikzcd}
  \end{equation}
  as well as composition maps for $p_1 \in \cP_1$:
    \begin{equation}
    \begin{tikzcd}
      s_0 \bX(p_0, p_1)  \times s_0 \bX(p_1, q_k) \ar[d] &  s_0 \bX(p_0, q_k) \\
     \cD_0 \bX(p_0, p_1)   \times   \bX(p_1, q_k) \ar[r] & \cD_0 (p_0, q_k) \ar[u].
    \end{tikzcd}
  \end{equation}
The same analysis as in Lemma \ref{lem:diagonal_bimodule_is_morphism} shows that these data determine an elementary $n$-simplex, which we denote $ s_0 \bX$, whose remaining composition maps, which do not involve elements of $\cP_0$, are given by the composition maps of $\bX$. We note the following consequences of the construction:

\begin{defin} \label{def:s0-degeneracy}
  The \emph{initially degenerate $n$-simplex} $s_0 \bX $ associated to an $n-1$-simplex $\bX$ is the simplex with $\cP_0 = \cP_1$, which assigns to a stratum $\sigma$ of $\Delta^n$
  \begin{itemize}
  \item the degenerate simplex $s_0 \bX_{\partial^0 \sigma} $, if the vertices $0$ and $1$ both lie in $\sigma$,
    \item the simplex $\bX_{\sigma} $ otherwise (using the identification of $\cP_0$ and $\cP_1$).
  \end{itemize}
  The morphisms
  \begin{equation}
   s_0  \bX_{\tau} \to \partial^\tau s_0  \bX_\sigma    
  \end{equation}
  asociated to each inclusion $\tau \subset \sigma$ of strata are given by
  \begin{itemize}
  \item The map $\bX_{\tau} \to \partial^\tau \bX_\sigma $ if $\sigma$ and $\tau$ do not contain both $0$ and $1$.
  \item The map $s_0  \bX_{\partial^0 \tau} \to s_0 \partial^{\partial^0 \tau} \bX_{\partial^0  \sigma} $ if $\tau$ and $\sigma$ both contain $0$ and $1$.
  \item The maps $\bX_{\tau} \to \partial^{\tau} \bX_{\partial^0 \sigma} $ or $\bX_{\tau} \to \partial^{\tau} \bX_{\partial^1 \sigma} $ if $\sigma$ contains both $0$ and $1$ but $\tau$ does not.
  \end{itemize}
\end{defin}

To define the terminal degeneracy map, we use the right module action on an $n$-simplex with vertices stratified by $(\cP_0, \ldots, \cP_n)$  instead of the left one in Equation \eqref{eq:left_bimodule_map}, to define a degeneration $ \cD_n \bX$ for an elementary simplex, which defines an $n+1$ simplex with vertices associated to the sequence $(\cP_0, \ldots, \cP_n, \cP_n)$. Following Definition \ref{def:s0-degeneracy}, replacing $\{0,1\}$ with $\{n-1,n\}$, we obtain
\begin{equation}
  s_n \co  \Flow_{n} \to \Flow_{n+1}.
\end{equation}

At this stage, we use the fact that the codimension $1$ strata associated the left and right module structure are disjoint to see that we have a natural isomorphism
\begin{equation}
    \cD_{n+1} \cD_0 \bX \cong    \cD_0 \cD_n \bX 
\end{equation}
for elementary simplices. The fact that this isomorphism is not the identity map means that we cannot expect the property $s_{n+1} \circ s_0 = s_0 \circ s_n$ to hold for the definitions given above. However, the fact that the isomorphism is canonical implies that we may simply redefine $s_{n+1}$ on vertices which are degenerate with respect to $s_0$ so that the equality holds.  This redefinition is well-defined because such a degenerate simplex admits a unique expression as $s_0 \tau$ for some $\tau$.

Tracing through the compatibility of these two constructions, we conclude:
\begin{lem}
  The initial and terminal degeneracy maps agree with $s$ on $\Flow_0$, and satisfy
  \begin{align}
    \partial^0 \circ s_0 & = \id \\
    \partial^{i+1} s_0 & = s_0 \circ  \partial^i \; \textrm{for all } i \neq 0 \\
    s_{n+1} \circ s_0 & = s_0 \circ s_n \\
    \partial^{n+1} \circ s_n & = \id \\
    \partial^{i} \circ s_{n} & = s_{n-1} \circ  \partial^i \; \textrm{for all } i \neq n.
  \end{align} \qed
\end{lem}

\subsection{Terminal degeneracies for structured flow categories}
\label{sec:degen-struct-flow}

The construction of degeneracies on the semisimplicial set constructed in Section \ref{sec:semisimplicial-set} requires a comparison between the tangent spaces of an orbifold $X$ and its degeneration $\cD X$ to the normal cone. The basic difficulty is that the projection map $\cD X \to X$ has critical points along the corners of codimension $2$ or more, so while we have an isomorphism between the tangent space of $\cD X$ and the direct sum of the pullback of $TX$ with a line, which is canonical up to contractible choice, we have to make a choice, and formulate its compatibility at the boundary.

To formulate this compatibility, it is useful to return to the definition of structured flow bimodules, and observe that what we ultimately expect is an isomorphism
\begin{equation}
    T s \bX(p,q) \oplus V_q \oplus \bR^{\{q\}} \oplus W(p,q) \cong I(p,q) \oplus V_p \oplus W(p,q) \oplus \bR.
\end{equation}
The naive idea is that the $\bR$ factor on the right corresponds to the direction of the fibres of $  sX(p,q)$ over $X(p,q)$, but the above discussion implies that this cannot be true over the boundary.

In the local model, the role of $X(p,q)$ is played by the base $[0,\infty)^n$, while that of $ sX(p,q) $ is $\cD_n$.
We need to define isomorphisms for $0 \leq n $
\begin{equation} \label{eq:framing_conic_bundle}
    T \cD_{n+1} \oplus \bR^{\{q\}} \cong T [0,\infty)^n  \oplus \bR^{\{q\}}  \oplus \bR. 
  \end{equation}
  In order to phrase the compatibility of these isomorphisms, recall that the boundary strata of $ \cD_{n+m+2} $ are enumerated by Equation \ref{eq:boundary_strata_Dn-over-corner}. In the exceptional cases, we proceed as follows:
  \begin{enumerate}
  \item For the stratum $\cD_{0} \times [0,\infty)^{n+m+1} $, we identify the normal direction with $\bR^{\{q\}}$ so that the outward pointing vector has positive coordinate and take the direct sum with the identity on $[0,\infty)^{n+m+1}  $ and with the identity map $  \bR^{\{q\}} \cong \bR$.
    \item For the stratum $ [0,\infty)^{n+m+1} \times \cD_{0}$, we identify the inward pointing normal direction with the positive generator of $\bR$ and take the direct sum with the identity on $[0,\infty)^{n+m+1}  $ and with the identity map $  \bR^{\{q\}} \cong \bR^{\{q\}}$.
  \end{enumerate}

 It thus remains to consider the following diagrams, where the horizontal maps have yet to be defined, we respectively restrict to $ \cD_{n+1} \times [0,\infty)^{m}$ or $ [0,\infty)^{n} \times \cD_{m+1} $, and the vertical arrows have yet to be defined on the $\bR^{\{q\}}$ factors (we use the identity on the $\bR$ factor):
  \begin{equation} \label{eq:consistent_framing_Dn}
    \begin{tikzcd}
        T \cD_{n+1} \oplus \bR^{\{q\}}  \oplus T [0,\infty)^{m} \ar[r] \ar[d] &  T [0,\infty)^n  \oplus \bR^{\{q\}}  \oplus \bR \oplus T [0,\infty)^{m}  \ar[d]  \\
        T \cD_{n+m+2} \oplus \bR^{\{q\}} \ar[r] &  T [0,\infty)^{n+m+2}  \oplus \bR^{\{q\}}  \oplus \bR \\
       T [0,\infty)^{n} \oplus   T \cD_{m+1} \oplus \bR^{\{q\}}   \ar[r] \ar[u] &  T [0,\infty)^{n}   \oplus T [0,\infty)^{m}  \oplus \bR^{\{q\}}  \oplus \bR \ar[u]
    \end{tikzcd}
  \end{equation}
\begin{lem}
  There is an inductive choice of isomorphisms of vector bundles in Equation \eqref{eq:framing_conic_bundle}, so that in the interior of $\cD_{n+1}$ the following properties hold:
  \begin{enumerate}
  \item The map $ T \cD_{n+1} \to  T [0,\infty)^n $ agrees with the map on tangent spaces up to post-multiplication by a diagonal matrix with strictly positive entries.
   \item The standard generator of $ \bR^{\{q\}}$  maps to the first quadrant in $\bR^{\{q\}}  \oplus \bR $  (both coordinates are positive).
  \item The generator of the projection $ T \cD_{n+1} \to  T [0,\infty)^n$ corresponding to the vector field $\partial_x$ on $ \bT \bP^1  $  maps to the second quadrant in $\bR^{\{q\}}  \oplus \bR $  (the first coordinate is negative, and second is positive).  
  \end{enumerate}
\end{lem}
\begin{proof}
  It is straightforward to see that the space of choices in the interior is contractible, so the essential point to prescribe the boundary condition, i.e., complete the construction of Diagram \ref{eq:consistent_framing_Dn} by defining the vertical maps and the middle horizontal arrow. The case of the inclusion of $ [0,\infty)^n \times \cD_{m+1} $, which is the bottom half of the diagram, is quite easy:  we identify the normal direction in $\cD_{n+m+2}$ with the normal direction of  $[0,\infty)^n \times [0,\infty)^{m}  $ in $ [0,\infty)^{n+m+1} $, and identify $ \bR^{\{q_{m+1}\}} $ with $\bR^{\{q_{n+m+2}\}}$.

  The other case is more delicate: Define the right vertical map to be the direct sum of the identity on $\bR$, the inclusion of boundary stratum of this generalised quadrant, and the map which takes $ \bR^{\{q_{m+1}\}} $ to the normal direction in $T [0,\infty)^{n+m+1}$. For the middle horizontal arrow, we identify $\bR^{\{q_{n+m+2}\}}$ on the right with the normal direction on the left, and vice-versa, which determines this map by induction because the complement of the normal direction is the tangent space of $\cD_{n+1} \times [0,\infty)^m$, on which the map is determined by the top horizontal arrow. Note that the left vertical map is similarly determined by these choices.  

Tracing through the construction, we find that the maps we have
constructed lie at the boundary of the space of maps satisfying the
required conditions, so that they can be extended to the interior. 

\end{proof}

We now apply this construction to degeneracies, discussing the case of $s_0$: given an elementary structured $n$-simplex $\bX$ with vertices labelled $\{1, \ldots, n+1\} $, we equip $s_0 \bX$ with the corresponding structure as follows: for object $(p_i,r_k)$, with $0 < i$, the corresponding morphisms in $s_0 \bX$ are canonically identified with those in $\bX$, so there is nothing to do.  Otherwise, we set $r_k = q$ in Equation \eqref{eq:framing_conic_bundle}, and identify the $\bR$ factor on its right as $\bR^{\{1\}}$. Writing the tangential isomorphism for $\bX$ as
\begin{equation} \label{eq:structure_elementary_simplex}
  T \bX(p_0,r_k) \oplus \bR^{\{r_k\}} \oplus V_{r_k} \oplus W(p_0,r_k) \cong I(p_0,r_k) \oplus V_{p_0} \oplus \bR^{\{2, \ldots, k\}} 
\end{equation}
a choice of collars for the boundary strata of the discriminant of $\cD_0$ locally fibres $\bX(p_0,r_k) $ over $[0, \infty)^{n}$, so we obtain an isomorphism
\begin{equation} \label{eq:structure_elementary_simplex_degenerate}
  T \cD_0 \bX(p_0,r_k) \oplus \bR^{\{r_k\}} \oplus V_{r_k} \oplus W(p_0,r_k) \cong I(p_0,r_k) \oplus V_{p_0} \oplus \bR^{\{1, \ldots, k+1\}}.
\end{equation}

The compatibility conditions of the framings given by the two squares
in Diagram \eqref{eq:consistent_framing_Dn} translate to the following
data: the top square asserts the compatibility for framings involving
composition with the diagonal bimodule, which is associated to the
edge from $0$ to $1$, while the second asserts that forgetting the
vertex $1$ recovers the original framing on the $n$-simplex. The
compatibility of these constructions with face maps implies: 

\begin{lem}
The semisimplicial set $\Flow^{\cS}$ of structured flow categories admits initial and terminal degeneracies, lifting those of $\Flow$. \qed
\end{lem}

\subsection{Weak units in semisimplicial sets}
\label{sec:quasi-semi-categ}
In order to complete the arguments of this section, we introduce
a criterion for when a
semisimplicial set  that satisfies the weak Kan condition (i.e., for which inner horns have fillers) 
can be equipped with the
structure of a simplicial set, thus defining 
a quasicategory.  More conceptually, this criterion
explains when the data necessary to rigidify weak units.

Steimle~\cite{Steimle} has given a criterion of this kind as follows:
he requires the existence of a map $s_0 \colon X_0 \to X_1$ such that
$s_0 x$ is an idempotent equivalence for any $x \in X$.  Here 
a map $f$ is an
equivalence if any horn $\Lambda^n_0 \to X$ with first edge $f$ has a
filler and any horn $\Lambda^n_n \to X$ with last edge $f$ has a
filler.

For our application, we need the following variant; the proof is
basically a minor modification of Steimle's argument.

\begin{prop}\label{prop:degeneracy_condition}
Let $X$ be a semisimplicial set that satisfies the weak Kan
condition.  Assume that for each $n$ we have maps:
\begin{enumerate}
\item $s_0 \co X_n \to X_{n+1}$ so that
    $d_0 \circ s_0$ and $d_1 \circ s_0 $ agree with the
    identity of $X_n$, and $d_i \circ s_0 = s_0 \circ
    d_{i-1}$ otherwise,
\item $s_n \co X_n \to X_{n+1}$ so that
      $d_n \circ s_n$ and $d_{n-1} \circ s_n$ agree with the
      identity of $X_n$, and $d_i \circ s_n = s_n \circ
      d_{i+1}$ otherwise,
\item and these maps satisfy the identities
\begin{equation}
s_{n+1} \circ s_0 = s_0 \circ s_n.
\end{equation}
\end{enumerate}
Then $X$ is the underlying semisimplicial set of a simplicial set,
i.e., a quasicategory.
\end{prop}

\begin{proof}
We shall show that the hypotheses above suffice to carry out Steimle's
arguments, specifically, to prove variants of \cite[Lemmas 2.3 and 2.5]{Steimle}.
The argument relies on the notion of an $N$-good system:
this is a collection of maps $s_k \colon X_n \to X_{n+1}$ for all $n
\geq 0$ and $0 \leq k \leq \min(n,N)$ that satisfy the simplicial
identities.  That is, an $N$-good system has degeneracies $s_0,
\ldots, s_{\min_{n,N}}$ at each level $X_n$.  We say that a system is
almost $N$-good if the last degeneracy $s_N$ at each level is not
required to satisfy the simplicial identity $d_{N+1} s_N = \id$.

Steimle proves in Lemma 2.3 that under his hypotheses, a $(N-1)$-good
system extends to an almost $N$-good system.  He proves this by
induction, successively constructing $s_N \colon X_n \to X_{n+1}$ for
each $n \geq N$.  The base of the induction is the case $n = N = 0$, where he simply uses $s_0$.  Then he uses the condition that $s_0 x$ is
an equivalence in order to start his induction, in the case that $N =
n$; i.e., to produce a map $\sigma_N \colon X_N \to X_{N-1}$ that
satisifies the simplicial identities with respect to the face maps
(but not yet the degeneracies).  Under our hypotheses, we already have
a map $s_N \colon X_N \to X_{N+1}$ that commutes with the face maps
and so we use it as $\sigma_N$ and proceed with the argument.  In the
remaining cases where $n > N$, he constructs a candidate map $\sigma_N
\colon X_n \to X_{n+1}$ by using the inner horn filling to solve the
equations that express compatibility with the simplicial identities
for the value of $d_i s_N$.  Since the map $\sigma_N$ produced in this
way need not satisfy the simplicial identities with respect to the
other degeneracy maps, we define $s_N$ to be $\sigma_N$ except that
the value on a degenerate simplex $s_i (y)$ is $s_i
s_{N-1}(y)$. 
Note that Steimle's technical Lemma 2.4 applies in our situation
without change to ensure that the modifications we described result in
well-defined maps.  Due to this adjustment, the face maps we produce coincide with our given maps for $s_0$ but not necessarily for
$s_N$.

Steimle concludes the argument by proving in Lemma 2.5 that under his
hypotheses, an almost $N$-good system $(s_0, s_1, \ldots, s_N)$ can be
turned into an $N$-good system by replacing $s_N$.  The idea is to
construct a system of candidates $T_N \colon X_n \to X_{n_2}$ for the
iterated degeneracy $s_N^2$, and replace $s_N$ with $d_N T_N$.  Once
again, his argument proceeds by induction.  When $N=0$ and $n=0$, we
can simply use $s_0^2$ in place of Steimle's appeal to the idempotence
of $s_0$ (although note that our hypotheses suffice to verify that
$s_0$ is idempotent).  When $N > 0$, the required maps are produced by
inner horn filling conditions to satisfy the necessary equations.
When $N = 0$, we again use $s_0^2$.  As in Lemma 2.3, we need to
modify the maps $T_N$ to account for degenerate
simplices.
Steimle's Lemma 2.4 again applies to ensure that the modifications we
describe result in well-defined maps.

\end{proof}

\section{The stable structure}
\label{sec:prestable-structure}

The purpose of this section is to prove the main theorem of the paper,
producing a stable $\infty$-category of structured flow categories.
Since $\Flow^{\cS}$ is a quasicategory, for any pair of objects we
have a space of maps.  Showing that ${\Flow^{\cS}}$ is stable will
allow us to construct for any pair of objects a canonical mapping
spectrum, for which the space of maps arises as the zero-space.  Floer
homotopy types then arise as maps to and from the unit flow category.

\begin{rem}
As the terminology suggests, ${\Flow^{\cS}}$ is in fact a symmetric
monoidal stable $\infty$-category; we study the monoidal structure in
a sequel to this paper.
\end{rem}

Recall that a stable $(\infty,1)$-category is defined to be an
$(\infty,1)$-category that is pointed (i.e., has a zero object), every
morphism admits a cofiber and a fiber, and fiber sequences and cofiber
sequences coincide~\cite[1.1.1.9]{Lurie2012}.  We work with the
following characterization of a stable $\infty$-category.  In this
description, for an object of an $(\infty,1)$ category $x$ we write
$\Sigma x$ to denote the colimit diagram

\begin{equation}
\begin{tikzcd}
x \ar[r] \ar[d] & \ast \ar[d] \\
\ast \ar[r] & \Sigma x.
\end{tikzcd}
\end{equation}

\begin{prop}{~\cite[1.4.2.27]{Lurie2012}}
Let $\cC$ be an $(\infty,1)$-category.  Then $\cC$ is stable if and
only if it satisfies the following conditions:
\begin{enumerate}
\item $\cC$ is pointed (i.e., has a zero object).
\item For each object $x \in \ob(\cC)$, the colimit $\Sigma x$ exists
  and the associated functor $\Sigma \colon \cC \to \cC$ is an
  auto-equivalence of $\cC$.
\item Every map $f \colon x \to y$ in $\cC$ admits a cofiber.
\end{enumerate} \qed
\end{prop}

We will proceed as follows: we begin by considering the
$(\infty,1)$-category of unstructured flow categories.  In
Section~\ref{sec:flow-pointed}, we verify the first condition and show
that ${\Flow}$ is pointed by the empty flow category.  Next, in
Section~\ref{sec:semistable}, we show that the suspension functor
$\Sigma \colon {\Flow} \to {\Flow}$ exists and induces an
auto-equivalence, i.e., that every object of ${\Flow}$ is equivalent
to a suspension and that the induced map of morphism spaces
${\Flow}(\bX,\bY) \to {\Flow}(\Sigma \bX, \Sigma \bY)$ is an
equivalence.  This condition implies that ${\Flow}$ is a full
subcategory of its stabilization $\Stab({\Flow})$; the mapping
spectrum ${\Flow}(\bX,\bY)$ can be constructed to have $n$\th space
${\Flow}(\bX, \Sigma^n \bY)$, and the structure maps arise from the
equivalences
\begin{equation}
{\Flow}(\bX,\bY) \htp {\Flow}(\Sigma \bX, \Sigma
\bY) \htp \Omega {\Flow}(\bX, \Sigma \bY).
\end{equation}
In particular, this shows that the homotopy category $\Ho({\Flow})$ is
equipped with a canonical enrichment in abelian groups.

In Section~\ref{sec:dist-triangl}, we will construct a candidate
cofiber $C(f)$ associated to a morphism $f \colon \bX \to \bY$ in
${\Flow}$ equipped with a canonical comparison morphism $C_f
\to C(f)$, where $C_f$ denotes the actual cofiber in
$\Stab({\Flow})$.  Finally, we will verify that for any flow
category $\bZ \in \ob({\Flow})$, there is a long exact
sequence of abelian groups
\begin{equation}
\cdots \to [\bZ, \bX] \to [\bZ, \bY] \to [\bZ, C(f)] \to [\bZ, \Sigma \bX]
\to \cdots
\end{equation}
where $[-,-]$ denotes the set of maps in $\Ho({\Flow})$ and moreover that these
sequences are natural in $\bZ$.  This implies that the natural map $C_f
\to C(f)$ is an isomorphism in $\Stab({\Flow})$ by the following argument: the existence of the
long exact sequence for $C_f$ and the $5$-lemma shows that the map
$[\bZ, C(f)] \to [\bZ, C_f]$ is an isomorphism for $\bZ \in
\ob({\Flow})$, and since ${\Flow}$ generates
$\Stab({\Flow})$ under colimits, we see that the map $C_f \to
C(f)$ is an equivalence.  As a consequence, we conclude that the
universal functor 
\begin{equation}
\Stab \colon {\Flow} \to \Stab({\Flow})
\end{equation}
is a categorical equivalence and hence that ${\Flow}$ is a stable
$\infty$-category. Section \ref{sec:structured-case} then explains how to extend these results to the structured case.

Because we do not have explicit control of the degeneracy maps of
${\Flow}$, there is some additional technical complexity which we
address in the next subsection, where we review some details of the
category of semisimplicial sets. 

\subsection{Preliminary facts about semisimplicial sets}

As usual, let $\Delta$ denote the category of finite ordered sets
(nonempty) and monotone maps, and  $\Delta_+$ its subcategory with morphisms the injective maps. The category of simplicial sets is 
the category of presheaves $\Fun(\Delta^{\op}, \Set)$, while the category of
semisimplicial sets are the presheaves
$\Fun(\Delta^{\op}_+, \Set)$.  Both are closed symmetric monoidal
categories, with symmetric monoidal product $\otimes$ on semisimplicial sets given by the
``geometric product'' of semisimplicial sets; this can be defined by
left Kan extension using the nerve functor or given by explicit
formulas.  The geometric product has the property that there is a
natural isomorphism $|X \otimes Y| \cong |X| \times |Y|$.

The inclusion functor $i \colon \Delta_+ \to \Delta$ induces a
forgetful functor
\begin{equation}
\Fun(\Delta^{\op}, \Set) \to \Fun(\Delta^{\op}_+, \Set)
\end{equation}
which has both a 
left adjoint  $i_!$ given by the left
Kan extension along $i;$ 
this is the
simplicial set given by freely adjoining degeneracies.  
Writing
$\Delta^n$ for the presheaf represented by $[n] \in \Delta$ and
$\Delta^n_+$ for the presheaf represented by $[n] \in \Delta_+$, we have
\begin{equation}
  i_! \Delta_+^n = \Delta^n.
\end{equation}
More generally, we also can calculate the boundaries and horns of
represented functors directly.
\begin{equation}
i_! \partial \Delta_+^n = \partial \Delta^n \qquad \textrm{and} \qquad
i_! \Lambda^n_{+,k} \cong \Lambda^n_k.
\end{equation}

The functor $i_!$ is strong monoidal in the sense that  there are natural
isomorphisms 
\begin{equation}
i_! X \times i_! Y \to i_! (X \otimes Y).
\end{equation}

Given a semisimplicial set $X$ that satisfies the conditions of Proposition \ref{prop:degeneracy_condition} we have an extension to a simplicial set; denote the
extension by $\tilde{X}$.  Clearly, $X = i^* \tilde{X}$, and the
adjunctions imply that
\begin{align}
\Map(\Delta^1_+, X) & \cong  \Map(\Delta^1, \tilde{X}) \\
\Map(\Delta^1_+ \otimes \Delta^k_+, X) &\cong  \Map(\Delta^1 \times \Delta^k, \tilde{X}).
\end{align}
As a consequence, we can conclude:
\begin{cor}
The underlying semisimplicial
set associated to the simplicial mapping space $\Fun(\Delta^1,
\tilde{X})$ is the semisimplicial mapping space $\Fun(\Delta^1_+,
X)$:
\begin{equation}
i^* \Fun(\Delta^1, \tilde{X}) \cong \Fun(\Delta^1_+, X).
\end{equation} \qed
\end{cor}

We now turn to consideration of the homotopy theory of semisimplicial
sets and the relationship to the homotopy theory of simplicial sets.
The notion of Kan complex extends naturally to semisimplicial sets
(where we ask that any map $\Lambda^{n,k}_+ \to X$ have an
extension to $\Delta^n_+ \to X$), and it is an old result of Rourke
and Sanderson \cite{Rourke1971} that any semisimplicial set $X$ satisfying the Kan condition
can be equipped with a choice of degeneracies that make it a
simplicial set $\tilde{X}$ which satisfies the Kan condition. 
Furthermore, the mapping spaces to targets which satisfy the Kan condition are again Kan, both in the semisimplicial and simplicial settings.

Recall that for a Kan complex $X$ and a basepoint vertex $x \in X_0$,
we can compute $\pi_n(X,x)$ as the set of homotopy classes of maps
$\Delta^n/ \partial \Delta^n \to X$.  Unpacking this definition, a
representative of a homotopy class is specified by an $n$-simplex of
$X$ such that the restriction to the boundary is the constant map to
$x$.  (Here note that the choice of basepoint $x \in X_0$ induces a
system of points $(s_0)^k x \in X_k$; equivalently, by the Yoneda
lemma the choice of basepoint is the same as a map of simplicial sets
$\Delta^0 \to X$.)  The equivalence relation on the maps is determined
by defining two maps to be in the same class if they are connected by
a homotopy through simplices that restrict to $x$ on their boundaries.
Using the fact that $X$ is a Kan complex, we can restate this as
observing that two maps $f,g \colon \Delta^n/ \partial \Delta^n \to X$
are equivalent if there exists a map $h \colon \Delta^{n+1} 
\to \tilde{X}$ such that
\begin{equation}
d_i h = \ast, \,\, i < n, \,\, d_n h = f, \,\,\textrm{and}\,\, d_{n+1} h = g.
\end{equation}

A basic observation in the theory of semisimplicial sets is that maps
$X \to Y$ that induce weak equivalences on geometric realizations $|X|
\to |Y|$ can be detected by homotopy groups for Kan complexes.

Specifically, observe that the formulas
\begin{align}
  \Map(\Delta^n, \tilde{X}) & \cong \Map(\Delta_+^n, X) \\
  \Map( \partial \Delta^n, \tilde{X}) & \cong
\Map( \partial \Delta^n_+, \tilde{X})
\end{align}
imply that we can compute the set of maps $\Delta^n/ \partial \Delta^n
\to \tilde{X}$ in terms of the semisimplicial structure on
$X$.  The boundary requirement involves only the faces and the
basepoint, which is specified by the given degeneracy $s_0$.  The
equivalence relation can evidently be checked without having explicit
control on the degeneracies.  The definition of the sum in the group
structure also only depends on the face maps.

\begin{rem}
As an alternate description, observe that we are considering
homotopies which take the form 
\[
\Map(\Delta^n \times \Delta^1, \tilde{X}),
\]
and these can be obtained by the prismatic decomposition of $\Delta^n
\times \Delta^1$ into $(n+1)$-simplices glued along faces; this can be
described entirely in terms of the semisimplicial data.  We will make
use of this perspective below.
\end{rem}

Summarizing, we have the following result.

\begin{prop}  \label{prop:equivalence_infty_category}
The $\infty$-category of semisimplicial sets and weak equivalences
(detected by the homotopy groups on Kan semisimplicial sets) is
equivalent to the $\infty$-category of simplicial sets. \qed
\end{prop}

As a corollary, when $Y$ is a Kan simplicial set, we have equivalences
of mapping semisimplicial sets 
\begin{equation}
\Map(i^* Y, i^* X) \htp i^* \Map(X,Y)
\end{equation}
and equivalences of spaces
\begin{equation}
|\Map(i^* Y, i^* X)| \htp |i^* \Map(X,Y)| \htp |\Map(X,Y)|. 
\end{equation}

In particular, we shall elide the difference between the semisimplicial set underlying $\Flow$, and its simplicial structure arising from  
Proposition \ref{prop:degeneracy_condition}, and thus state all our results in terms of simplicial sets and quasicategories, while providing all proofs using only the underlying semisimplicial structure. 

\subsection{Pointed mapping spaces in $\Flow$ }
\label{sec:flow-pointed}
In what follows, we let ${\Flow}(\bX,\bY)$ denote the space
of maps between flow categories $\bX$ and $\bY$, which we can model as
the simplicial set obtained as the pullback
\begin{equation}
\begin{tikzcd}
  {\Flow}(\bX,\bY) \ar[r] \ar[d] & \Fun(\Delta^1,{\Flow}) \ar[d] \\
  \{\bX\} \times \{\bY\} \ar[r] & {\Flow} \times {\Flow},
\end{tikzcd}
\end{equation}
where we are using the canonical isomorphism
\begin{equation}
\Fun(\partial \Delta^1, {\Flow}) \cong {\Flow} \times
{\Flow}
\end{equation}
to define the righthand vertical map.  Explicitly, the $n$-simplices
of ${\Flow}(\bX,\bY)$ are specified by the set of simplicial maps
\begin{equation}
\Delta^1 \times \Delta^n \to {\Flow}
\end{equation}
that restrict to $\bX$ and $\bY$ on the boundaries of the copy of
$\Delta^1$, with the simplicial identities determined by the standard
cosimplicial object $\Delta^{\bullet}$.

As discussed in the preceding section, the underlying semisimplicial
set associated to this Kan complex (equivalently the semisimplicial
mapping space between flow categories $\bX$ and $\bY$) is given by the
analogous pullback in semisimplicial sets
\begin{equation}
\begin{tikzcd}
\Flow(\bX,\bY) \ar[r] \ar[d] & \Fun(\Delta_+^1, \Flow) \ar[d] \\
\{\bX\} \times \{\bY\} \ar[r] & \Flow \times \Flow.
\end{tikzcd}
\end{equation}

The mapping spaces ${\Flow}(\bX,\bY)$ are Kan complexes. 
Therefore, as discussed above we can compute the
homotopy groups of $\Flow(\bX, \bY)$ 
combinatorially.  More generally, we have that the loop space
\begin{equation}
\Omega_+ \Flow(\bX,\bY) = \Map(\Delta_+^1 / \partial \Delta_+^1,
\Flow(\bX,\bY))
\end{equation}
is a Kan complex.

Finally, we identify canonical basepoints in the mapping spaces of
$\Flow$:

\begin{lem}\label{lem:pointedmappingspace}
The simplicial set ${\Flow}(\bX, \bY)$ can be given a canonical
basepoint using the vertex corresponding to the empty elementary
$1$-simplex with boundaries $\bX$ and $\bY$, which we write as
$\emptyset_{\bX,\bY}$. \qed
\end{lem}

It is useful for subsequent work to write out explicitly what $s_0^n
(\emptyset_{\bX,\bY})$ looks like.  When $n = 1$, we have an elementary
flow $2$-simplex with vertices $\bX$, $\bX$, and $\bY$ and one edge given by the diagonal bimodule on $\bX$. The other two edges agree with $\emptyset_{\bX,\bY}$, and the value on the top stratum is again empty (in the sense that it assigns the empty derived orbifold to each pair of objects of $\bX$ and $\bY$). 
When $n = 2$, we have a flow $3$-simplex with vertices $\bX$, $\bX$, $\bX$, and $\bY$; each stratum involving a copy of $\bY$ is empty, while those involving only copies of $\bX$ are assigned the diagonal on $\bX$, or its image under $s_0$. 
The basepoint $(s_0^n) \emptyset_{\bX, \bY}$ for $n
> 2$ can be described analogously.

We now use the following criterion to check that the empty flow category
specifies a zero object for ${\Flow}$: given an
$\infty$-category $\cC$, an object $c \in \cC$ is a zero object if
every map of simplicial sets
\begin{equation}
\partial \Delta^n \to \cC
\end{equation}
where the $0$-vertex is $c$ can be extended to a map $\Delta^n \to
\cC$.

\begin{lem}
The quasicategory ${\Flow}$ is pointed with zero object
specified by the flow category $\emptyset$ whose set of
objects is empty.
\end{lem}

\begin{proof}
The empty elementary
$n$-simplex provides a filler that extends a map of semisimplicial sets $\partial \Delta_+^n$ to $\Flow$ with $0$-vertex
the empty flow category to a map $\Delta_+^n
\to \Flow$; it is clear that this is compatible with the assignments
on the boundary $\partial \Delta_+^n$.
\end{proof}

\subsection{${\Flow}$ is a semistable quasicategory}
\label{sec:semistable}

The purpose of this section is to construct a pair $({\Sigma}, \Sigma^{-1}) $ of quasi-inverse endofunctors of 
$ {\Flow}$ and show that
for flow categories $\bX$ and $\bY$ there are natural equivalences of
mapping spaces
\begin{align}
 \Omega_+ \Flow(\bX, {\Sigma} \bY)  \leftarrow  \Flow(\bX, \bY) \to \Omega_+ \Flow({\Sigma}^{-1} \bX, \bY).
\end{align}
By Proposition \ref{prop:equivalence_infty_category},  these  equivalences imply that there are corresponding equivalences of
simplicial mapping spaces.
By the Yoneda lemma, the first equivalence
shows that ${\Sigma}$ is the suspension in ${\Flow}$.
The second shows that ${\Flow}$ is semistable.

In the unstructured case that we are considering, 
${\Sigma}$ and $\Sigma^{-1}$ are the identity functor, and so the desired
equivalences collapse to the single comparison 
\begin{equation} \label{eq:comparison_map_loops}
\Flow(\bX, \bY) \to \Omega_+ \Flow(\bX, \bY),
\end{equation}
which equivalently by adjunction is determined by a map 
\begin{equation} \label{eq:suspension_map}
\Sigma_+ 
\Flow(\bX, \bY) \to \Flow(\bX, \bY).
\end{equation}

The comparison map is constructed in the following lemma.  Recall that
for simplicial sets $X_\bullet$ and $Y_\bullet$, a map $\Sigma
X_\bullet \to Y_\bullet$ is specified by a family of maps $\sigma_n
\colon X_n \to Y_{n+1}$ that are compatible with the simplicial
identities and such that
\begin{equation}
d_0 \sigma_n = \ast \,\,\textrm{and}\,\, d_1 \circ d_2 \circ
\ldots \circ d_{n+1} \sigma_n = \ast.
\end{equation}

In what follows, we denote by $\Flow_{k}(-,-)$ the $k$-simplices of
the
mapping space $\Flow(-,-)$.  Recall that by
Lemma~\ref{lem:pointedmappingspace}, $\Flow(-,-)$ is pointed with
basepoint specified (in degree $0$) by the preferred representative of
$\bX \to \emptyset \to \bY$ specified by the empty elementary flow
$1$-simplex.

\begin{lem} \label{lem:suspension_unstructured}
For each $n \geq 0$, there are maps
\begin{equation}
\sigma_n \co  \Flow_{n}(\bX, \bY) \to \Flow_{n+1}(\bX, \bY)
\end{equation}
of sets, which assemble into the map in Equation  \eqref{eq:suspension_map}.

\end{lem}

\begin{proof}
The desired map in Equation \eqref{eq:suspension_map}  is explicitly given as maps of sets
\begin{equation}
\sigma_n \co \Map(\Delta^1_+ \otimes \Delta^n_+, \Flow) \to
\Map(\Delta^1_+ \otimes \Delta^{n+1}_+, \Flow)
\end{equation}
such that the restrictions to the endpoints of $\Delta^1_+$ are
$\bX$ and $\bY$ respectively and satisfying the conditions above.
To explain the idea of the construction, consider the case of $n=0$.
Then we are looking for a map of sets
\begin{equation}
\Map(\Delta^1_+, \Flow) \to \Map(\Delta^1_+ \otimes \Delta^1_+, \Flow),
\end{equation}
where on each side we are requiring that the restrictions to the
endpoints of the first copy of $\Delta^1_+$ are $\bX$ and $\bY$.
An element of the left hand side is a semisimplicial map $\Delta^1_+
\to \Flow$ that represents a flow $1$-simplex from $\bX$ to $\bY$.
An element of $\Map_{\bX, \bY}(\Delta^1 \times \Delta^1, \Flow)$
is specified by a homotopy commutative diagram  
\begin{equation}
\begin{tikzcd}
\bX \ar[r] \ar[d] & \bY \ar[d] \\
\bX \ar[r] & \bY.  
\end{tikzcd}  
\end{equation}
Using the standard decomposition of the prism $\Delta^1 \times
\Delta^1$ into the union of a pair of $2$-simplices, it suffices to
construct compatible maps $\Map(\Delta^1_+, \Flow) \to
\Map(\Delta^2_+, \Flow)$, which we do as follows.

Given $\vec{\cP} = (\cP_0, \ldots, \cP_n)$, there exists a map 
\begin{equation}
\vec{\cP}^{\Gamma} \to  s_1 \vec{\cP}^{\Gamma},
\end{equation}
whose image consists of arcs whose vertex carries the label $1$.  This
map permits us to relabel the underlying derived orbifolds of an
elementary flow $1$-simplex representing a map from $\bX \to \bY$
to produce an elementary flow $2$-simplex that repeats $\bY$.  On
the other hand, the relabeling induced by the analogous map
\begin{equation}
\vec{\cP}^{\Gamma} \to s_0 \vec{\cP}^{\Gamma}
\end{equation}
produces an elementary flow $2$-simplex that repeats $\bX$.  These
two simplices are compatible on the inner edge and thus produce a prism
with $\bX$ on the top and $\bY$ on the bottom.
Thus, this data specifies the required simplicial map
$\Delta^1 \otimes \Delta^1 \to \Flow$ and therefore a map of sets 
\[
\sigma_0 \colon \Map(\Delta^1, \Flow) \to \Map(\Delta^1 \otimes
\Delta^1, \Flow). 
\]
More generally, to specify an element of $\Map(\Delta^1 \otimes
\Delta^{n+1}, \Flow)$ given an element of $\Map(\Delta^1 \otimes
\Delta^n, \Flow)$, it suffices to consider the decomposition of
$\Delta^1 \otimes \Delta^n$ into $n+1$ $(n+1)$-simplices; 
we can perform analogous relabeling operations to
extend elementary flow $(n+1)$-simplices to elementary flow
$(n+2)$-simplices.  Specifically, we consider the maps
\begin{equation}
\vec{\cP}^{\Gamma} \to  s_i \vec{\cP}^{\Gamma}
\end{equation}
for $0 \leq i \leq n$ which result in relabeled flow simplices where
the first $i$ vertices are $\bX$ and the remaining vertices are
$\bY$.  These agree on their boundaries and the glued prism has
$\bX$ on the top and $\bY$ on the bottom; these relabelings glue
together to produce the desired map.

We now observe that the collection of maps $\sigma_n$ assemble to form
a map of semisimplicial sets.  Again starting with $\sigma_0$, we
observe that $d_0 \sigma_0 = d_1 \sigma_0 = \ast$; the relabeling
operation assigns the empty flow category to the faces.  Similarly,
for $\sigma_n$, we have that $d_0 \sigma_n = \ast$ since the
relabeling has the $0$th face always the image of the empty flow
category and $d_1 \circ \ldots \circ d_{n+1} \sigma_n$ is the empty
flow category.
\end{proof}

We now need to show that the adjoint of this map is an equivalence.
The basic idea of the proof is that for any spaces $X$ and $Y$, the
space of maps $X \to \Omega Y$ is equivalent to the space of null
homotopies of the constant map $X \to \ast \to Y$.  As a null homotopy
of the constant map on $\Flow(\bX,\bY)$ is equivalent to the data
of a flow bimodule representing a map from $\bX$ to $\bY$, the
result will follow.

\begin{lem} \label{lem:equivalence_loop_space}
The map produced in Lemma \ref{lem:suspension_unstructured} 
induces an equivalence in Equation  \eqref{eq:comparison_map_loops}. 
\end{lem}

\begin{proof}
We show that the map induces an isomorphism on homotopy groups.  We first
consider the components; i.e., the induced map 
\begin{equation}
\pi_0(\Flow(\bX, \bY)) \to \pi_0(\Omega_+ \Flow(\bX, \bY)).
\end{equation}
An element of $\pi_0(\Flow(\bX, \bY))$ is specified by an
equivalence class of vertices of $\Flow(\bX, \bY)$, i.e.,
equivalence classes of maps $\Delta^1_+ \to \Flow$ (flow bimodules)
which restrict to $\bX$ and $\bY$ on the boundaries of $\Delta_+^1$.
An element of $\pi_0(\Omega_+ \Flow(\bX, \bY))$ is specified by a
map $\Delta^1 \otimes \Delta^1 \to \Flow$ with prescribed behavior on
the boundaries.  Unwinding, the set of such diagrams is equivalent to
the set of $2$-simplices (null homotopies)
\begin{equation}
\begin{tikzcd}
  \bX \ar[dr] \ar[d] & \\
  \emptyset \ar[r] & \bY
\end{tikzcd}
\end{equation}
where the edge $\bX \to \bY$ is the empty bimodule, as are the the other two edges by necessity.

We now unpack what a flow $2$-simplex of this form encodes.  We have
three sets, $\emptyset, \cP_0, \cP_1$, where $\cP_0$
encodes the stratification of $\bX$ and $\cP_1$ encodes the
stratification of $\bY$.  Associated to the vertices are the empty
flow category, $\bX$, and $\bY$.  The edges $[\bX, \emptyset]$
and $[\emptyset, \bY]$ themselves encode simply the flow categories
$\bX$ and $\bY$.
The top stratum
is a flow category $\bM$ with objects $\ob(\cP_0) \times \ob(\cP_1)$
and the boundary inclusions specify action maps
\begin{equation}
 \bX(p,q) \times \bM(q,r) \to \bM(p,r) \leftarrow \bM(p,q) \times \bY(q,r) 
\end{equation}
that are associative.  That is, this data is precisely a flow
bimodule, and it is straightforward to check that the equivalence
relation of a homotopy of such diagrams that preserves the boundaries
is exactly an equivalence of flow bimodules.  By inspection, we can
see that the adjoint map takes the diagram above to the associated
flow bimodule $\bX \to \bY$; the map is an isomorphism on
$\pi_0$.

The argument for the higher homotopy groups follows from an analogous
analysis.  Specifically, the map
\begin{equation}
\pi_n(\Flow(\bX, \bY)) \to \pi_n(\Omega_+ \Flow(\bX, \bY)),
\end{equation}
is a map from the set of equivalence classes of maps
$\Delta^n \to \Flow$ with null boundaries to equivalence classes of maps $\Delta^n
\times \Delta^1 \to \Flow$ with null boundaries.  As above, the target is equivalent to the set of
equivalence classes of maps $\Delta^{n+1} \to \Flow$ with null boundaries. 
\end{proof}

\subsection{Cofibers in $\Flow$}
\label{sec:dist-triangl}

Let $\bB \co \bX \to \bY$ be a morphism in $\Flow$, i.e., a flow
$1$-simplex.  Our purpose in this section is to construct a flow
category $C(\bB)$, which we will refer to as a \emph{cone} of $\bB$,
that provides an explicit model of the cofiber of $\bB$.  This
completes the proof that ${\Flow}$ is a stable
$\infty$-category.

\subsubsection{The construction of the cone}

We begin by recalling that, by assumption, a flow bimodule is bounded below in the sense that the energy map
\begin{equation}
  \bB(x,y) \to \bR  
\end{equation}
has a uniform lower bound. If this lower bound is negative, let $\gamma \in \Gamma$ be an element whose energy is larger than its absolute value (otherwise, we set it equal to $0$), and define $T^\gamma    \bB(p,q)$ to be the derived orbifolds underlying $\bB$, with energy shifted by $\gamma$.  

\begin{defin}
The objects of $C(\bB)$ are the disjoint union $\ob(\bX) \coprod \ob(\bY)$,
and the morphisms are specified by the formulas: 
\begin{equation}
C(\bB)(p,q) \equiv
\begin{cases}
  \bX(p,q) & \textrm{ if } p,q \in \Ob(\bX) \\
 T^\gamma \bB(p,q) & \textrm{ if } p \in \Ob(\bX) \textrm{ and } q \in \Ob(\bY) \\
  \bY(p,q) & \textrm{ if } p,q \in \Ob(\bY) \\
  \emptyset & \textrm{ if } p \in \Ob(\bY) \textrm{ and } q \in \Ob(\bX).
\end{cases}
\end{equation}
\end{defin}

In order to verify that $C(\bB)$ is a flow category, observe that the
required composition maps are induced by the composition on $\bX$,
the composition on $\bY$, and the bimodule structure maps 
of $\bB$.  Implicit in this observation is the fact that the
stratification on $C(\bB)$ is induced from the stratifications on $\bX$,
$\bY$, and $\bB$; these are compatible because of the bimodule
structure.  Finally, $C(\bB)$ inherits properness over $[0,\infty)$ from the fact that
$\bX$, $\bY$, and $\bB$ are proper, and the energy shift.  Summarizing, we have the
following proposition.

\begin{lem}
The category $C(\bB)$ is a flow category with object set $\ob(\bX)
\coprod \ob(\bY)$. \qed
\end{lem}

Next, we show that $\bB$ fits into a candidate cofiber sequence with
$C(\bB)$ by constructing maps $\bI \co \bY \to C(\bB)$ and $\bP \co
C(\bB) \to \bX$ equipped with null homotopies of the composites
\begin{align}
\bI \circ \bB \co &  \bX \to \bY \to C(\bB) \\
\bP \circ \bI \co  & \bY \to C(\bB) \to \bX \\
\bB \circ \bP \co  & C(\bB) \to \bX \to \bY.
\end{align}

We begin by specifying the maps by giving flow bimodules.  The
morphism $\bI \co \bY \to C(\bB)$ is represented by the flow bimodule 
\begin{equation}
\bI(p,q) \equiv
\begin{cases}
  \emptyset  & \textrm{ if } p \in \Ob(\bY) \textrm{ and } q \in \Ob(\bX) \\
T^{\gamma}  s\bY(p,q) & \textrm{ if } p,q \in \Ob(\bY),
\end{cases}
\end{equation}
where we recall that $s \bY$ is the diagonal bimodule and $T^\gamma$ indicates a shift in action.  The left 
action by $\bY$ is given by the left action of $\bY$ on $s\bY$ and the right
action by $C(\bB)$ is specified by the right action of $\bY$ on
$s\bY$. 

The morphism $\bP \co C(\bB) \to \bX $ is specified by the flow bimodule
\begin{equation}
\bP(p,q) \equiv
\begin{cases}
  s\bX(p,q) & \textrm{ if } p,q \in \Ob(\bX)   \\
    \emptyset  & \textrm{ if } p \in \Ob(\bY) \textrm{ and } q \in \Ob(\bX).
\end{cases}
\end{equation}
The left action of $C(\bB)$ is given by the left action of $\bX$ on
$s\bX$ and the right action of $\bX$ is given by the right action of
$\bX$ on $s\bX$.

We now show that the composites of these morphisms are null:
\begin{lem} \label{lem:null_homotopy_exact_sequence}
There are distinguished null-homotopies for the composites $\bI
\circ \bB$, $\bP \circ \bI$, and $\bB \circ \bP$.
\end{lem}

\begin{proof}
The argument is trivial for $\bP \circ \bI$, while the  the other two cases are entirely analogous; we describe the null homotopy
for the composite of $\bI \circ \bB$.  Recall that to exhibit such a
null homotopy we need to construct a flow $2$-simplex $\bH$ such that
$d_0 \bH = \bB$, $d_1 \bH$ is the flow bimodule representing the
trivial map, and $d_2 \bH = \bI$.  The vertices of this $2$-simplex
must then be the flow categories $\bX$, $\bY$, and $C(\bB)$.

We view $\bB$ as a flow $1$-simplex, and observe that by construction
the degenerate flow $2$-simplices $s_0\bB$ and $s_1 \bB$ both have a
face which is canonically equivalent to $\bB$ as a flow bimodule.  The
vertices of $s_0 \bB$ are given by $\bX$, $\bX$, and $\bY$ and the
vertices of $s_1 \bB$ are given by $\bX$, $\bY$, and $\bY$.  The faces
of $s_0 \bB$ are $\bB$, $s \bX$, and $\bB \circ s \bX$; analogously, the
faces of $s_1 \bB$ are $\bB$, $s \bY$, and $s \bY \circ \bB$.

Consider the category enriched in derived orbifolds
\begin{equation}
s_0\bB \cup_{\bB} s_1 \bB,
\end{equation}
with objects the disjoint union of two copies of the objects of $\bX$ and of two copies of the objects of $\bY$, obtained by gluing the categories
associated to the top strata of $s_0 \bB$ and $s_1 \bB$, along this common stratum, to produce a
new category.  We use this to define the flow $2$-simplex $\bH$
as follows.  Given an object of $p$ in $\bX$ and an object $q$ of $C(\bB)$ arising from $\bY$, define
\begin{equation} \label{eq:definition_null-homotopy}
\bH (p,q)  \equiv
T^{\gamma} \left( s_0\bB(p,q) \cup_{\bB(p,q)} s_1 \bB(p,q) \right) 
\end{equation}
We extend this to the case $q$ lies in the subset of $ \ob(C(\bB))$ corresponding to the objects of $\bX$ by the enriched bimodule associated to the diagonal bimodule.

We build a flow $2$-simplex out of this category, which assigns the above data to the simplex $\Delta^2$ as follows: set the vertices to be $\bX$,
$\bY$, and $C(\bB)$,  then assign $\bB$ to the edge between $\bX$
and $\bY$; this uses the identification of the boundary faces of
$s_0\bB $ with the composition of $s\bX \circ \bB$ and of those of
$s_1 \bB$ with the composition $\bB \circ s\bY$.  By construction, Equation \eqref{eq:definition_null-homotopy} is consistent with  specifying that we
can assign $\bI$ to the edge between $\bY$ and $C(\bB)$, and similarly
the trivial flow bimodule to the edge between $C(\bB)$ and $\bX$. Thus, $\bH$ is a flow $2$-simplex that encodes the desired null
homotopy.
\end{proof}

\subsubsection{The long exact sequence of the cofiber}
\label{sec:long-exact-sequence}

Finally, we complete the proof that $\Flow$ is stable by showing that
$\bP$ and $\bI$ induce a long exact sequence of abelian groups of maps
in $\Ho(\Flow)$ associated to $C(\bB)$.  Given flow categories $\bX$
and $\bY$, recall that the abelian group of maps $[\bX, \bY]$ in
$\Ho(\Flow)$ is given by the set of flow bimodules $\bM \colon \bX \to
\bY$ modulo the equivalence relation given by homotopy, as exhibited
by the existence of a flow $2$-simplex witnessing the homotopy.

\begin{prop} \label{lem:exactness_at_B}
Let $\bZ$ be a flow category.  The following periodic sequence of
abelian groups, whose maps are induced by $\bB$, $\bI$, and $\bP$, is exact: 
\begin{equation}\label{eq:flow-LES}
\cdots \to  [\bZ, \bX] \to [\bZ, \bY] \to [\bZ, C(\bB)] \to  [\bZ, \bX] \to \cdots.
\end{equation}

\end{prop}

\begin{proof}
The three exactness statements have completely analogous proofs; we
give the proof for exactness at $C(\bB)$.  The statement that
Equation~\eqref{eq:flow-LES} is exact at $C(\bB)$ is equivalent to the
assertion that if a flow bimodule $\bF \colon \bZ \to C(\bB)$
satisfies the property that the composite $\bP \circ \bF \colon \bZ
\to \bX$ is null-homotopic, then $\bF$ is homotopic to the composite
$\bI \circ \bG$ for some flow bimodule $\bG \colon \bZ \to \bY$. To
prove this, we show that both conditions are equivalent to the
condition that $\bF$ is homotopic to a flow bimodule $\bM \colon
\bZ \to C(\bB)$ which is trivial on $\bX$, in the sense that the
morphisms are the empty orbifold for any target $x \in \ob(\bX)$.

Thus, we consider the set of flow bimodules $\bM \colon \bZ \to
C(\cB)$ that are trivial on $\bX$, up to the equivalence relation
given by homotopy.  On the one hand, since $s\bY$ is a unit for $\bY$,
such a bimodule is clearly homotopic to the composite $\bI \circ \bG$
for some flow bimodule $\bZ \to \bY$.  And given a flow bimodule $\bZ
\to \bY$, the composition with $\bI$ has the specified property, i.e.,
is trivial on $\bX$ by definition.  On the other hand, assume that
$\bF$ is homotopic to a flow bimodule $\bM$ which is trivial on
$\bX$.  Then by definition, the composite $\bP \circ \bF$ is homotopic
to the trivial bimodule, since $\bP$ is specified as the weak unit
$s\bX$ on the objects of $\bX$ and is trivial on the objects of $\bY$.
Moreover, when the composite $\bP \circ \bF$ is homotopic to the
trivial bimodule, because $\bP$ specified by the weak unit $s\bY$,
there is an obvious homotopy connecting $\bF$ to a bimodule which is
trivial on $\bX$.

\end{proof}

\subsection{The structured case}
\label{sec:structured-case}

We complete this section by extending its results to structured flow
categories. Since the empty derived orbifold is tautologically framed
(with isomorphism $\bR^{\{q\}} \cong \bR^{\{1\}}$ induced by the map
of sets), we start with a discussion of the stable structure.

We define pseudofunctors $\Sigma$ and $\Sigma^{-1}$ of $d \Orb^{\cS}$
which act on $0$-cells by
\begin{align} \label{eq:suspend_vspace}
\Sigma (V_+, V_-) & \equiv (V_+ \oplus \bR, V_-)
\\ \label{eq:inverse_suspend_vspace} \Sigma^{-1} (V_+, V_-) & \equiv
(V_+, V_- \oplus \bR)
\end{align}
and on (relative) framings of $1$-cells by taking the direct sum of
all isomorphisms of vector bundles with the identity on $\bR$. It will
be convenient to refer to these copies of $\bR$ as \emph{suspension
coordinates}.  Note that these two pseudofunctors commute, and that we
have a pseudonatural transformation $\id \to \Sigma \circ \Sigma^{-1}
$ given by the point derived orbifold, equipped with the framing
arising from the direct sum of the identity on $V$ which the identity
on the copies of $\bR$ that have been added to the positive and
negative parts of $V$.  This pseudonatural transformation is clearly
an equivalence.

By composition, these shift functors induce endofunctors $\Sigma$ and
$\Sigma^{-1}$ of $\Flow^{\cS}$ defined as follows: given a vertex of
$\Flow^{\cS}$, i.e., a flow category, the induced map comes from the
shift in the framing on each morphism object; this specifies the map
$\Flow^{\cS}_0 \to \Flow^{\cS}_0$.  Next, applying the composition to
each flow category associated to a structured flow simplex produces a
new structured flow simplex; the shifts on morphism objects are
clearly compatible with the passage to the boundary.  Similarly, this
assembles into a semisimplicial map $\Flow^{\cS} \to \Flow^{\cS}$, as
the same observation implies that the shifts are compatible with the
face operators.  Finally, it is immediate from the discussion in
Section~\ref{sec:degen-struct-flow} that the shift strictly commutes
with the construction of the degeneracy $s^0$; i.e., $s^0 \Sigma =
\Sigma s^0$ and $s^0 \Sigma^{-1} = \Sigma^{-1} s^0$.
By~\cite{Steimle,Tanaka}, this suffices to show that up to coherent
homotopy $\Sigma$ and $\Sigma^{-1}$ define endofunctors of the
quasicategory ${\Flow^{\cS}}$.

Finally, we have a canonical natural transformation from the identity
functor of ${\Flow^{\cS}}$ to the composite $\Sigma \circ
\Sigma^{-1}$, which on every simplex $\bX$ is obtained from the
degenerate simplex $s_0 \bX$, equipped with the direct sum of its
framing with the identity on the copy of $\bR$ as above.

We now consider the construction of
Lemma~\ref{lem:suspension_unstructured}, which we claim lifts in the
structured case to maps  
\begin{align} \label{eq:suspension_map_structure}
 \Flow^{\cS}(\bX, {\Sigma} \bY) \leftarrow \Sigma_+ \Flow^{\cS}(\bX, \bY) \to    \Flow^{\cS}({\Sigma}^{-1}  \bX, \bY)
\end{align}
of semisimplicial sets.

This is clear from the definitions: the unstructured simplices
realizing the map in Equation~\eqref{eq:suspension_map} are obtained
by considering a $k$-simplex in the source as a $(k+1)$-simplex in the
target. Defining a relative framing on a $(k+1)$-simplex involves an
additional copy of $\bR$ according to Definition
\ref{def:framed_flow_simplex-elementary}, which we isomorphically map
to the copy of $\bR$ in the stable vector spaces associated to the
objects of $ {\Sigma}^{-1} \bX$ or $ {\Sigma} \bY$. We are ready to
consider the analogue of Lemma \ref{lem:equivalence_loop_space}:
\begin{lem}
The adjoint of the maps in
Equation~\eqref{eq:suspension_map_structure} are weak equivalences 
\begin{align} 
   \Omega_+ \Flow^{\cS}(\bX, {\Sigma} \bY)  \leftarrow  \Flow^{\cS}(\bX, \bY) \to \Omega_+ \Flow^{\cS}({\Sigma}^{-1} \bX, \bY).
\end{align}
\end{lem}
\begin{proof}
One additional argument is required, which we explain in
one of the two cases: the relative framing on an $(n+1)$-simplex of
$\Flow^{\cS}$ with vertices $(\bX, \emptyset, \ldots, \emptyset,
{\Sigma} \bY)$ need not satisfy the property that it maps the
distinguished copy of $\bR$ in the virtual vector space associated to
each object of $\bY$ to the copy of $\bR$ which in Definition
\ref{def:framed_flow_simplex-elementary} is labelled
$\bR^{\{1\}}$. However, this may be achieved by passing to a
stabilization, by induction on the codimension of the largest corner
stratum of the derived orbifolds underlying the simplex.
\end{proof}

Next, we may lift the construction of the cone object $C(\bB)$ to the
structured setting: noting again that the relative framing on flow
bimodules involves an additional copy of $\bR$ relative to those of
flow categories, we set the objects of $C(\bB)$ to be the union of
those in $\Sigma \bX$ and of $\bY$, and identify the suspension
coordinate on $\bX$ with this factor. The maps $\bI$ and $\bP$
obviously lift to maps
\begin{equation}
  \begin{tikzcd}
    \bY \ar{r}{\bI} & C(\bB) \ar{r}{\bP} & \Sigma \bX.
  \end{tikzcd}
\end{equation}

We now complete the proof of the main result of this paper:

\begin{proof}[Proof of Theorem \ref{thm:main-thm-Flow}]
It remains to verify that $\bI$, $\bB$, and $\bP$ induce long exact
sequences on morphism spaces.  The statement that the composite of
these maps vanishes is the lift of
Lemma~\ref{lem:null_homotopy_exact_sequence}.  Restricting our attention to
the composition $\bI \circ \bB$, we have to equip the $2$-simplex
$\bH$ witnessing the null-homotopy of this composite with a relative
framing.

Given an object $p$ in $\bX$ and $q$ in $C(\cB)$ that comes from
$\bY$, recall that the essential construction produces derived
orbifolds from the union $s_0\bB(p,q) \cup_{\bB(p,q)} s_1 \bB(p,q)$ of
degenerate flow bimodules along a common face.  Since the gluing takes
place at opposite ends of the interval, the relative framings are
consistent across the boundary which is glued, so that a choice of
collar determines a relative framing on the union.  It is now
immediate that this relative framing is compatible with the
definition of the flow $2$-simplex $\bH$.

This finishes the main part of the argument, because the proof of
Lemma~\ref{lem:exactness_at_B} goes through in this case without
modification.
\end{proof}

\section{Framed flow categories and spectra}
\label{sec:comp-with-sphere}

A foundational aspect of the theory we are constructing is that the
stable $\infty$-categories $\Flow^{\cS}$ can be described as
categories of modules over suitable bordism spectra, provided that the monoid $\Gamma$ is trivial. This assumption is required in order to conclude that coproducts are represented by disjoint union of flow categories, as the description of morphisms from a disjoint union is otherwise given by the colimit over the lower bound in Condition \eqref{eq:lower_bound_flow_simplex}. The same argument gives us a description of morphisms into a coproduct, so the fact that $\Flow^{\cS}$ is stable implies:

\begin{lem}  
If $\Gamma$ is the trivial monoid, the category $\Flow^{\cS}$ is
closed under (homotopy) colimits, and the unit is a compact object.
\end{lem}

In the remainder of this section, we use these basic structural observations to show
that $\Flow^{\fr}$ is equivalent to the $\infty$-category of spectra,
i.e., modules over the sphere spectrum.

\subsection{$\Flow^{fr}$ is generated by the unit}
\label{sec:generation}

In this subsection we will prove that the unit object generates
$\Flow^{fr}$, in the sense that an object of $\Flow^{fr}$ is zero if
and only if maps from all shifts of the unit are zero.  Since
$\Flow^{fr}$ is stable, this is equivalent to $\Flow^{fr}$ being
generated under (homotopy) colimits by the unit.

Recall from the discussion of
Remark~\ref{rem:properness_flow_category} that the set of objects of
any flow category $\bX$ (with trivial action) can be equipped with a
partial order with $q < p$ if there is a morphism from $p$ to
$q$. This convention is designed to be compatible with the fact that
if the set of homotopy classes of maps $S^p \to S^q$ is nontrivial,
then $q \leq p$.  We map this partial order to a totally ordered set
$A$.

Given elements $\alpha$ and $\beta$ in $A$, let $\bX_{[\alpha,\beta]}$ 
denote the flow category obtained from $\bX$ by restricting the
objects to those which project to elements of $A$ lying between
$\alpha$ and $\beta$ (inclusive).  We write $\bX_{[\beta,\infty)}$ and
  $\bX_{(-\infty,\alpha]}$ for the categories obtained by imposing a
bound only on one side.

\begin{lem}\label{lem:flowcolimit}
There is a natural map $ \bX_{(-\infty, \alpha']} \to
  \bX_{(-\infty,\alpha]}$ whenever $\alpha' \leq \alpha$.  These maps
    are compatible in the sense that for each triple $\alpha '' <
    \alpha' < \alpha $, the diagram 
  \begin{equation}
    \begin{tikzcd}
      \bX_{(-\infty,\alpha'']} \ar[r] \ar[dr] &  \bX_{(-\infty,\alpha']} \ar[d] \\
      &  \bX_{(-\infty,\alpha]}
    \end{tikzcd}
  \end{equation}
  homotopy commutes.
\qed
\end{lem}

We also have analogous maps $\bX_{(-\infty, \alpha)} \to \bX$ that are
compatible with Lemma~\ref{lem:flowcolimit} and so induce a natural
map
\begin{equation}
\hocolim_\alpha \bX_{(-\infty, \alpha)} \to \bX.
\end{equation}
In fact, this map is an equivalence.  We use in the following
result that the homotopy colimit can be constructed simply as the
strict colimit (union), since the maps are essentially inclusions.

\begin{lem} \label{colimit_bounded_above}
The homotopy colimit of the system
\begin{equation}
\ldots \to \bX_{(-\infty,\alpha]} \to \bX_{(-\infty, \alpha']} \to
  \bX_{(-\infty, \alpha'']} \to \ldots
\end{equation}
is represented by $\bX$.
\end{lem}

\begin{proof}
Given a fixed flow category $\bY$, the map of simplicial sets for
$\alpha < \alpha'$
\begin{equation}
\Flow^{\cS}(\bX_{(-\infty,\alpha']}, \bY) \to \Flow^{\cS}( \bX_{(-\infty,\alpha]}, \bY)    
\end{equation}
is represented by restricting the flow simplex to those objects of $
\bX_{(-\infty,\alpha']} $ which lie below $\alpha$.  Therefore, it is
  clear that  
\begin{equation}
\Flow^{\cS}(\bX_{(-\infty,\alpha]}, \bY ) \cong \lim_{\alpha' \leq
    \alpha}  \Flow^{\cS}(\bX_{(-\infty,\alpha']}, \bY).
\end{equation}
The result follows.
\end{proof}

Dually, the flow categories $\bX_{[\beta,\infty)}$ form an inverse
  system, but its homotopy limit may not be represented by $\bX$:
  
\begin{example}
Suppose that $\bX$ is a flow category with objects labelled by the natural
numbers and all morphisms the empty derived orbifold.  The space of
morphisms from $\ast$ to $\bX$ is identified, by compactness, with
direct sum of endomorphisms of $\ast$, indexed again by the natural
numbers. On the other hand, we have an inverse system
  \begin{equation}
   \cdots \to \ast \amalg \ast \amalg \ast \to \ast \amalg \ast \to
   \ast
  \end{equation}
  where the $k$\th bonding map is the diagonal bimodule on the first
  $k$ factors, and the empty bimodule on the last. The inverse limit
  of the space of morphisms from $\ast$ is then the direct product of
  endomorphisms of $\ast$. Even at the level of homotopy groups, the
  natural map between them is not an isomorphism whenever the
  endomorphisms of the unit are nontrivial.
\end{example}

We thus restrict attention to the situation where the $\infty$-category
of flow categories admits an additional structure which guarantees
that for a fixed object $p$ in a flow category $\bX$, the morphisms
$\bX(p,q)$ are eventually empty as the value of $q$ decreases in the
order on the object of $\bX$.

\begin{defin}
Let $\Flow^{\cS}$ denote the $\infty$-category of structured flow
categories.  The tangential structure is \emph{graded connective} if
there is an assignment for each structured flow category $\bX$ of an
integer $\dim(p)$ to each object $p \in \ob(\bX)$ satisfying the
following properties:
\begin{enumerate}
\item the morphism space $\bX(p,q)$ is empty whenever $\dim(p) \geq \dim (q)$,
\item  the value of $\dim$ is increased by $1$ under the natural identification of objects of $\bX$ and its suspension, and 
  \item  the dimension function on the mapping cone of a map $\bX \to \bY$ agrees with the dimension functions on $\bY$ and the suspension of $\bX$.
\end{enumerate}
\end{defin}

Since $\bX(p,q)$ is empty when $\dim(p) = \dim(q)$, as $q$ decreases
in the order on the objects, we have that $\dim(q) \to -\infty$.  To
apply this, we choose the set indexing the filtration to be the integers:  

\begin{lem} \label{limit_bounded_below}
If the tangential structure is graded connective, the homotopy limit
of the flow categories $\bX_{[\beta,\infty)}$ is equivalent to $\bX$.
\end{lem}

\begin{proof}
We will prove that the natural map
\begin{equation}
\Flow^{\cS}(\bY, \bX)  \to  \lim_{\beta} \Flow^{\cS}(\bY,  \bX_{[\beta,\infty)} )
\end{equation}
is an equivalence for any flow category $\bY$.  The key point is to
show surjectivity, i.e., to reconstruct a flow simplex $\bB$ in $
\Flow(\bY, \bX)$ from the datum of an inverse system $\bB_{\beta}$ of
flow simplices. Such a flow bimodule is determined by its value on any
object $p$ of $\bY$, but the connectivity assumption implies that the
collection $\coprod_{q \in \bX_{[\beta,\infty)}} \bB_{\beta}(p,q)$ is
  eventually constant in $\beta$, thus determining a flow simplex
  $\bB$ as desired.
\end{proof}

Now, since the subcategory generated by a given object is always
closed under limits and closed under colimits when the object is
compact, the fact that a discrete flow category is a coproduct of
copies of the unit allows us to conclude the following:

\begin{prop} \label{prop:generation}
The category of flow categories for a graded connective tangential
structure is generated by the unit. \qed 
\end{prop}

We will apply this below by endowing $\infty$-category $\Flow^{fr}$ with 
a graded connective tangential structure induced by the dimension of a
manifold.

\subsection{Framed bordism of a point}
\label{sec:framed-bordism-point}

We begin by quickly reviewing the explicit construction of mapping
spectra in the setting of $\infty$-categories with suspension.  That
is, we assume we have a pointed quasicategory $\cC$ equipped with a
functor $\Sigma \colon \cC \to \cC$ such that there is a natural
equivalence of mapping spaces $\cC(\Sigma x, y) \htp \Omega \cC(x,
y)$.  Given objects $x,y \in \ob(\cC)$, we can form a mapping spectrum
$\underline{\cC}(x,y)$ that has $k$th space the mapping space $\Map_{\cC}(x,
\Sigma^k y)$, where $\Sigma^k$ denotes the $k$-fold iteration of the
suspension functor.  The structure map is induced from the adjoint of
the composite
\begin{equation}
\cC(x, y) \to \cC(\Sigma x, \Sigma y) \to \Omega \cC(x, \Sigma y).
\end{equation}
In general, $\Omega^\infty \underline{\cC}(x,y)$ is not equivalent to $\cC(x,y)$; passing to mapping spectra is a kind of stabilization of $\cC$.  (See for example~\cite{Blumberg2020} for a detailed study of this kind of stabilization.)
A functor $F \colon \cC \to \cD$ between pointed quasicategories with suspension which
preserves the point and commutes with the suspension in the sense that
there is a natural transformation $F \Sigma \to \Sigma F$ induces maps
on mapping spectra $\underline{\cC}(x,y) \to \underline{\cD}(Fx, Fy)$.

We now turn to the identification of the endomorphism spectrum of the
unit in $\Flow^{fr}$.  Given a finite-dimensional inner product space
$V$, let $D(V)$ denote the unit ball in $V$ and $\partial D(V)$ denote
its boundary.  In the following, in mild abuse of notation we write
$(X,A) \times I^k$ to denote the product of pairs $(X,A) \times (I^k,
\emptyset)$.  To control set-theoretic issues, we tacitly assume that
all of the finite dimensional inner product spaces we consider are
subspaces of some fixed universe $U$, i.e., a countably infinite
dimensional inner product space.

\begin{defin}
We define $\Omega^{\infty} (D,\partial D)^{\infty}$ to be the simplicial set
whose vertices are finite-dimensional inner product spaces, and where
an $n$-simplex with vertices $(V_0, \ldots, V_{n})$ is given by a
collection of maps of pairs 
\begin{equation}
h_{i,k} \colon (D(V_i), \partial D(V_i)) \times [0,1]^{\{i+1, \ldots, k-1 \}}\to (D(V_k), \partial D(V_k))   
\end{equation}
for each pair $i < k$, whose restriction to the face given by the
vanishing of the coordinate labelled by $j \in \{i+1, \ldots, k-1 \}$
is given by the composition 
\begin{align*}
(D(V_i), \partial D(V_i)) \times [0,1]^{\{i+1, \ldots,j-1, j+1 ,
    \ldots  k-1 \}} &\to (D(V_j), \partial D(V_j))    \times
  [0,1]^{\{j+1 , \ldots  k-1 \}} \\
  &\to (D(V_k), \partial D(V_k)).
\end{align*}
The face maps are defined by the evident restrictions and
compositions, and the degeneracies are defined by repeating vector
spaces and inserting the identity map and constant homotopy.
\end{defin}

The simplicial set $\Omega^{\infty} (D, \partial D)^{\infty}$ is
evidently a quasicategory.  Specifically, it provides a convenient
model of the homotopy coherent nerve of the simplicial category $\cS$
with objects representation spheres $S^V$ and morphism spaces
continuous maps $S^V \to S^{V'}$.  Although it is does not have enough
objects to be a stable $\infty$-category, there is the evident
suspension functor $\Sigma \colon \Omega^{\infty} (D, \partial
D)^{\infty} \to \Omega^{\infty} (D, \partial D)^{\infty}$ given by
smashing with $S^1$ and using the identity map and constant homotopy
on the suspension variable.
  
We now turn to a precise identification of $\Omega^{\infty} (D,
\partial D)^{\infty}$ as follows.  Denote by $\Omega^{\infty} (D,\partial
D)^{\infty}_{\strict}$ the category with objects
finite-dimensional inner product spaces and morphisms continuous maps
of pairs
\begin{equation}
(D(V), \partial D(V)) \to (D(V'), \partial D(V')).
\end{equation}
Passing to the nerve and inverting maps which induce weak equivalences
of pairs yields a pointed quasicategory with a suspension functor, and
the standard rectification arguments show that the evident inclusion 
\begin{equation}\label{eq:strictifyspherecategory}
N_\bullet \Omega^{\infty} (D,\partial D)^{\infty}_{\strict} \to
\Omega^{\infty} (D,\partial D)^{\infty}
\end{equation}
induces a categorical equivalence of quasicategories after inverting
the weak equivalences in $\Omega^{\infty} (D,\partial
D)^{\infty}_{\strict}$.  Moreover, the inclusion evidently commutes
with the suspension functor and therefore there is an induced
equivalence of mapping spectra.

On the other hand, the canonical isomorphisms 
\begin{equation}
D(V) / \partial D(V) \cong S^V
\end{equation}
yields a comparison between $\Omega^\infty (D,\partial
D)^{\infty}_{\strict}$ and the nerve of the category $\cS$ with
objects the representation spheres $S^V$ and morphism spaces the
continuous maps: there is a categorical equivalence
\begin{equation}\label{eq:diskstospheres}
N_\bullet \Omega^{\infty} (D,\partial D)^{\infty}_{\strict} \to
N_\bullet \cS
\end{equation}
after inverting weak equivalences on both sides.  Again, this is an
equivalence of $\infty$-categories with suspension and therefore
induces an equivalence of mapping spectra.

\begin{rem}
Although we do not use this here, note that the direct sum of inner
product spaces and Cartesian product of disks induces an
$E_\infty$-space structure (i.e., symmetric monoidal $\infty$-category
structure) on $\Omega^{\infty} (D, \partial D)^{\infty}$; homotopies
are determined by the obvious product homotopy.
\end{rem}

The zigzag described above in
equations~\eqref{eq:strictifyspherecategory}
and~\eqref{eq:diskstospheres} induces an equivalence of spectra
between $\bS$ and the endomorphism spectrum of the point in
$\Omega^{\infty} (D, \partial D)^{\infty}$.

Now consider the subsimplicial set $\Bord^{fr}$ of $\Omega^{\infty}
(D,\partial D)^{\infty}$ which has the same vertices and whose higher
simplices consist of maps which are smooth near the inverse image of
the origin, and transverse to it.  Composition again yields a
quasicategory structure on this simplicial set.  Note that the
defining condition is preserved by the direct sum operation, and so
$\Bord^{fr}$ inherits the structure of an $E_\infty$ space.  Moreover,
the suspension functor descends to $\Bord^{fr}$.  Standard
transversality theory combined with the discussion above now implies:

\begin{lem}
The inclusion $\Bord^{fr} \subset \Omega^{\infty} (D,\partial
D)^{\infty}$ is a categorical equivalence of quasicategories and
induces an equivalence of mapping spectra between $\bS$ and the
endomorphisms of the point in $\Bord^{fr}$.
\end{lem}

Passing to the zero locus $X$ of a simplex in $\Bord^{fr}$,
transversality determines a framing 
\begin{equation} \label{eq:framing-zero-locus}
  T X \oplus V' \cong V \oplus \bR^{ \{1, \ldots, k-1 \}} 
\end{equation}
which is canonical up to contractible choice. We take the direct sum
of this framing with the identity map $\bR^{\ast} \cong \bR^{\{k\}}$
in order to obtain a framing in the sense of Definition
\ref{def:framed_flow_simplex}. Choosing the framings in Equation
\eqref{eq:framing-zero-locus} to be compatible with the product
decomposition at the boundary, we obtain a simplex in $\Flow^{fr}$,
so that inductive choices yield a map of simplicial sets
\begin{equation} \label{eq:bordism_to_flow_categories}
\Bord^{fr}  \to \Flow^{fr}
\end{equation}
and hence a functor of quasicategories.  Furthermore, this functor is
clearly compatible with suspension; the suspension functor adds a copy
of $\bR$ to the framing.

We will now use this comparison to identify the endomorphisms of the
unit in $\Flow^{\fr}$.  Since the comparison $\Bord^{fr} \to
\Flow^{\fr}$ commutes with the suspension functor, it induces a map of
mapping spectra and in particular induces a comparison of
endomorphism spectra of the unit.  The following result shows that
this map is a stable equivalence.

\begin{lem} \label{lem:bordism_is_flow-in-limit}
Let $V$ and $V'$ be finite dimensional inner product spaces.  The
induced map $\Bord^{fr}(V, V') \to \Flow^{fr}(V,V')$ is $(\dim
V)$-connected.
\end{lem}

\begin{proof}
It suffices to prove that the map on homotopy groups is an isomorphism
in the desired range. The key point is to prove surjectivity, so we
consider a representative for a homotopy class in $\Flow^{fr}(V,V') $
given by an $n$-simplex with empty boundary, and our goal is to show
that it is equivalent, relative the boundary condition, to a simplex
in the image of $\Bord^{fr}(V, V') $. The first step is to appeal to
the standard transversality argument to assume that the section of the
derived manifold that we are considering is transverse to the
$0$-section. Next, by a degeneration to the normal cone of this zero
locus, we may assume that the obstruction bundle is trivial. We are
thus in the situation of a closed manifold $X$ of dimension $n + \dim
V - \dim V'$, equipped with a stable framing
 \begin{equation}
  T X \oplus V' \oplus \bR^{\ast} \oplus W \cong V \oplus \bR^{ \{1,
    \ldots, k \}} \oplus W.
\end{equation}
Our goal is to deform this framing so that it is given by the direct
sum of a framing as in Equation~\eqref{eq:framing-zero-locus}, with
the identity on
\begin{equation}
  \bR^{\ast} \oplus W \cong \bR^{ \{ k \}} \oplus W.
\end{equation}
This can be achieved under our assumption by the connectivity of the
map $O(V) \to O(V \oplus \bR^{\infty})$.
\end{proof}

Putting it all together we conclude the following corollary.

\begin{cor}\label{cor:endospectrum}
The endomorphism spectrum of the unit in $\Flow^{\fr}$ is equivalent
to the sphere spectrum.
\end{cor}

Finally, we can deduce the desired identification of $\Flow^{\fr}$
as the stable category.

\begin{proof}[Proof of Proposition~\ref{prop:framed_flow=spectra}]
Proposition~\ref{prop:generation} implies that $\Flow^{\fr}$ is
generated by the unit, so it suffices to show the assertion that the
endomorphism spectrum of the unit is equivalent to the sphere spectrum
as associative ring spectra.  Corollary~\ref{cor:endospectrum} shows
that there is an equivalence of spectra, and the result then follows
from the fact that the sphere spectrum admits a unique $A_\infty$ ring
structure.
\end{proof}

\appendix
\section{A review of orbifolds and orbibundles}
\label{sec:orbif-orbib}

The purpose of this section is to review the elementary theory of
effective orbifolds and vector bundles on orbifolds.  Our goal is to
provide a minimal treatment that supports the definition of the
symmetric monoidal category of derived orbifold charts.  The
discussion of Proposition~\ref{prop:horn_filling} (in
Section~\ref{sec:enough-orbibundles}) requires the more elaborate
theory of orbifolds as stacks, and we refer the reader to
Pardon~\cite{Pardon2019}.

\subsection{``Classical'' orbifolds}
\label{sec:classical-orbifolds}

We give a terse review of the definition of an orbifold in terms of
charts and atlases on a suitable topological space,
c.f.~\cite{Chen2002}.  Let $X$ be a second countable Hausdorff space.

\begin{defin}\label{def:orbifold-chart}
An orbifold chart centered at a point $x \in X$ is a triple $(G, U,
\phi)$ where $G$ is a finite group that acts smoothly and effectively
on $\bR^n$,  $U \subset \bR^n$ is an open $G$-invariant neighborhood
of $0$, and $\phi \colon U \to X$ is a $G$-invariant map such that
$\phi(0) = x$, which induces a homeomorphism of $U / G$ onto its
image.
\end{defin}

An {\em embedding} of orbifold charts $(G, U, \phi) \rightarrow (G',
U', \phi')$ is a smooth embedding $i \colon U \to U'$ such that $\phi
= \phi' \circ i$.  Such an embedding gives rise to a unique
homomorphism $G \to G'$.

Two orbifold charts $(G, U, \phi)$ and $(G', U', \phi')$ are {\em
  compatible} if for every $z \in U \cap U'$, there exists a chart
$(G'', U'', \phi'')$ centered at $z$ such that there exist embeddings
$(G'', U'', \phi'') \to (G, U, \phi)$ and $(G', U', \phi') \to (G, U,
\phi)$.

\begin{defin}\label{def:orbifold-atlas}
An orbifold atlas for $X$ is a collection of compatible orbifold
charts $\{(G_{\alpha}, U_{\alpha}, \phi_{\alpha})\}$ whose images cover $X$.
\end{defin}

We shall need the generalisation of these notions to allow for
boundaries and corners; the definition that we are about to give is
not the most general possible, as we shall require the action to be
trivial in the normal direction of each corner stratum:

\begin{defin}\label{def:orbifold-chart-with-corners}
An orbifold chart with corners for $X$ centered at $x \in X$ is a
triple $(G, U, \phi)$ where $G$ is a finite group acting smoothly and
effectively on $\bR^{n-k}$, $U \subset \bR^k_+ \times \bR^{n-k}$ is an
open $G$-invariant neighborhood of $0$, and $\phi \colon U \to X$ is a
$G$-invariant map such that $\phi(0) = x$, which induces a
homeomorphism of $U / G$ onto its image.
\end{defin}

An embedding of orbifold charts with corners is specified by an
embedding of orbifold charts that is compatible with the corner
structure.  The notion of an atlas is then analogous to
Definition~\ref{def:orbifold-atlas}.

\begin{defin}\label{def:stratified-orbifold-atlas}
A orbifold atlas with corners for $X$ is a collection of
compatible orbifold charts with corners $\{(G_{\alpha}, U_{\alpha},
\phi_{\alpha})\}$, whose images cover $X$. 
\end{defin}

An atlas $A'$ {\em refines} an atlas $A$ if every chart in $A$ embeds into a
chart in $A'$.  We say that two atlases are {\em equivalent} if there exists
a common refinement.  In particular, any atlas is contained in a
maximal atlas and two atlases are equivalent if and only if they are
contained in the same maximal atlas.

\begin{defin}\label{defn:effective-orbifold}
An effective orbifold (with corners) is a second countable Hausdorff
space $X$ equipped with a maximal orbifold atlas, or equivalently an
equivalence class of orbifold atlases.
\end{defin}

The most basic example of an effective orbifold is the global quotient
corresponding to a group acting on a manifold.

\begin{example} 
An {\em effective global quotient orbifold} (of Lie type) is the
orbifold corresponding to a smooth, effective, and almost free action
of a compact Lie group $G$ on a smooth manfold $M$.  We will write $M
/ G$ to denote a global quotient orbifold.  Here the orbifolds charts
are constructed from the charts of $M$ and the slices of the
action.
\end{example}

We have a natural variant of this in the context of an orbifold with
corners.

\begin{example}
An {\em effective global quotient orbifold with corners} (of Lie type)
is the orbifold with corners corresponding to a smooth, effective, and
almost free action of a compact Lie group $G$ on a smooth manifold
with corners $M$ by smooth maps of manifolds with corners.  We again
write $M / G$ to denote the global quotient orbifold with corners.
\end{example}

There are many different group actions that can present a given
orbifold.  An immediate consequence of the description of the atlas
for the orbifold $M / G$ is the following comparison result.

\begin{prop}
Suppose that we are given a subgroup $G_1 \subset G_2$ such that $G_1$
acts smoothly, effectively, and almost freely on a manifold $M$.  Then
there is an isomorphism of orbifolds
\begin{equation}
M / G_1 \cong (M \times_{G_1} G_2) / G_2.
\end{equation} \qed
\end{prop}

Defining maps of orbifolds in the context of charts and atlases is in
general subtle; the simplest notion is a smooth map of orbifolds,
but as we discuss below this is not adequate for most purposes.

\begin{defin}
A (smooth) map of orbifolds $(X, A) \to (Y, A')$ is a continuous map
$f \colon X \to Y$ that for each point $x \in X$ there are charts
$(G_x, U_x, \phi_x)$ and $(G_{f(x)}, U_{f(x)}, \phi_{f(x)})$ around
$x$ and $f(x)$ respectively such that $f(\phi(U_x)) \subset
\phi_{f(x)}(U_{f(x)})$ and $f$ lifts to a (smooth) map $U_x \to
U_{f(x)}$.  Similarly, a (smooth) map of orbifolds with corners is a
(smooth) map of orbifolds such that the lifts are to smooth maps $U_x
\to U_{f(x)}$ of manifolds with corners.
\end{defin}

The natural notion of stabilization for a derived orbispace involves a
vector bundle; as in Definition~\ref{def:orbifold-chart} and
Definition~\ref{def:orbifold-chart-with-corners}; this is specified in
terms of charts.

\begin{defin} \label{def:vect-bundl-orbif}
A vector bundle on an effective orbifold (with corners) $X$ with atlas $A$ is
specified by giving a map of orbifolds $E \to X$ along with a maximal
atlas of bundle charts, where a bundle chart is specified by a chart
of $(G_i, U_i, \phi_i)$ of $X$ and a $G_i$-equivariant vector bundle
$\zeta_i \colon E_i \to U_i$ compatible with the structure map $E \to
X$.
\end{defin}

Vector bundles on effective orbifolds do not pull back under smooth
maps of orbifolds; this is the reason for the introduction of good
maps in the sense of Chen and Ruan~\cite[4.4.1]{Chen2002}.  This class
of maps admits well-behaved pullback functors on categories of vector
bundles on orbifolds, as we discuss below.

\begin{defin}\label{defn:good-map-of-orbifolds}
A {\em good map} of orbifolds $(X,A) \to (Y,A')$ is specified by a
smooth map of orbifolds $f$ such that there exist covers $\{U_i\}$ of
$X$ and $\{U'_i\}$ of $Y$ with the property that the indices specify a bijection
between the elements of the covers so that:
\begin{enumerate}
\item the bijection is compatible with $f$ in the sense that $f(U_i)
  \subset U'_i$,
\item the bijection is a map of posets (for the partial order by
  inclusion), and
\item there exists a collection of smooth lifts of $f$ to each chart
  $U_i$ such that there are corresponding injections of charts that
  are compatible with the lifts.
\end{enumerate}
This definition carries over without modification to the setting of orbifolds with corners.
\end{defin}

\begin{prop} \label{prop:pullback_good_maps}
Let $f\colon (X,A) \to (X',A')$ be a good map of effective orbifolds
(with corners).  Then for each vector bundle $E$ on $X'$, there is a
pullback bundle $f^* E$ on $X'$ and this assignment specifies a
functor.\qed
\end{prop}

Diffeomorphisms of orbifolds provide examples of good maps.

\begin{defin}
A diffeomorphism of orbifolds (with corners) from $(X, A)$ and $(Y,
A')$ is specified by a pair of smooth maps of orbifolds $f \colon
(X,A) \to (Y,A')$ and $g \colon (Y,A') \to (X,A)$ such that $f \circ g
= \id_Y$ and $g \circ f = \id_X$.
\end{defin}

The following basic result allows us to understand effective orbifolds
by studying (effective) global quotient orbifolds.

\begin{prop}
Any effective orbifold is diffeomorphic to an effective global
quotient orbifold.  Any effective orbifold with corners is
diffeomorphic to an effective global quotient orbifold with
corners. \qed
\end{prop}

\section{Hamiltonian Floer theory}
\label{sec:hamilt-floer-theory-1}

We now explain how the work of Bai-Xu \cite{BaiXu2022} and Rezchikov \cite{Rezchikov2022}, extending the results of \cite{Abouzaid2021b} from Gromov-Witten to Hamiltonian Floer theory, combined with some techniques from \cite{Abouzaid2021a}, implies the following result, where we use $\Flow^U$ to denote the category of complex oriented derived orbifold flow categories, and $\Flow^U_+$ for its subcategory consisting of bimodules on which the action is non-negative.
\begin{prop} \label{prop:Hamiltonian_Flow_category}
  Associated to a Hamiltonian function with non-degenerate time-$1$ orbits on a closed symplectic manifold $M$ is a complex oriented flow category satisfying the following properties:
  \begin{enumerate}
  \item the equivalence class in  $\Flow^U_+$ only depends on $H$,
  \item  each smooth path of Hamiltonians parametrised by an interval, with the property that the endpoint Hamiltonians are non-degenerate, can be lifted to a morphism $\Flow^U $, in such a way that the constant family lifts to an equivalence in  $\Flow^U_+$, and
    \item   A homotopy between a concatenation of paths and a smooth path determines a $2$-simplex in $\Flow^U$. 
  \end{enumerate}
\end{prop}
Since the space of Hamiltonians is path connected, we conclude:
\begin{cor}
 Hamiltonian Floer theory associates an object of $\Flow^U$ to each closed symplectic manifold, whose equivalence class is independent of the choice of Hamiltonian.
\end{cor}

 As discussed in the introduction, and proved in Section \ref{sec:fram-flow-categ} below, when the target symplectic manifold is symplectically aspherical (in particular, there are no $J$-holomorphic spheres), a lift to the orthogonal group of the monodromy map  $\Omega M \to U$ on the tangent space of $M$ determines a lift to a framed flow category.

To prove Proposition \ref{prop:Hamiltonian_Flow_category}, we begin by fixing a Hamiltonian
\begin{equation}
 H \co M \times S^1 \to \bR 
\end{equation}
and an almost complex structure  $J$ on $M$, requiring the non-degeneracy condition that the return map associated to any closed time-$1$ Hamiltonian orbit of $H$ not have $1$ as an eigenvalue. This implies that the set of such Hamiltonian orbits is isolated, hence finite.

The essential construction in Floer theory is that of the moduli spaces $\Mbar(p,r)$ of stable solutions to Floer's equation associated to these data. These are equivalence classes of finite-automorphism maps
\begin{equation}
  u \co \Sigma \to M  
\end{equation}
whose domain is the complement of two marked points and the nodes separating them in a pre-stable Riemann surface, and so that all the components lying in the chain connecting these marked points are solutions to Floer's equation 
\begin{equation}
  (du - X_H \otimes dt)^{0,1} = 0 
  \end{equation}
  with asymptotic conditions given by time-$1$ periodic orbits of the Hamiltonian vector field $X_H$, which at the two marked points are $p$ and $r$, while the components not lying on this chain are $J$-holomorphic curves.

  These are the morphism spaces of a (non-unital) topologically enriched category; our goal is to realise these morphism spaces as (the coarse spaces) of zero-loci of morphism spaces in a flow category, and equip this flow category with a stable complex structure. The unstructed lift is already part of the results of \cite{Rezchikov2022} and \cite{BaiXu2022} so we begin by presenting a minor variant of their construction. Afterwards, we adapt the construction of stable complex structures from \cite{Abouzaid2021a} to the setting of global charts.

\subsection{Preliminary constructions}
\label{sec:prel-constr}

\subsubsection{Framed curves in projective space}
\label{sec:fram-curv-proj-1}

A key issue in Floer-theoretic constructions is that the domains of curves representing elements of moduli spaces such as $\Mbar(p,r)$ have nontrivial automorphisms; any geometric construction on the moduli spaces then entails breaking these symmetries. The classical approach, following \cite{Fukaya1999}, is to break these symmetries using choices of hypersurfaces in the target $M$, which are assumed to be transverse to the chosen curve; this is only possible locally in the moduli space.

One of the fundamental insights of \cite{Abouzaid2021b} is that any such domain is realised as the domain of a curve with target a complex projective space, and that the moduli space of such curves can serve, when the degree is large enough, as a parameter space for the space of maps with target any given symplectic manifold, as long as redundancies are appropriately accounted for. The starting point is the classical description of the space of holomorphic maps from $\bC \bP^1$ to $\bC \bP^d$ of degree $d$: the natural action of $PGL_{d+1}(\bC)$ on this space is transitive, and the stabiliser of any point is a copy of $PGL_2(\bC)$.

Let $* = [0, \ldots, 1]$, and define  $\Fbar_2(d,*)$ to be the open subset of the moduli space of stable maps to $\bC \bP^d$  of degree $d$ and genus $0$ with $2$ marked points (labelled $z_\pm$) so that $z_-$ maps to $*$, consisting of curves whose image is not contained in any proper linear subspace. As in \cite[Lemma 6.4]{Abouzaid2021b}, this is a smooth manifold, carrying an action of the unitary group $U(d)$, embedded as a subgroup of $PGL_{d+1}(\bC)$ as the image of those unitary matrices fixing the last basis vector of $\bC^{d+1}$. We write $\cF_2(d,*)$ for the open submanifold consisting of curves with the property that every node separates the two marked points, carrying the restricted $U(d)$ action; the interior of $ \cF_{2}(d,*)$ is diffeomorphic to the quotient of $PGL_{d+1}(\bC) $ by $\bC^*$.

Let $\Fbar^{\bR}_2(d,*)$ be the $U(d)$-manifold consisting of a curve in  $\Fbar_2(d,*)$, a choice of real line in the tangent space of the point labelled $z_+$, and a choice of a real line in $T_z \Sigma_+ \otimes_{\bC} T_z \Sigma_-$ for each node $z$ which separates the marked points $z_\pm$.  We denote by $ \cF^{\bR}_2(d,*)$ the open subset of  $\Fbar^{\bR}_2(d,*)$ lying over $ \cF_2(d,*) $.

\subsubsection{Deforming the energy to be integral}
\label{sec:deforming-energy-be}

The energy defines a (locally constant) proper and bounded below map
\begin{align}
  E \co \Mbar(p,r) & \to (0,\infty) \\
  u & \mapsto \int u^* \omega + \partial_s H(u(s,t)) ds \wedge dt.
\end{align}
We write $\Mbar(p,r)^{E}$ for the union of components of $\Mbar(p,r)$ consisting of elements of energy $E$. For geometric applications, this is the fundamental quantity that is associated to each moduli space. However, our method for lifting $ \Mbar(p,r) $ to a derived orbifold require us to deform $E$ to an integrally valued function.

By considering a $2$-form $\Omega$ representing a large integral multiple of a rational approximation to $\omega$, we can assign an integral energy map
\begin{equation}
  u \mapsto \int u^* \Omega \in \bN  
\end{equation}
to each closed $J$-holomorphic curve, which vanishes only for constant curves. This is the starting point of \cite{Abouzaid2021b}, where the form $\Omega$ is assumed to tame $J$, and is interpreted as the curvature of a complex line bundle on $M$ equipped with a Hermitian connection. For the setting of Hamiltonian Floer theory, we follow \cite[Section 3.1]{Rezchikov2022}, and use a cutoff function to ensure that $\Omega$ vanishes near the orbits. In this way, by also choosing a deformation $\bar{H}$ of a large integral multiple of $H$ which is supported near the orbits, we obtain a quantized energy
\begin{align}
  \bar{E} \co \Mbar(p,r) & \to \bN \setminus \{0\} \\
  u & \mapsto \int \tilde{u}^* \Omega_{H} 
\end{align}
where $ \tilde{u}^* \Omega_{H}  = u^* \Omega + \partial_s \bar{H}(u(s,t)) ds \wedge dt$. 
The key result we need is: 
\begin{lem}[Lemma 19 of  \cite{Rezchikov2022}]
  The deformed energy $\bar{E}$ is proper on $\Mbar(p,r)$, and is additive under concatenation of Floer solutions, and bubbling of spheres. \qed
\end{lem}

\begin{rem}
  The Floer complex constructed in \cite{Rezchikov2022} is generated by \emph{contractible periodic orbits,} In this setting, one can assign to each generator a set of action values $\{ \cA(x) \} \subset \bR$ which is a torsor over the periods of $\Omega$ on spheres mapping to $M$, with the property that the energy of any solution to Floer's equation lies in the set of difference between the  possible action values of its asymptotic conditions. Since we do not restrict our attention to contractible orbits, it makes less sense from our point of view to introduce the action.
\end{rem}

We shall write $\Mbar^{(d)}(p,r)$ for the subset of the moduli space $ \Mbar(p,r)$ consisting of those curves whose deformed energy equals $d$.

\subsubsection{Framed stable Floer solutions}
\label{sec:framed-stable-floer}

We now combine the discussions of the previous two sections, by constructing, for each pair of Hamiltonian orbits and each fixed energy level $d$, a topological space equipped with a group action (with finite isotropy), so that the resulting quotient space is $ \Mbar^{(d)}(p,r)$. In other words, we present the orbispace associated to $ \Mbar^{(d)}(p,r)$ as a global quotient.

The relationship between the two constructions is that, given a punctured Riemann surface $\Sigma$ which is the domain of an element of $\Fbar^{\bR}_2(d,*)$, and a map  $u \co \Sigma \to M$ representing a point in $\Mbar(p,r)$, there are two ways to produce a holomorphic line bundle on $\bar{\Sigma}$:
\begin{enumerate}
\item the pullback of $\cO(1)$ on projective space, and
  \item the line bundle with connection $2$-form the pullback of $ -2 \pi  \Omega_H$. 
\end{enumerate}

If, on each component of $\Sigma$, the degree of the map to $\bC \bP^d$ associated to the element in $\Fbar^{\bR}_2(d,*)$ agrees with the energy, these two line bundles are isomorphic; the space of such isomorphism inherits a free transitive action of $\bC^*$, arising from dilating either line bundle.

\begin{defin}
  A \emph{unitarily framed stable Floer solution} consists of  (i) a curve $\Sigma$ representing a point in $ \Fbar^{\bR}_2(d,*)$, and (ii) a stable solution $u \co \Sigma \to M$ to Floer's equation so that the quantised energy of each component agrees with the degree of the corresponding map to $\bC \bP^d $, and such that, under any isomorphism between $\cO(1)$ and the line bundle associated to $\Omega_H$, the standard basis of $\cO(1)$ is projectively unitary with respect the pairing defined by $\Omega_H$.\end{defin}

\begin{lem}[Section 3.2.3 of \cite{Rezchikov2022} or Section 5.3.1 of \cite{BaiXu2022}] \label{lem:framed-curves-quotient}
  The moduli space $\Mbar^{(d)}(p,r)$ is homeomorphic to the quotient,  by the action of $ U(d) $,  of the space of unitarily framed stable Floer solutions of deformed energy $d$. The stabiliser of any point is finite, and is naturally isomorphic to the isotropy of the corresponding stable Floer solution. \qed
\end{lem}

Note that we cannot a priori determine the image of the projection map from the space of unitarily framed stable Floer solutions to $\Fbar^{\bR}_2(d)$, so it is convenient to introduce the space of  \emph{framed stable Floer solutions}, in which the unitary condition on the pullback of the standard basis is dropped. Because we assume that the basepoint $z_-$ maps to $[0, \ldots, 0, 1]$, the space of such framings admits a natural action of the subgroup of $GL_{d+1}(\bC)$ consisting of matrices whose last column is $(0, \ldots, 1)$. The maximal compact subgroup of this Lie group is $U(d)$,  so that the failure of the basis to be unitary is measured by a map to its normal bundle
\begin{equation} \label{eq:normal-b-spaces}
  N_{d} =   i\mathfrak{u}_{d} \oplus \bC^d.
\end{equation}

\subsubsection{The thickening}
\label{sec:thickening}

For the next construction, we consider a complex finite dimensional $U(d)$ representation $\scrV^{(d)}(p,r)$, equipped with an invariant inner product, together with an equivariant map $\lambda$ from $\scrV^{(d)}(p,r)$ to the (infinite dimensional) vector space whose elements are compactly supported smooth families of $(0,1)$-forms on the Riemann surfaces representing elements of $\Fbar^{\bR}_2(d,*)$, valued in the space of vector fields on $M$ which vanish in the neighbourhood of the Hamiltonian orbits where the $2$-form $\Omega$ vanishes.

\begin{defin} \label{def:thickening}
The \emph{thickening} $\cT^{(d)}(p,r)$ is the space of pairs $(u,v)$, with $u$ a map to $M$ with domain a curve $\Sigma$ representing a point in $ \Fbar^{\bR}_2(d,*)$, and $v$ is an element of $\scrV^{(d)}(p,r)$ so that the following conditions hold:
\begin{enumerate}
\item $u$ represents a stable finite energy solution to the pertubed Floer equation
\begin{equation} \label{eq:perturbed_equation}
(du - X_H \otimes dt)^{(0,1)}  - \lambda(v)  = 0,
\end{equation}
\item the energy of each component of $\Sigma$ agrees with their degree in projective space,
\item the matrix $\int \langle x_i, x_j \rangle_{\Omega_H} \tilde{u}^* \Omega_{H} $ of pairings between the elements of the standard basis of $\cO(1)$ is positive definite, 
\item the linearisation of Equation \eqref{eq:perturbed_equation} is surjective, and
\item the asymptotic conditions along the ends of $\Sigma$ are the orbits $p$ at $z_-$ and $r$ at $z_+$.
\end{enumerate}
\end{defin}

Note that the thickening depends on the choice of the map $\lambda$ 
which is not included in the notation. The action of $U(d) $ on $\Fbar^{\bR}_2(d,*) $ lifts to this space, and the stabiliser of any map is therefore constant. The following result is a consequence of standard transversality and gluing results, and is a variant of the analogous result in \cite{Abouzaid2021b} for the case of genus $0$ Gromov-Witten theory.
\begin{prop}[Proposition 3 of  \cite{Rezchikov2022} or Proposition 6.22 of \cite{BaiXu2022}]
The thickening $\cT^{(d)}(p,r)$ is a topological manifold with boundary, which, assuming that the image of $\lambda$ is sufficiently large, contains the space of unitarily framed Floer solutions, and satisfies the property that the projection to $ \Fbar^{\bR}_2(d,*)$ is a $U(d)$-equivariant topological submersion which is naturally equipped with an equivariant fibrewise smooth structure. \qed
\end{prop}
Using Lashof's equivariant extension of smoothing theory \cite{Lashof1979}, which can be extended to manifolds with corners by a direct argument \cite[Appendix B]{BaiXu2022} or by applying a generalisation of the doubling procedure for manifolds with boundary \cite[Section 4.2]{Rezchikov2022}, we conclude:
\begin{cor} \label{cor:smooth_uniqut_up_to_concordance}
  Up to replacing $\scrV^{(d)}(p,r)$ by a larger representation, the thickening $\cT^{(d)}(p,r)$ admits a $U(d)$-equivariant smooth structure which is canonical up to concordance. \qed
\end{cor}
We write $X^{(d)}(p,r)$ for the orbifold quotient of  $\cT^{(d)}(p,r)$ by $U(d)$.

\subsubsection{The global chart}
\label{sec:global-chart}

The direct sum of the  $U(d)$-representations
\begin{equation} \label{eq:obstruction-bundle}
\scrV^{(d)}(p,r) \oplus N_{d}
\end{equation}
defines a  $U(d)$-equivariant vector bundle over  $\cT^{(d)}(p,r)$, and hence a vector bundle over $X^{(d)}(p,r)$.
\begin{defin}[Definition 12 of \cite{Rezchikov2022} or Definition 5.33 of \cite{BaiXu2022}]
  The global chart $\bX^{(d)}(p,r)$ of stable Floer cylinders of energy $d$ is the derived orbifold with total space $X^{(d)}(p,r) $, obstruction bundle given by Equation \eqref{eq:obstruction-bundle}, and defining section which is  the direct sum of the projection to $v$ with the map which assigns to each framed curve the corresponding element of $  N_{d} $ measuring the failure of the framing to be unitary.
\end{defin}

Lemma \ref{lem:framed-curves-quotient} implies that this is global chart for the moduli space of stable Floer cylinders, with asymptotic conditions $p$ and $r$ and (quantised) energy $d$, in the sense that the coarse space of the zero locus of the defining section is $\Mbar^{(d)}(p,r)$.

\subsection{Unstructured flow category}
\label{sec:unstr-flow-categ-1}

The moduli space of stable Floer cylinders has a stratification with the codimension $k$ strata of $\Mbar^{(d)}(p,r)$ given by the images of natural maps
\begin{equation}
   \Mbar^{(d_1)}(p,q_1) \times    \Mbar^{(d_2)}(q_1,q_2) \times \cdots \times  \Mbar^{(d_k)}(q_{k-1},r)
\end{equation}
with $\sum d_i = d$. The thickenings constructed in Section \ref{sec:global-chart} do not have such a stratification because the representations $\scrV^{(d)}(p,r)$ are chosen independently.

\subsubsection{Stratification of framed curves}
\label{sec:strat-fram-curv}

Define $\bar{B}_d$ to be the smooth stack quotient of the $U(d)$ action on  $\Fbar^{\bR}_2(d,*)$. One begins by constructing a map
\begin{equation} \label{eq:map_bases}
   \bar{B}_{d_1} \times \bar{B}_{d_2} \to \bar{B}_d 
\end{equation}
as follows: given a curve representing an element of $\bar{B}_{d_2}$, use parallel transport (along the line distinguished by the framing at the output marked point), to induce a trivialisation of the line over the input marked point, from the standard trivialisation of the last factor of $\bC^{d_2 +1}$, then define a linear map of projective spaces
\begin{equation}
  \bC \bP^{d_1} \to   \bC \bP^{d}
\end{equation}
which at the level of affine spaces is given by the inclusion of the first $d_1$ coordinates, and the map which takes the last factor of $\bC^{d_1 + 1}$ to the above line (thought of as a subspace of $\bC^{d+1}$ by the inclusion of $\bC^{d_2 +1}$ as the last coordinates). This induces a map
\begin{equation}
   \Fbar^{\bR}_2(d_1,*) \times  \Fbar^{\bR}_2(d_2,*) \to    \Fbar^{\bR}_2(d,*) 
\end{equation}
since, by construction, the output marked point in the first curve and the input marked point in the second curve map to the same point in $  \bC \bP^{d} $, and the construction descends to the quotient by unitary groups since it is equivariant with respect to the inclusion homomorphism
\begin{equation}
  U(d_1) \times U(d_2) \to U(d)
\end{equation}
of block matrices.

This map takes value in the codimension-$1$ boundary stratum of the moduli spaces $ \Fbar^{\bR}_2(d,*)$, which is obtained by oriented real blowup of the divisor of  $\Fbar_2(d,*)$ associated to a single node that separate the two marked points, with the property that the degrees of the two components of the domain which are separated by the node are given by $(d_1,d_2)$. 
This divisor is a $\bC^{d_1 \times d_2}$ bundle over the locus where the two subspaces are orthogonal; this subspace consists of a single orbit of the $U(d)$ action, and the inclusion of Equation \eqref{eq:map_bases} identifies it diffeomorphically with
\begin{equation}
     \left(\bar{B}_{d_1} \times \bar{B}_{d_2}\right) \times_{U(d_1) \times U(d_2)} U(d). 
\end{equation}

We conclude that the boundary stratum associated to splitting into two components is the total space of a complex vector bundle over $\bar{B}_{d_1} \times \bar{B}_{d_2}$. Given a triple $d_1 + d_2 + d_3 = d$, it is then straightforward to see that the diagram
\begin{equation}
  \begin{tikzcd}
    \bar{B}_{d_1} \times \bar{B}_{d_2} \times \bar{B}_{d_3} \ar[r] \ar[d] &  \bar{B}_{d_1 + d_2} \times \bar{B}_{d_3} \ar[d] \\
     \bar{B}_{d_1} \times \bar{B}_{d_2 + d_3} \ar[r] &  \bar{B}_{d} 
  \end{tikzcd}
\end{equation}
commutes. Writing $\partial^{d_1,d_2,d_3} B_{d}$ for the corner stratum of $B_{d}$ associated to such a partition, we find that the projection maps
\begin{equation}
    \partial^{d_1,d_2,d_3} \bar{B}_{d} \to \bar{B}_{d_1} \times \bar{B}_{d_2} \times \bar{B}_{d_3} 
\end{equation}
obtained from the two factorisations agree, as do the associated complex vector bundle structures. Proceeding inductively, we conclude:
\begin{lem}[Section 5.2.5 of \cite{BaiXu2022}]
  The assignment $d \mapsto \bar{B}_d$ extends to a lax monoidal functor from the discrete symmetric monoidal category on the natural numbers to the category of smooth stacks. The corner stratum $\partial^{d_1, \ldots, d_k} \bar{B}_d $  labelled by a paritition $d_1 + \cdots + d_k = d$ is the vector bundle on $   \bar{B}_{d_1} \times \cdots \times \bar{B}_{d_k}$ associated to the representation  $\mathfrak{u}_{d}/ \bigoplus_{i} \mathfrak{u}_{d_i}$ of $U(d_1) \times \cdots \times U(d_n)$.  The $0$-section of this vector bundle is given by the image of the associated map
  \begin{equation}
        \bar{B}_{d_1} \times \cdots \times \bar{B}_{d_k} \to \bar{B}_d.
  \end{equation}  \qed
\end{lem}

We also consider the spaces $B_d \subset \bar{B}_d$ obtained by restricting to the submanifold  $ \cF^{\bR}_2(d,*)$ consisting of chains of rational curves. Since all points with nontrivial isotropy lie away from this submanifold, this quotient is smooth, and, following from the discussion in Section \ref{sec:fram-curv-proj-1} is moreover contractible. This implies that each corner stratum of $B_d$ is again contractible.

\subsubsection{Consistent choice of inhomogeneous terms}
\label{sec:cons-choice-inhom}

The representations $\scrV^{(d)}(p,r)$ that enter in the definition of the thickening in Section \ref{sec:thickening} can be interpreted as a vector bundle over $\bar{B}_d$, for which we will use the same notation. Since every representation of $U(d)$ appears as a subrepresentation of a representation restricted from $U(d')$, we can therefore arrange, by induction on $d$, to have an inclusion of complex vector bundles
\begin{equation}
\scrV^{(d_1)}(p,q) \oplus \scrV^{(d_2)} (q,r) \to \scrV^{(d_1+d_2)}(p,r)|_{\bar{B}_{d_1} \times \bar{B}_{d_2}},
\end{equation}
so that the following diagram commutes for every quadruple of orbits:
\begin{equation} \label{eq:consitence_inhomogeneous_term}
  \begin{tikzcd}
  \scrV^{(d_1)} (p,q) \oplus \scrV^{(d_2)}(q,q') \oplus   \scrV^{(d_2)}(q',r)\ar[r] \ar[d] &   \scrV^{(d_1+d_2)}(p,q')|_{\bar{B}_{d_1} \times \bar{B}_{d_2} }  \oplus   \scrV^{(d_2)}(q',r)  \ar[d] \\
   \scrV^{(d_1)} (p,q)\oplus \scrV^{(d_2+d_3)}(q,r)|_{ \bar{B}_{d_2} \times \bar{B}_{d_3}  }  \ar[r] &   \scrV^{(d_1+d_2+d_3)}(p,r)|_{ \bar{B}_{d_1} \times \bar{B}_{d_2} \times \bar{B}_{d_3}} .
  \end{tikzcd}
\end{equation}

Moreover, the choices of inhomogeneous terms $\lambda$ give rise to maps from these vector bundles to an infinite dimensional stack over $\bar{B}_d$ whose values at a point are $(0,1)$-forms on the associated Riemann surface (this is not quite an infinite dimensional bundle because the topology of the Riemann surfaces changes). It is then clear that we can restrict the map $\lambda$ on $  \scrV^{(d_1+d_2)}(p,r)$ to $\bar{B}_{d_1} \times \bar{B}_{d_2} $.
\begin{defin}
A \emph{consistent choice} of inhomogeneous terms is a choice of maps $\lambda$ for all triples $(d,p,r)$, so that the restriction of the inhomogeneous term on $\scrV^{(d)}(p,r)$ to $\bar{B}_{d_1} \times \bar{B}_{d_2} $ for all curves with asymptotic conditions $q$ at the node separating the two marked points, vanishes on the orthogonal complement of $\scrV^{(d_1)}(p,q) \oplus \scrV^{(d_2)}(q,r)$. 
\end{defin}
The existence of such consistent choices is proved in \cite[Lemma 28]{Rezchikov2022}. 

\subsubsection{Topological flow category}
\label{sec:topol-flow-categ}

Given a consistent choice of inhomogeneous terms, the boundary stratum $\partial^{(d_1,d_2)}_{q}  X^{(d)}(p,r) $  associated to an orbit $q$ and a partition $d_1 + d_2 = d$ consists of solutions to the same equation (on the same domains) as elements of the product of orbifolds
\begin{equation}
     X^{(d_1)}(p,q) \times  X^{(d_2)}(q,r).
\end{equation}
Identifying $   \scrV^{(d)}(p,r)|_{\partial^{(d_1, d_2)} \bar{B}_{d}} $ as a vector bundle over $ \bar{B}_{d_1} \times \bar{B}_{d_2}$ (with fibre the direct sum of $ \scrV^{(d)}(p,r)$  with the normal bundle of  $ \bar{B}_{d_1} \times  \bar{B}_{d_2}$ in $\partial^{(d_1, d_2)} \bar{B}_{d}$), we obtain a surjection of complex of vector bundles
\begin{equation}
   \scrV^{(d)}(p,r)|_{\partial^{(d_1, d_2)} \bar{B}_{d}}  \to \scrV^{(d_1)}(p,q) \oplus   \scrV^{(d_2)}(q,r)
 \end{equation}
 over $ \bar{B}_{d_1} \times \bar{B}_{d_2}$. From this, we conclude:
 \begin{lem}
   The stratum $  \partial^{(d_1,d_2)}_{q} X^{(d)}(p,r)$ is the total space of a complex vector bundle over $   X^{(d_1)}(p,q) \times   X^{(d_2)}(q,r)$. Moreover, given a pair of orbits $(q,q')$ and a partition $d_1 + d_2 + d_3 = d$, the two induced complex vector bundles structures on codimension $2$ strata
   \begin{equation} \label{eq:compatibility_codimension_2_strata}
     \begin{tikzcd}
   \partial^{(d_1,d_2,d_3)}_{q,q'}     X^{(d)}(p, r)  \ar[r] \ar[d] &   X^{(d_1)}(p,q) \times \partial^{(d_2,d_3)}_{q'}   X^{(d_2+d_3)} (q, r)\ar[d]  \\
   \partial^{(d_1,d_2)}_{q}       X^{(d_1+d_2)} (p,q')\times   X^{(d_3)}(q',r) \ar[r] &    X^{(d_1)}(p,q) \times   X^{(d_2)}(q,q') \times  X^{(d_3)}(q,r)
     \end{tikzcd}
   \end{equation}
agree.
   \qed
\end{lem}

From the above, we obtain, by associativity, that all the strata of $ X(p,r)$ are total spaces of vector bundles over the product of orbifolds for the given label, and that the inclusions of adjacent strata give rise to inclusions of vector bundles. These vector bundles are naturally isomorphic to the quotient of the obstruction bundles, yielding a commutative diagram in the category of derived orbifolds.

\subsubsection{Collared completions}
\label{sec:collared-completions}

The smooth structures arising from Corollary \ref{cor:smooth_uniqut_up_to_concordance} are not necessarily compatible with restriction to boundary strata, in the sense that the smooth structure on $\partial_{q}^{(d_1, d_{2})}  X^{(d)}(p,r)$ inherited from $X^{(d)}(p,r)$ may not agree with its smooth structure arising from its description as the total space of a vector bundle over $   X^{(d_1)} (p,q)\times   X^{(d_2)}(q,r)$. However, the fact that the vector bundle lifts of the tangent microbundle that give rise to these two smooth structure are the same, implies that there is a smooth structure on the product of  $ \partial_{q}^{(d_1, d_{2})}  X^{(d)}(p,r) $ with an interval which restricts to these two structures at the endpoints. This implies that the union of $X^{(d)}(p,r) $ with a collar of $ \partial_{q}^{(d_1, d_{2})}  X^{(d)}(p,r) $
has a smooth structure which restricts, on the new boundary stratum, to its given smooth structure as the total space of a vector bundle over $X(p,q)^{(d_1)} \times   X(q,r)^{(d_2)}$.

This process can be implemented, inductively over $d$, for all the moduli spaces, by attaching to $ X^{(d)}(p,r)$ the product of each codimension $k$ stratum with $[0,1]^k$; this is the \emph{collared completion} $\hat{X}^{(d)}(p,r)$.
\begin{prop}[Section 4.4 of \cite{Rezchikov2022} and Section 6.4 of \cite{BaiXu2022}]
  There are structures of  smooth orbifolds with corners on the collared completions $\hat{X}^{(d)}(p,r)$, with codimension $1$ boundary strata labelled by an orbit $q$ and a decomposition $d = d_1 + d_2$, so that these strata are the total spaces of smooth complex vector bundles
  \begin{equation}
    \partial^{(d_1,d_2)}_{q} \hat{X}(p, r)^{(d)}  \to      \hat{X}^{(d_1)}(p,q) \times   \hat{X}^{(d_2)}(q,r)
      \end{equation}
      These maps are associative in the sense that the resulting squares for codimension $2$ strata commute.
\end{prop}

It is straightforward to extend the vector bundle $T^- \bX^{(d)}(p,r)$, as well as the defining section of the derived orbifold $ \bX^{(d)}(p,r)$  to $ \hat{X}^{(d)}(p,r)$, so we obtain a derived orbifold $\hat{\bX}^{(d)}(p,r)$, whose zero-locus is a completion of $\Mbar^{(d)}(p,r)$. This completion is compatible with the structure maps, so we conclude:

\begin{prop}[Theorem 4.8 of \cite{BaiXu2022} and Lemma 40 of \cite{Rezchikov2022}] \label{prop:unstructured_flow_category}
  The collection of derived orbifolds  $\{ \hat{\bX}^{(d)}(p,r)\}_{0 < d}$ are the underlying morphisms of a flow category  with object the time-$1$ orbits of $H$. \qed
\end{prop}

Having constructed a flow category, we recall that the deformed energy used in the definition was an artificial choice, which is unrelated to geometric quantities that are of relevance in symplectic topology. This leads us to define $ X(p,r)^{E}$ to be the union of components of $\coprod_{d} X^{(d)}(p,r)$  consisting of curves with topological energy $E$, and $ \bX(p,r)^{E} $ to be the corresponding derived orbifold.
\begin{cor}
   The collection of derived orbifolds  $\{ \hat{\bX}(p,r)^{E}\}_{0 < E}$ are the underlying morphisms of a flow category with object the time-$1$ orbits of $H$, depending on the choice of the form $\Omega_H$ and the inhomogeneous terms $\lambda$. \qed
\end{cor}
To be completely precise, the flow category also depends on our appeal to Lashof's work to produce a smooth structure and a concordance; we will not forget to address this below.

\subsubsection{Continuation equations}
\label{sec:cont-equat}
With the goal of proving the unstructured version of Proposition \ref{prop:Hamiltonian_Flow_category}, consider a pair $H_-$ and $H_+$ of non-degenerate Hamiltonian functions, as well as  almost complex structures  $J_-$ and $J_+$  used to define associated flow categories $\hat{\bX}_{-}$ and $ \hat{\bX}_{+}$. Choose as well a  smooth path $J_{+-}$ defined on $\bR$, interpolating between $J_-$ near $-\infty$ and $ J_+$ near $+\infty$.

We choose an arbitrary path $H_{+-}$ of time-dependent Hamiltonians interpolating between $H_-$ for $s \ll 0$  and $H_+$ for $0 \ll s$. The pair $(H_{+-}, J_{+-})$  determines moduli spaces $\Mbar_{-+}(p_-,r_+)$ of solutions to the continuation equation for each Hamiltonian orbit $p_-$ of $H_-$ and $r_+$ of $H_+$, equipped with associative structure maps
\begin{equation}
    \Mbar_-(p_-,q_-) \times   \Mbar_{-+}(q_-,r_+) \to \Mbar_{-+}(p_-,r_+) \leftarrow    \Mbar_{-+}(p_-,q_+) \times   \Mbar_+(q_+,r_+).
  \end{equation}
There is a natural energy map 
\begin{equation}
  \Mbar_{-+}(p_-,r_+) \to \bR
\end{equation}
which is proper and bounded below and is compatible with structure maps in the sense that, if we write $\Mbar^{E}_{-+}(p_-,r_+)  $  for the subset of the moduli space of energy $E$, the structure maps factor as:
\begin{equation} \label{eq:bimodule_structure_map_solutions}
    \Mbar_{-}(p_-,q_-)^{E_1} \times   \Mbar_{-+}^{E_2}(q_-,r_+) \to \Mbar^{E_1 + E_2}_{-+}(p_-,r_+) \leftarrow    \Mbar_{-+}(p_-,q_+)^{E_1} \times   \Mbar_+(q_+,r_+)^{E_2}.
\end{equation}

If we were given flow categories associated to $(J_+,H_+,\Omega_+)$ and $(J_-, H_-, \Omega_-)$, with the property that $\Omega_-$ and $\Omega_+$ are cohomologous, we could use these forms to define a quantized energy on the moduli spaces of continuation maps. As the next step in the construction of the thickening, we would then introduce the space $\Fbar^{\bR}_{2}(d,*,1)$, consisting of an element of $\Fbar^{\bR}_{2}(d,*)$ and a lift to a stable map with target $\bC \bP^d \times \bC \bP^1$, with the property that the composition with the projection maps $z_-$ to $0 \in  \bC \bP^1 $ and $z_+$  to $\infty$, the resulting rational curve in $ \bC \bP^1$ has degree $1$, and the distinguished real line connecting the marked points $z_\pm$ maps to the positive real axis. The fibre of the projection map $\Fbar^{\bR}_{2}(d,*,1) \to \Fbar^{\bR}_{2}(d,*) $ is homeomorphic to an interval, and an analysis of the boundary strata implies:
\begin{lem} \label{lem:additional_marked_point_conic}
  The space $\Fbar^{\bR}_{2}(d,*,1) $ is the total space of the conic degeneration over $ \Fbar^{\bR}_{2}(d,*)$ with discriminant $\partial   \Fbar^{\bR}_{2}(d,*)$. \qed
\end{lem}
The component projecting nontrivial to $\bC \bP^1$ is then canonically identified with $\bR \times S^1$, and hence can be equipped with the continuation map equation (which is not translation invariant). We leave the remainder of the construction of the thickening of the continuation maps, in this special case, to the reader, as we shall consider a more general construction in the next section.

 Given a representation of $U(d)$, mapping equivariantly to the space of inhomegenous terms for the Cauchy-Riemann operator with target $M$, so that it surjects onto the cokernel of the linearisation of the continuation map equation for all elements of $\Mbar^{(d)}_{-+}(p_-,r_+)$, we obtain the thickening $X^{(d)}_{-+}(p_-,r_+)$ as the orbifold quotient of the space of stable continuation maps, with domains given by elements of  $\Fbar^{\bR}_{2}(d,*,1)$, satisfying the positivity, transversality, and asymptotic conditions from Definition \ref{def:thickening}. This space is equipped with a obstruction bundles as in Equation \eqref{eq:obstruction-bundle}, and we denote by $\bX^{(d)}_{-+}(p_-,r_+)$ the resulting derived orbifold, where we assume that the representation $\scrV^{(d)}_{-+}(p_-,r_+)$  is chosen inductively, together with the data of $U(d_1) \times U(d_2)$ equivariant embeddings
\begin{equation}
  \scrV^{(d_1)}_{-}(p_-,q_-) \oplus \scrV^{(d_2)}_{-+}(q_-,r_+)  \to \scrV^{(d)}_{-+}(p_-,r_+) \leftarrow  \scrV^{(d_1)}_{-+}(p_-,q_+) \oplus \scrV^{(d_2)}_+(q_+,r_+)
\end{equation}
whenever $d_1 + d_2 = d$, and $q_\pm$ is a time-$1$ Hamiltonian orbit of $H_\pm$. These embeddings are chosen so that the (three) analogues of Diagram \eqref{eq:consitence_inhomogeneous_term} commute, and so that the choices of inhomogenous terms which they parametrise agree.

In this way, we obtain natural equivalences of derived orbifolds
\begin{align}
  \bX^{(d_1)}_{-}(p_-,q_-) \times   \bX^{(d_2)}(q_-,r_+) & \to \partial^{(d_1,d_2)}_{q_-}  \bX_{-+}^{(d_1 + d_2)}(p_-,r_+) \\
  \bX^{(d_1)}_{-+}(p_-,q_+) \times   \bX^{(d_2)} _+(q_+,r_+)& \to \partial^{(d_1,d_2)}_{q_+}  \bX^{(d_1 + d_2)}_{-+}(p_-,r_+) .
\end{align}
The maps satisfy the commutativity properties required for $\bX_{-+}$ to define a bimodule over $\bX_+$ and $\bX_-$ in the category of (topological) derived orbifolds. Since the above maps enumerate the codimension $1$ boundary strata of $ \bX_{-+}^{(d)}(p_-,r_+) $, passing to the collared completion, as in Section \ref{sec:collared-completions} yields the flow bimodule $\hat{\bX}_{-+}$, after relabelling all components by their topological energy (rather than the quantized energy used in the construction).

The construction which we just presented will be generalised below to the case where the forms $\Omega_+$ and $\Omega_-$ are no cohomologous. In addition to the expositional justification, the main reason that we introduced this special case is that it includes the case of the constant continuation map: if the paths $(H_{-+}, J_{-+})$ are both constants, the two forms $\Omega_-$ and $\Omega_+$ agree, and all inhomogeneous terms for the continuation equation may be chosen to be those used in defining the Floer data.  
The resulting thickenings $X_{-+}(p,r)$ of the space of continuation maps with endpoints $p$ and $r$, is the same as an element of $X_-(p,r)$, together with a choice of lift of the domain from $ \Fbar^{\bR}_{2}(d,*)$ to $\Fbar^{\bR}_{2}(d,*,1)$. Using Lemma \ref{lem:additional_marked_point_conic}, we conclude:
\begin{lem} \label{lem:geometric_diagonal}
Assuming that all the data for a continuation map from $H$ to itself are given by the data for the Floer equation, the resulting  flow bimodule is the diagonal bimodule. \qed 
\end{lem}

\subsubsection{Doubly framed curves}
\label{sec:doubly-framed-curves}
In general, we have to vary the cohomology classes of the $2$-forms used to define the flow categories. The construction which we shall implement relies on an additional assumption:
\begin{lem} \label{lem:bimodule_pair_of_forms}
  Let $(H_{-+}, J_{-+})$ be continuation data connecting a pair $(H_-, J_-)$ and $(H_+,J_+)$ of Floer data. If $(\Omega_-, \Omega_+)$ are a pair of integral $2$-forms whose cohomology classes are sufficiently close to multiples of $\omega$, then, for any pair $\hat{\bX}_- $ and $\hat{\bX}_+ $ of flow categories associated to the triples $(H_-, J_-, \Omega_-)$ and $(H_+,J_+, \Omega_+)$, the moduli space of continuation maps lifts to a flow bimodule $\hat{\bX}_{-+}$ over  $\hat{\bX}_- $ and $\hat{\bX}_+ $. \qed
\end{lem}
The starting point is to use the proximity of  $(\Omega_-, \Omega_+)$  to multiples of $\omega$ to apply the procedure of Section \ref{sec:deforming-energy-be} to the moduli spaces $  \Mbar_-(p_-,q_-)  $ and $ \Mbar_+(q_+,r_+) $  for both of these forms, yielding maps
\begin{equation} \label{eq:energy_map_two_coordinates}
   \Mbar_-(p_-,q_-) \to \bN  \times \bN  \leftarrow  \Mbar_+(q_+,r_+) 
\end{equation}
whose two components are respectively associated to $\Omega_-$ and $\Omega_+$. This procedure can then be adapted to define a quantized energy map 
\begin{equation} \label{eq:pair-energy-continuation}
  \Mbar_{-+}(p_-,r_+) \to \bN \times \bN
\end{equation}
which is again proper, and is,  similarly to Equation \eqref{eq:bimodule_structure_map_solutions},  compatible with the pairs of quantized energies in Equation \eqref{eq:energy_map_two_coordinates} for the Floer trajectories of $H_{+}$ and $H_-$ under breaking. As before, we write $  \Mbar^{(d_-,d_+)}_{-+}(p_-,r_+)$ for those moduli spaces of quantized energies $(d_-,d_+)$. The key point in achieving positivity of this map is a standard fact in Floer theory that the energy of continuation maps is uniformly bounded below (once the path $H_{-+}$ is fixed), so we may choose a compactly supported $2$-form on $\bR \times S^1$, and add its integral over the component carrying the continuation equation to the energy.

To proceed further, we need the notion of doubly framed curves from \cite{Abouzaid2021b}: define  $\Fbar_{2}(d_-,d_+,*)$ to be the open subset of the moduli space of stable maps to $\bC \bP^{d_-} \times \bC \bP^{d_+}$  of genus $0$ curves of bidegree $(d_-,d_+)$, with $2$ marked points (labelled $z_\pm$ as before) so that $z_-$ maps to a distinguished point, satisfying the property that the projection of the image in either factor is not contained in any proper linear subspace. As before, this is a smooth manifold, carrying an action of the product $U(d_-) \times U(d_+)$.  We have equivariant projection maps
\begin{equation} \label{eq:free_quotient-double-framed}
 \Fbar_2(d_-,*) \leftarrow    \Fbar_2(d_-,d_+,*) \to  \Fbar_2(d_+,*) 
\end{equation}
which are respectively compositions of free $U(d_-)$ and $U(d_+)$ quotients, and smooth fibre bundles with contractible fibres (with fibres $\mathfrak{u}_{d'}$ and $\mathfrak{u}_{d}$). 

Let $\Fbar^{\bR}_2(d_-,d_+,*)$ be the $U(d_-) \times U(d_+) $-manifold consisting of a curve in  $\Fbar_2(d_-,d_+,*)$, a choice of real line in the tangent space of the point labelled $z_+$, and a choice of a real line in $T_z \Sigma_+ \otimes_{\bC} T_z \Sigma_-$ for each node $z$ which separates the marked points $z_\pm$. We further define $ \Fbar^{\bR}_{2}(d_-,d_+,*,1)$ to be the space consisting of an element of $\Fbar_2(d_-,d_+,*)$ together with a degree $1$ map from the domain to $\bC \bP^1$, mapping $z_-$ to $0$ and $z_+$ to $\infty$, and the real line between them to the positive real axis.

Given a pair $(p_-,r_+)$ consisting of Hamiltonian orbits of $H_-$ and $H_+$, consider a choice $\scrV^{(d_-,d_+)}_{-+}(p,r)$ of a $U(d_-) \times U(d_+)$ representation together with an equivariant map to the space of inhomogeneous perturbations of the constant continuation map equation on the domains of curves represented by elements of  $\Fbar^{\bR}_2(d_-,d_+,*,1)$, with the property that the image of this map surjects onto the cokernel of the linearisation of the continuation map at every solution. By extending Definition \ref{def:thickening}, we obtain a thickening $\cT_{-+}^{(d_-,d_+)}(p_-,r_+)$ of the moduli space of solutions to the constant continuation map equation, and hence, as in Section \eqref{sec:global-chart},  a global chart which we denote $ \bX_{-+}^{(d_-,d_+)}(p_-,r_+)$ .

In order for this global chart to give rise to a flow bimodule, we need to choose the representation  $\scrV^{(d_-,d_+)}_{-+}(p_-,r_+)$ consistently with the choices made to construct flow categories associated to the pairs $(H_-,J_-)$ and $(H_+,J_+)$. To state the compatibility, it is convenient as before to pass to the stack quotient, and write $\bar{B}_{d_-,d_+}(1)$ for the quotient of the action of  $U(d_-) \times U(d_+) $ on $ \Fbar^{\bR}_{2}(d_-,d_+,*,1)$: this has two types of codimension $1$-boundary strata strata, corresponding respectively to the cases where a component whose projection map has degree $0$ lies over $0$ or $\infty$ in $\bC \bP^1$. The fact that the two maps in Equation \eqref{eq:free_quotient-double-framed} are free quotients can readily be used to see that these two strata are each determined by choices of decompositions $d_- = d'_- + d''_-$ and $d_+ = d'_+ + d''_+$, and are respectively total spaces of complex vector bundles over
\begin{align}
&  \bar{B}_{d'_-} \times  \bar{B}_{d''_-,d''_+}(1)  \\
&  \bar{B}_{d'_-,d'_+}(1) \times \bar{B}_{d''_+}. 
\end{align}
This means in particular that we can impose the condition that the restriction of $\scrV^{(d_-,d_+)}_{1}(p_-,r_+) $, considered as a vector bundle, to these strata respectively admit an embedding of the products
\begin{align}
& \scrV^{(d'_-)}_{+}(p_-,q_-) \oplus  \scrV^{(d''_-,d''_+)}_{-+}(q_-,r_+)  \\
& \scrV^{(d'_-,d'_+)}_{-+}(p_-,q_+) \oplus \scrV^{(d''_+)}_{+}(q_+,r_+)  . 
\end{align}
We thus inductively choose such embeddings, satisfying the usual associativity diagrams for quadruples of orbits (there are three such diagrams, depending on whether the intermediate orbits are both associated to $H_-$, both to $H_+$, or one each). With this choice, we obtain equivalences
\begin{align}
\bX^{(d'_-, d'_+)}_{-}(p_-,q_-) \times  \bX^{(d''_-,d''_+)}_{-+}(q_-,r_+)  \to  & \partial^{ ((d'_-, d'_+), (d''_-,d''_+) )}_{q_-} \bX_{-+}^{(d_-,d_+)}(p_-,r_+)  \\
 \bX^{(d'_-,d'_+)}_{-+}(p_-,q_+) \oplus \bX^{(d''_-, d''_+)}_{+}(q_+,r_+)   \to  & \partial^{ ((d'_-, d'_+), (d''_-,d''_+) )}_{q_+} \bX_{-+}^{(d_-,d_+)}(p_-,r_+)  ,
\end{align}
where $\bX^{(d'_-, d'_+)}_{-}(p_-,q_-) $ is the derived orbifold obtained by restricting to the components of $ \bX^{(d'_-)}_{-}(p_-,q_-) $ with the given quantized energy with respect to $\Omega_+$.

Stabilising by a sufficiently large representation, and applying the collaring construction, as before, to obtain compatible smoothings of (stabilisations) of these derived orbifolds, then pass to a completion $\hat{\bX}_{-+}^{(d_-,d_+)}(p_-,r_+)  $. Relabelling the components by their energy then completes the proof of Lemma \ref{lem:bimodule_pair_of_forms}.

\subsubsection{Composition of continuation maps}
\label{sec:comp-cont-maps}

In the statement of the next result, we have flow categories $\hat{\bX}_-$ and $\hat{\bX}_+$ associated to pairs $(H_-, J_-, \Omega_-)$ and $(H_+,J_+, \Omega_+)$ of Floer data, and flow bimodules $\hat{\bX}_{-+}$ and $\hat{\bX}_{+-}$  between them associated to continuation data $(H_{-+}, J_{-+}, (\Omega_-, \Omega_+))$ and to the inverse path, arising from the construction of Lemma \ref{lem:bimodule_pair_of_forms}. One small complicating factor that we encounter at this stage is that we artificially added a term to the energy of continuation solutions to $H_{-+}$ to achieve positivity of the quantized energies in Equation \eqref{eq:pair-energy-continuation}, and separately for $H_{+-}$ when considering continuation in the opposite direction.

\begin{lem} \label{lem:invariance_continuation}
  The composition of $\hat{\bX}_{-+}$ and $\hat{\bX}_{+-}$ is represented by the diagonal bimodule of $\hat{\bX}_-$.

\end{lem}
\begin{proof}

  We need to construct a $2$-simplex $\hat{\bX}_{-+-}$ with boundary edges $\hat{\bX}_{-+}$, $\hat{\bX}_{+-}$, and the diagonal  bimodule, which, according to Lemma \ref{lem:invariance_continuation} is represented by the constant continuation map with all data pulled back from the projection to the space of Floer trajectories.

At the level of Floer equations, we consider a $1$-parameter family of continuation maps interpolating between the concatenation of  $(H_{+-},J_{+-})$ with  $(H_{-+},J_{-+})$ and the constant continuation map on $(H_{+}, J_{+})$. This determines moduli spaces $\Mbar_{-+-}(p_-,r_-)$ for each pair $(p_-,r_-)$ of Hamiltonian orbits of $H_-$, and these have two additional boundary strata in addition to the ones corresponding to Diagram \eqref{eq:bimodule_structure_map_solutions}: 
\begin{equation}
   \Mbar_{-+}(p_-,q_+) \times   \Mbar_{+-}(q_+,r_-) \to \Mbar_{-+-}(p_-,r_-) \leftarrow    \Mbar_{--}(p_-,r_-),
\end{equation}
where we have assumed that we choose a family of data to make the energy integral, extending the choices for the continuation maps the Floer equations. We now consider a triple of maps
\begin{equation}
    \Mbar_{-+-}(p_-,r_-)  \to \bN \times \bN \times \bN,
\end{equation}
with the first given by the integral of $\Omega_{-,H_-}$, and second two by adding to the integrals of $\Omega_{-,H_-}$ and $\Omega_{+,H_-}$ the areas of the pairs of $2$-forms used to achieve positivity of the pairs of quantised energies for the continuation maps $H_{-+}$ and $H_{+-}$.

At the level of framed curves, we thus introduce the moduli space consisting of triply framed curves together with two marked point, labelled $z_{-+}$ and $z_{+-}$ along the distinguished real line connecting the input from the output, satisfying the constraint that the marked point $z_{-+}$ lies closer to the input than $z_{+-}$. The triple framing corresponds to an action of $U(d_1) \times U(d_2) \times U(d_3)$ on this space, whose stack quotient we denote $\Bar{B}_{d_1,d_2,d_3}(2)$. We again inductively choose representations $\scrV^{(d_1,d_2,d_3)}_{-+-}$ of $U(d_1) \times U(d_2) \times U(d_3)$ and inhomogeneous terms which they parametrise, in order to define a derived orbifold $\hat{\bX}^{(d_1,d_2,d_3)}_{-+-}(p_-,r_-)$ for each triple of natural numbers $(d_1,d_2,d_2)$ and pair $(p_-,r_-)$ of Hamiltonian orbits of $H_-$, and lift the above map of moduli spaces of solutions to strong isomorphisms of derived orbifolds between its codimension $1$ boundary strata and one of the following previously defined derived orbifolds, for some decomposition $d'_i + d''_i = d_i$, and some choice of orbit $q_-$ of $H_-$ or $q_+$ of $H_+$:
\begin{align}
  &  \hat{\bX}^{(d'_1,d'_2,d'_3)}_{-+-}(p_-,q_-) \times  \hat{\bX}^{d''_1, d''_2, d''_3}_{-}(q_-,r_-) \\
  &  \hat{\bX}^{(d'_1,d'_2,d'_3)}_{-}(p_-,q_-) \times  \hat{\bX}^{d''_1, d''_2,d''_3}_{-+-}(q_-,r_-) \\
  &   \left( \hat{\bX}^{(d'_2,d'_3)}_{-+}(p_-,q_+) \times  \hat{\bX}^{(d''_2,d''_3)}_{+-}(q_+,r_-) \right)^{(d_1)} \\
   &  \hat{\bX}^{(d_1,d_2,d_3)}_{--}(p_-,r_-),
\end{align}
where, as before, we take components of the previous defined derived orbifolds according to the additional quantized energy maps which we introduced. The key point in order to achieve this identification of boundary strata  is to choose the inhomogeneous term, which parametrised by the representation $\scrV^{(d_1,d_2,d_3)}_{-+-}$ to depend, along the boundary stratum associated to the constant continuation map, only on a subrepresentation of $U(d_1)$ which is identified with $\scrV^{(d_1)}_{-}$, and to be given by the inhomogenous terms used to define the flow category $\hat{\bX}_-$ along this stratum. We impose a similar constraint, in terms of the group $U(d_2) \times U(d_3)$ along the boundary stratum where the continuation map breaks.

Rearranging as before the components of these derived orbifolds according to their topological energy, we obtain the desired $2$-simplex in $\Flow$, establishing the equivalence.

\end{proof}

\begin{cor}
  All Hamiltonian flow categories associated to a closed symplectic manifold are equivalent.
\end{cor}
\begin{proof}
We can interpolate between any two pairs $(H_-,J_-)$ and $(H_+,J_+)$ by a finite sequence $\{(H_i,J_i)\}_{i=0}^{n}$, with $(H_-,J_-) = (H_0,J_0)$ and $(H_+,J_+) = (H_n, J_n)$ so that there are integral $2$-forms $\Omega_{ij}$ which tame $J_i$ and $J_j$. The result then follows by concatenating the equivalences between the successive flow categories.
\end{proof}

\subsection{Structured flow categories}
\label{sec:struct-flow-categ-1}

Let $\hat{\bX}$ be a Hamiltonian flow category obtained by applying the construction of Section \ref{sec:collared-completions}. The smooth structure on the orbifolds $\hat{X}^{(d)} (p,r)$ has the property that the tangent bundle of $\hat{X}^{(d)} (p,r)$ is naturally isomorphic to the direct sum
\begin{equation}
    T \bar{B}_d \oplus T^{\pi}  \hat{X}^{(d)}(p,r)
\end{equation}
where the second component is the tangent space of the fibre of the projection map to $\bar{B}_d$. In order to lift $\hat{\bX}$ to a structured flow category, we shall separately analyse the orientability of these two vector bundles; this is possible because the above decomposition is compatible with restriction to boundary strata, in the sense that we have a commutative diagram:
\begin{equation}
  \begin{tikzcd}
    T \bar{B}_{d_1} \oplus   T^{\pi}  \hat{X}^{(d_1)}(p,q) \oplus T \bar{B}_{d_2} \oplus  T^{\pi}  \hat{X}^{(d_2)}(q,r)
    \ar[r] \ar[d] &   T  \hat{X}^{(d_1)}(p,q) \oplus  T \hat{X}^{(d_2)}(q,r) \ar[d] \\
     T \bar{B}_d \oplus T^{\pi}  \hat{X}^{(d)}(p,r)  \ar[r] & T  \hat{X}^{(d)}(p,r).
  \end{tikzcd}
\end{equation}

\begin{rem}
  The reader who is concerned about the notion of \emph{tangent space} of $\bar{B}_d$ due to the fact that it is not an orbifold can replace all our argument below by arguments taking place at the level of equivariant vector bundles, at the cost of having to repeatedy have to use the comparison between the tangent space of an $H$ manifold $M$ and that of the $G$ manifolds $M \times_{H} G$ associated to an injection $H \to G$ of compact Lie groups.
\end{rem}

\subsubsection{The tangent space of the base}
\label{sec:orient-tang-space}

Recall that $ \bar{B}_d $ is defined as the $U(d)$ quotient of the manifold $ \Fbar^{\bR}_2(d,*)$, which is the oriented real blowup of the unit sphere bundle obtained of the holomorphic line bundle given by the tangent space at the marked point $z_+$ over the complex $U(d)$-manifold $\Fbar_2(d,*)$. The tangent space of an oriented real blowup along a divisor is equipped with an isomorphism to the pullback of the tangent space of the base, determined up to contractible choice, which extends the tautological identification at the boundary that takes a normal vector to the boundary to its image in the base and the tangent to the circle fibre to its image under the complex structure, so we conclude that the direct sum of $T  \Fbar^{\bR}_2(d,*)$ with a copy of $\bR$ has an (equivariant) complex structure which is canonical up to contractible choice.

Passing to the quotient by the $U(d)$ action implies that $ \bar{B}_d$ has a stable complex structure relative the Lie algebra $\mathfrak{u}_{d}$, i.e., we have a complex vector bundle $I^B_d$ over $ \bar{B}_d $ and an isomorphism of vector bundles
\begin{equation} \label{eq:stable_complex_base}
  T  \bar{B}_d \oplus \mathfrak{u}_{d} \oplus \bR \cong  I^B_d,
\end{equation}
where we abuse notation and write $ \mathfrak{u}_{d} $ for the vector bundle over $\bar{B}_d$ induced by its description as a quotient, and we have fixed the standard trivialisation of the normal bundle of  $ \Fbar^{\bR}_2(d,*)$ in the total space of the tangent line at $z_+$, considered as a holomorphic bundle over $\Fbar_2(d,*)$.

This construction is compatible with compositions as follows: the normal bundle of $\bar{B}_{d_1} \times \bar{B}_{d_2}$ in the boundary stratum $\partial^{d_1, d_2} \bar{B}_d$ is identified with the vector bundle associated to $\mathfrak{u}_{d}/\left( \mathfrak{u}_{d_1} \oplus \mathfrak{u}_{d_2} \right)$,  which has a canonical complex stucture, since it is the tangent space of the Grassmannian of complex $d_1$-planes in $\bC^d$.  If we write $\bR_{d}$ for the copy of the real line appearing in Equation \eqref{eq:stable_complex_base}, the tautological isomorphism $ \bR_{d} \cong \bR_{d_2}$ respects the identification with the normal direction of the unit circle bundle, since the marked point $z_+$ lies in the second component. Using the identification of $\bR_{d_1}$ with the normal line of $ \partial^{d_1, d_2} \bar{B}_d $ in $ \bar{B}_d $, we then obtain a commutative diagram:
\begin{equation} \label{eq:compatibility_stable_complex_base}
  \begin{tikzcd}
    T  \bar{B}_{d_1}\oplus \bR_{d_1} \oplus   \mathfrak{u}_{d_1} \oplus T  \bar{B}_{d_2} \oplus \mathfrak{u}_{d_2} \oplus \bR_{d_2}  \oplus \frac{\mathfrak{u}_{d}}{\mathfrak{u}_{d_1} \oplus \mathfrak{u}_{d_2}} \ar[d] \ar[r] &   I^B_{d_1}  \oplus  I^B_{d_2} \oplus \frac{\mathfrak{u}_{d}}{\mathfrak{u}_{d_1} \oplus \mathfrak{u}_{d_2}} \ar[d] \\
    T  \bar{B}_d \oplus \mathfrak{u}_d \oplus \bR_{d}  \ar[r] &    I^B_d.
  \end{tikzcd}
\end{equation}

It is straightforward to check that the identifications we have are consistent for higher codimension strata. One way to encode the coherence data is to consider the derived stack consisting of the vector bundle on $\bar{B}_d $ induced by $ i \mathfrak{u}_d $, equipped with the trivial section. Equation \eqref{eq:stable_complex_base} gives rise to an isomorphism
\begin{equation} \label{eq:stable_complex_base-complex}
  T  \bar{B}_d \oplus \mathfrak{gl}_{d,\bC} \oplus \bR \cong  I^B_d \oplus   i \mathfrak{u}_d,
\end{equation}
which we interpret as a lift to a complex oriented derived stack in the sense of Section \ref{sec:stably-framed-stably}, relative the trivial vector space, i.e. an object of $d\Orb^{U}(0,0) $. We can then build a category with objects the natural numbers, with morphisms from $j$ to $k$ given by this enrichment of $(\bar{B}_{k-j},i \mathfrak{u}_{k-j} ,0)$, and compositions given by the result of taking the direct sum of each corner of Diagram \eqref{eq:compatibility_stable_complex_base} with $  i \mathfrak{u}_{k-j} \oplus i \mathfrak{u}_{\ell-k} \cong i \mathfrak{u}_{\ell-j}$.
\begin{lem}
  The complex oriented derived stacks $(\bar{B}_{k-j},i \mathfrak{u}_{k-j} ,0)$ are the morphism spaces of a non-proper flow category with objects the natural numbers.  \qed
\end{lem}
For our application, we need to shift the obstruction bundle so that it matches the part of the obstruction bundle of the moduli spaces of maps which comes from the choice of framing data. In Equation \eqref{eq:normal-b-spaces}, we exactly identified this as the direct sum $N_d$ of $ i \mathfrak{u}_{d}$ with $\bC^d$. Since the latter summands are a collection of representations which are additive under $d$ (i.e. the restriction of the $U(d_1 + d_2)$ representation $\bC^{d_1 + d_2}$ under the homomorphism $U(d_1) \times U(d_2) \to U(d_1 + d_2)$ is the direct sum of $\bC^{d_1}$ with $\bC^{d_2}$), we conclude:
\begin{cor}
  The complex oriented derived stacks $(\bar{B}_{k-j},N_{k-j} ,0)$ are the morphism spaces of a non-proper flow category with objects the natural numbers.  \qed 
\end{cor}

Let us now restrict attention to the open subset consisting of chains of curves, which is relevant to the study of Floer theory in the absence of sphere bubbling:
\begin{lem}
  The action of $U(d)$ on  $\cF^{\bR}_2(d,*)$ is free, and the quotient is contractible. \qed
\end{lem}
The contractibility implies that we can choose a global identification
\begin{equation} \label{eq:global_framing}
       T B_d \oplus \bR \cong i \mathfrak{u}_d. 
\end{equation}
We would like to check that these identifications can be chosen consistently. To this end, we have the following result, which follows from the fact that all the corner strata of $B_d$ are contractible, and the poset of such boundary strata has a unique maximal element:
\begin{lem}
  The manifold $B_d$ deformation retracts to its highest depth corner stratum $\partial^{1, \ldots, 1} B_d$. \qed 
\end{lem}
Proceeding inductively on the dimension $d$, we can therefore choose the framing in Equation \eqref{eq:global_framing} so that the following diagram commutes:
\begin{equation} \label{eq:compatibility_stable_framing_base}
  \begin{tikzcd}
    T  B_{d_1}\oplus \bR_{d_1} \oplus T  B_{d_2}  \oplus \bR_{d_2}  \oplus \frac{\mathfrak{u}_{d}}{\mathfrak{u}_{d_1} \oplus \mathfrak{u}_{d_2}} \ar[d] \ar[r] &   i \mathfrak{u}_{d_1} \oplus i \mathfrak{u}_{d_2}  \oplus i  \frac{\mathfrak{u}_{d}}{\mathfrak{u}_{d_1} \oplus \mathfrak{u}_{d_2}} \ar[d] \\
    T  B_d  \oplus \bR_{d}  \ar[r] &    i \mathfrak{u}_d .
  \end{tikzcd}
\end{equation}
We can encode the coherence of this data as follows:
\begin{lem}
  The framed derived manifolds $(B_{k-j},N_{k-j} ,0)$ are the morphism spaces of a non-proper flow category with objects the natural numbers.  \qed
\end{lem}

\subsubsection{Stable fibrewise complex structure}
\label{sec:stable-fibr-compl}

We now turn our attention to the tangent space  $T^{\pi}  \hat{X}(p,r)^{(d)}$  of the fibre of the projection map to $  \bar{B}_d$. To construct a stable complex structure on this vector bundle it is convenient to make the auxiliary choice of a complex-linear connection on $TM$, whose pullback under any time-$1$ periodic orbit of $H$ we assume to have trivial monodromy. Given such a non-degenerate time-$1$ periodic orbit $p$, we obtain a trivialisation of $p^* TM$, as a complex vector bundle, which we can use to write the linearisation of the Floer equation as an equation on the space of sections, over the pullback ot $TM$ to $\bR_{s} \times S^1_{t}$, satisfying an inhomogeneous equation whose leading order term is the equation for holomorphic sections. 

As in \cite[Section 11.3.1]{Abouzaid2021a}, we cutoff the inhomogeneous term by a function $\chi(s)$ which vanishes in the region $s \ll 0$ and is the identity when $0 \ll s$. 
This determines a Fredholm operator $D_p$ associated to a Cauchy-Riemann problem on the cylinder, which agrees with the linearised Floer operator along the positive end, and which agrees with the trivial complex linear operator with values in $T_{p(0)}(X)$ along the negative end (using the complex-linear connection to perform this identification). Filling in the puncture at the negative end, we can thus describe $D_u$ as an operator on the complex plane.  Denote by $(V^+_p, V^{-}_p)$ the kernel and cokernel of this operator (restricted to sections with vanish at the origin), and fix a splitting of the projection map onto the cokernel, with image consisting of compactly supported $(0,1)$-forms, so that we can associate to each $V^{-}_p$ an inhomogeneous equation on the cylinder.

As in \cite[Section 11.3.2]{Abouzaid2021a}, given a pair $(p,r)$ of orbits, we denote by $D_u$ the linearization of the Cauchy-Riemann operator associated to an element $u$ of the thickening $\cT(p,r)$, with domain $\Sigma$, which we can write again as an inhomegenous equation on sections of $u^* TM$, with leading order term given by holomorphicity with respect to the pullback connection. Assuming for simplicity that $\Sigma$ is a cylinder $\bR_{s} \times S^1_{t}$, we obtain a $1$-parameter family of equations by choosing the cutoff the inhomogeneous term (as displayed in Figure \ref{fig:homotopy_operators_gluing}), with respect to the family of functions $\chi(s + R)$ associated to a real number $R$. In the limit $R \to -\infty$, this family of operators converges to $D_p \# D_u$, while in the limit $R \to +\infty $, it converges to $D_{u}^{\bC} \# D_q$, where $D_u^{\bC}$ is the (complex-linear) operator of holomorphic sections of $u^* TM$.

\begin{figure}[h]

  \centering
  \begin{tikzpicture}
    \begin{scope}[shift={(-2,0)}, xscale=2]
         \begin{scope}[shift={(0,2.5)}]
           \coordinate[label= center:$p(0)$] (p) at (0,-2.25);
     
          \draw  (0,-2.5) .. controls (-.5,-2.5) and  (-.5,-4) .. (-.5,-4.5) ;
         \draw  (0,-2.5)  .. controls (.5,-2.5) and  (.5,-4) .. (.5,-4.5) ;
         \begin{scope}[xscale=.5]
           \draw [fill,color=black] (0,-2.5) circle (.05);
           \draw [fill,color=black] (0,-3.5) circle (.05);
          \end{scope}
     
         \draw [dashed]  (.5,-4.5) arc (180:0: -.5 and .1);
         \draw  (.5,-4.5) arc (180:360:-.5 and .1);

         \draw  (.5,-4.5) arc (180:360:-.5 and .1);
         \coordinate[label= center:$p$] (p) at (0,-4.75);
         \end{scope}
         
 \draw (.5,-2.5) arc (180:0: -.5 and .1);
         \draw  (.5,-2.5) arc (180:360:-.5 and .1);
         \draw  (-.5,-2.5) -- (-.5,-4.5) ;
         \draw  (.5,-2.5) -- (.5,-4.5) ;
          \coordinate[label= center:$D_u$] (u) at (0,-3.5);
         \draw [dashed]  (.5,-4.5) arc (180:0: -.5 and .1);
         \draw  (.5,-4.5) arc (180:360:-.5 and .1);
          \draw [dashed]  (.5,-2.5) arc (180:0: -.5 and .1);
         \draw  (.5,-4.5) arc (180:360:-.5 and .1);
         \coordinate[label= center:$r$] (r) at (0,-4.75);
             \end{scope}

             \begin{scope}[shift={(1,0)}, xscale=2]
               \begin{scope}[shift={(0,2.5)}]
                 \coordinate[label= center:$p(0)$] (p) at (0,-2.25);
         
          \draw  (0,-2.5) .. controls (-.5,-2.5) and  (-.5,-4) .. (-.5,-4.5) ;
         \draw  (0,-2.5)  .. controls (.5,-2.5) and  (.5,-4) .. (.5,-4.5) ;
       \end{scope}

                 \draw  (.5,-4.5) -- (.5,-2) ;
                 \draw  (-.5,-4.5) -- (-.5,-2);
            \begin{scope}[xscale=.5]
              \draw [fill,color=black] (0,0) circle (.05);
              \draw [fill,color=black] (0,-2.5) circle (.05);
          \end{scope}

          \draw [dashed]  (.5,-4.5) arc (180:0: -.5 and .1);
         \draw  (.5,-4.5) arc (180:360:-.5 and .1);
         \coordinate[label= center:$r$] (r) at (0,-4.75);
       \end{scope}

  \begin{scope}[shift={(4,0)}, xscale=2]
    
         \begin{scope}[shift={(0,2.5)}]
           \coordinate[label= center:$p(0)$] (p) at (0,-2.25);
    
          \draw  (0,-2.5) .. controls (-.5,-2.5) and  (-.5,-3.5) .. (-.5,-3.75) .. controls (-.5,-4) and  (-.5,-5) .. (0,-5) ;
         \draw  (0,-2.5)  .. controls (.5,-2.5) and  (.5,-3.5) .. (.5,-3.75) .. controls (.5,-4) and  (.5,-5) .. (0,-5) ;
   \draw [dashed]  (.5,-3.75) arc (180:0: -.5 and .1); 
         \draw  (.5,-3.75) arc (180:360:-.5 and .1);
       \coordinate[label= center:$D_{u}^{\bC}$] (u) at (0,-3.5);
         \coordinate[label= center:$r(0)$] (r) at (0,-4.75);
         \end{scope}
          \draw  (0,-2.5) .. controls (-.5,-2.5) and  (-.5,-4) .. (-.5,-4.5) ;
         \draw  (0,-2.5)  .. controls (.5,-2.5) and  (.5,-4) .. (.5,-4.5) ;
         \begin{scope}[xscale=.5]
           \draw [fill,color=black] (0,0) circle (.05);
             \draw [fill,color=black] (0,-2.5) circle (.05);
             \draw [fill,color=black] (0,-3.5) circle (.05);
          \end{scope}

   \draw [dashed]  (.5,-4.5) arc (180:0: -.5 and .1); 
         \draw  (.5,-4.5) arc (180:360:-.5 and .1);

         \draw  (.5,-4.5) arc (180:360:-.5 and .1);
         \coordinate[label= center:$r$] (r) at (0,-4.75);

       \end{scope}

  \end{tikzpicture}
  \caption{The homotopy between  $D_p \# D_u$  and $D^{\bC}_{u} \# D_r$; the marked point indicates the region where the cutoff takes place.}
    \label{fig:homotopy_operators_gluing}
\end{figure}
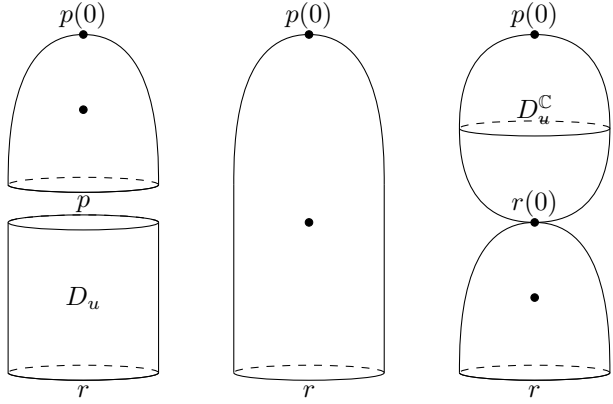

This procedure still makes sense for arbitrary elements of  $\cT(p,r)$: as discussed in  \cite[Section 11.3.2]{Abouzaid2021a}, we identify the interval of parametrisation with a choice of stable map to $M \times \bP^1$, which has degree $1$, and which maps the marked points labelled by $p$ and $r$ to $0$ and $\infty$. Pulling back a fixed cutoff function from $\bP^1$ yields a family of Cauchy-Riemann operators realising this interpolation. Note that this family is naturally equivariant under the action of $U(d)$ on $\cT(p,r)$, since none of our choices break this symmetry:

\begin{lem} \label{lem:parametrised_path_family_operators}
There is a $U(d)$-equivariant family of Fredholm operators, parametrised by the product of $\cT(p,r)^{(d)}$ with the interval $[0,1]$ and equivariant with respect to the action of $U(d)$ on $\cT(p,r)^{(d)}$, which over $(u,0)$ agree with the direct sum of $D_p$ with $D_u $, and over $(u,1)$ agree with the direct sum of a complex linear operator with $D_r$. \qed
\end{lem}

The above result implies that we may choose a finite dimensional representation of $U(d)$, equipped with an equivariant map to the domains of these Fredholm operators, so that the extended family of operators is surjective in a neighbourhood of the $0$-locus of the defining section of the moduli space of Floer trajectories. The kernel and cokernel of these operators then define vector bundles on the product of this neighbourhood with an interval.

In order to ensure consistency of the resulting orientations with the structure maps of the flow category, we have to choose the representation, and its map to the domain, to be compatible with inductive choices. It is convenient to choose these representations to decompose into two summands: the first, which is trivial, is the direct sum $W^{(d)}(p,r)$, over all orbits $q$ and integers $d_1 + d_2 = d$ labelling a boundary component of $\cT(p,r)$, of the vector spaces $V^-_q$. Recall that, by our splitting of the operator associated to each orbit, such a vector space parametrises a family of inhomogeneous deformations of the asymptotic Cauchy-Riemann operator for solutions to Floer's equation converging to $q$. In order to extend these deformations to general elements of $\cT(p,r)^{(d)} $ we fix a tubular neighbourhood of each orbit, which allows us to parallel transport the inhomogenous term parametrised by $V^-_q$ to the thin part of curves which approach a stable trajectory that is broken at $q$ (c.f. the discussion surrounding \cite[Definition 11.28]{Abouzaid2021a}). By a cutoff function, we extend this construction to all elements of $ \cT(p,r)^{(d)}$. 

The second choice is that of a finite dimensional complex $U(d)$ representation, which we denote $I^{\pi,-}(p,r)^{(d)}$, that is equipped with an embedding of $U(d_1) \times U(d_2)$ representations
\begin{equation}
    I^{\pi,-}(p,q)^{(d_1)} \oplus  I^{\pi,-}(q,r)^{(d_2)} \to I^{\pi,-}(p,r)^{(d)}
  \end{equation}
as well as with an equivariant maps to the space of domains of the operators in Lemma \ref{lem:parametrised_path_family_operators}, achieving surjectivity in a neighbourhood of the compact subset of $\cT(p,r)^{(d)}$ given by solutions to the undeformed Floer equation. We require that the inhomogeneous term associated to any element of $ I^{\pi,-}(p,r)^{(d)} $  in the orthogonal complement of the orbit under $U(d)$ of the image of the above map, vanish along the boundary stratum labelled by $q$ and the decomposition $d_1 + d_2 = d$.

The result of these choices is a family of equivariant vector bundles over the product of this this neighbourhood with an interval, which on one end agrees with the fibrewise tangent space $ T^{\pi} \cT(p,r)^{(d)}$, and at the other with an equivariant complex vector bundle. Trivialising this family of vector bundles along the $[0,1]$ direction, as prescribed in \cite[Definitions 11.32 and 11.33]{Abouzaid2021a}), and writing $ I^{\pi,+}(p,r)^{(d)} $ for the induced complex vector bundle over the quotient,  we conclude:
\begin{prop}[c.f. Section 11.3.5 of \cite{Abouzaid2021a}]
Over a neighbourhood $X^{\circ}(p,r)^{(d)}$ of the $0$-locus in $X(p,r)^{(d)}$, there is an equivalence of virtual bundles:
  \begin{equation}
    T^{\pi} \bX(p,r)^{(d)}  \oplus V_p \oplus W^{(d)}(p,r)   \cong    I^{\pi}(p,r)^{(d)} \oplus V_r \oplus W^{(d)}(p,r).
  \end{equation}
  This map is compatible with composition in the following sense: given an orbit $q$, and a decomposition $d = d_1 + d_2$, the intersection of $X^{\circ}(p,r)^{(d)} $ with the boundary stratum labelled by this data agrees with a vector bundle over the product $ X^{\circ}(p,q)^{(d_1)} \times X^{\circ}(q,r)^{(d_2)}$. Moreover, we have split-embeddings,
  \begin{align}
      W^{(d_1)}(p,q)   \oplus  V^-_q \oplus W^{(d_2)}(q,r) & \to  W^{(d)}(p,r) 
  \end{align}
  as well as an equivalence of virtual complex bundles
  \begin{equation}
         I^{\pi}(p,q)^{(d_1)} \oplus  I^{\pi}(q,r)^{(d_2)} \cong I^{\pi}(p,r)^{(d)} 
  \end{equation}
 over the corresponding boundary stratum of $ \cT(p,r)^{(d)}$, which are compatible along the codimension two boundary strata, so that the following diagram commutes:
  \begin{equation}
    \begin{tikzcd}
      \begin{aligned}
        T^{\pi} \bX(p,q)^{(d_1)}  \oplus V_p \oplus W^{(d_1)}(p,q) \\
         T^{\pi} \bX(q,r)^{(d_2)}  \oplus V_q \oplus W^{(d_2)}(q,r) 
      \end{aligned}
      \ar[r] \ar[d] &      \begin{aligned}
                             I^{\pi}(p,q)^{(d_1)} \oplus V_q \oplus W^{(d_1)}(p,q) \\
            I^{\pi}(q,r)^{(d_2)} \oplus V_r \oplus W^{(d_2)}(q,r) 
      \end{aligned}  \ar[d] \\
      T^{\pi} \bX(p,r)^{(d)}  \oplus V_p \oplus W^{(d)}  \ar[r] &    I^{\pi}(p,r)^{(d)} \oplus V_r \oplus W^{(d)}(p,r) .
    \end{tikzcd}
  \end{equation}
  \qed
\end{prop}

At this stage, we recall the smooth flow category $\hat{\bX}$ from Proposition \ref{prop:unstructured_flow_category} has morphism spaces given by the collared completions $\hat{\bX}(p,r)^{(d)} $  of the derived orbifolds $\bX^{(d)}(p,r)  $, in order to achieve consistency of smoothings. We apply the same construction to the open suborbifolds appearing above:

\begin{cor}
 The collared completions $\hat{\bX}^{\circ}(p,r)^{(d)}$ are the morphism spaces of a complex oriented flow category. 
\end{cor}
\begin{proof}
  Define complex vector bundles on  $X^{\circ}(p,r)^{(d)} $  as direct sums 
\begin{align}
    I^{+}(p,r)^{(d)}   & \cong   I^{\pi, +}(p,r)^{(d)}  \oplus I^B_d   \\
    I^{-}(p,r)^{(d)}    & \cong  I^{\pi, -}(p,r)^{(d)}   \oplus   \mathfrak{gl}_{d,\bC} ,
  \end{align}
and use the isomorphisms
  \begin{align}
    T^{+} \bX(p,q)^{d} & \cong  T^{+,\pi} \bX(p,q)^{d} \oplus T \bar{B}_d  \\
    T^{-} \bX(p,q)^{d} & \cong  T^{-,\pi} \bX(p,q)^{d} \oplus  i \mathfrak{u}_d,
  \end{align}
together with the fibrewise and base orientations defined above.\end{proof}

\subsubsection{Framed flow categories}
\label{sec:fram-flow-categ}

We now turn to the proof of Proposition \ref{prop:framed-ham-flow}. There are two fundamental issues: the first is the fact that different Hamiltonian orbits are not based as the same point, so that one needs to understand how an assumption about the based loop space will suffice for a construction where multiple basepoints are required, and the second is that we need models for the forgetful map from complex to real vector bundles as well as for the Bott periodicity isomorphism $\Omega U \cong \bZ \times BU$, which are compatible with each other. This second issue will arise because we shall not obtain a trivialisation of the complex vector bundles constructed in Section \ref{sec:stable-fibr-compl}, but rather a trivialisation of the associated real bundles.

\begin{rem}
  We do not need to construct our model of Bott periodicity as a map of classifying spaces, nor do we need to compare it with any other approach, so the above discussion is mostly about heuristics, and all constructions shall be implemented using explicit index theoretic considerations.
\end{rem}

The resolution of the problem with basepoints is quite straightforward: given a non-degenerate Hamiltonian $H$, we may fix a contractible subset of $M$ containing the starting point of all the time-$1$ periodic Hamiltonian  orbits of $H$, and assume that the Hermitian connection is flat on this subset. In this way, we can identify the tangent spaces $T_{p(0)}M$ for all such orbits with $\bC^{n}$. A lift to $O(n)$ of the monodromy map $  \Omega M \to U(n)   $ for a basepoint chosen in this contractible set then yields a lift of the monodromy map for all paths between the starting point of the time-$1$ periodic orbits.

The construction of an appropriate model for the forgetful map from $BU$ to $BO$ is intuitively straightforward, but its implementation is significantly more involved: we start with the fact that the Cauchy-Riemann problem of holomorphic sections of complex vector bundles over the disc, with Lagrangian boundary conditions, is a real Fredholm problem (in fact, it is a model for $\bZ \times BO$, but we shall not use this). Given a holomorphic vector bundle $E$ over the sphere, with a holomorphic identification of the fibre near $0$ with $\bC^n$, any degree $1$ map $\pi \co D^2 \to S^2$ taking the boundary to $0$, determines by pullback a real Cauchy-Riemann problem on the disc, as long as the map is holomorphic outside a neighbourhood of the boundary which maps to the region where the vector bundle is identified with $\bC^n$.

For the next statement, fix a finite dimensional complex vector space $W$ mapping to space of $(0,1)$-forms on $S^2$, valued in $E$, with support away from the origin in $S^2$, and which surjects onto the cokernel of the Cauchy-Riemann operator:
\begin{lem} \label{sec:BO-BU=gluing}
  If the projection map $\pi \co D^2 \to S^2$ is obtained by extending a biholomorphism between a ball in $D^2$ of sufficiently small radius (centered at any point) and the complement of a ball in $S^2$ of sufficiently small radius (centered at the origin), then there is a natural stable isomorphism of virtual index bundles
  \begin{equation} \label{eq:identification_BU-BO}
        ( \ker_{\bar{\partial}_E} \oplus W, \coker_{\bar{\partial}_E} \oplus W) \cong  ( \ker_{\bar{\partial}_{ \pi^* E}} \oplus W, \coker_{\bar{\partial}_{\pi^* E}} \oplus W),
      \end{equation}
    where we consider $\bar{\partial}_E$ and $ \bar{\partial}_{\pi^* E} $ as operators on sections which respectively vanish at $0 \in S^2$ and $1 \in D^2$, and $W$ is a finite dimensional vector space equipped with a map to the space of inhomogeneous terms of $\bar{\partial}_E$, that surjects onto the cokernel. 
\end{lem}
\begin{proof}[Sketch of proof:]
  This is a standard linear gluing result: consider the pre-stable Riemann surface consisting of a disc with an attached sphere at the origin. The sujectivity assumption on $W$ implies that the associated extended Cauchy-Riemann problem on $S^2$, whose domain is the direct sum of $W$ with smooth sections of $E$, has trivial cokernel (i.e. is regular). Cutting off these sections near the node then yields map from the kernel of this extended operator to an approximation of the kernel of the extended operator on $\pi^* E$, and this map is then an isomorphism for sufficiently large gluing parameters (the usual conventions are such that large gluing parameters correspond to small balls).  \end{proof}
The key point is that the left hand side of Equation \eqref{eq:identification_BU-BO} has a natural complex structure, while the right hand side only has a real structure. We shall use this result in (finite-dimensional) families where the vector bundle $E$ as well as the attaching point vary. 

We shall need to be able to recognise when the index bundle of a Cauchy-Riemann problem over the disc is trivial:
\begin{lem} \label{lem:trivialisation_clutching}
  If $E$ is a holomorphic vector bundle over the disc, and $\Lambda$ a family of Lagrangian boundary conditions over the circle, then an extension of $\Lambda$ to a Lagrangian sub-bundle of $E$ determines, up to contractible choice, a (stable) trivialisation of the real Cauchy-Riemann operator associated to sections of $E$ which vanish at $1 \in D^2$.
\end{lem}
\begin{proof}[Sketch of proof:]
The assumptions yield a canonical identification up to contractible choice the pair $(E,\Lambda)$ with the trivial complex bundle with fibre $E |_{1}$ and Lagrangian boundary conditions $\Lambda|_1$. The result then follows from the fact that the kernel and cokernel of the operator obtained after deformation are both trivial.
\end{proof}

We shall apply the above two results to holomorphic vector bundles over the disc given as $\pi^* E$ for a holomorphic vector bundle $E$ on $S^2$, which is itself produced by a clutching construction associated to a map $S^1 \to U(n)$.
\begin{cor} \label{cor:trivialisation_orthogonal_clutching}
A lift of the clutching map defining the vector bundle $E$ to $O(n)$ determines a stable trivialisation, as a real vector space, of the virtual index bundle of $E$, canonically up to contractible choice.
\end{cor}
\begin{proof}
  Lemma \ref{sec:BO-BU=gluing} identifies the real virtual space underlying the index bundle of $E$ as the index bundle of $ \pi^* E$. The lifting assumption equips $\pi^* E $ with a Lagrangian subbundle extending the trivial choice over the boundary.
\end{proof}

We shall need one more basic result about virtual index bundles: the identification of $\pi^*E$ as a holomorphic vector bundle on the disc relies on a trivialisation of $E$ near the origin in $S^2$. From the point of view of the pullback bundle, changing this trivialisation corresponds to changing the boundary condition by the given element of $U(n)$. If this element lifts to $O(n)$, the boundary conditions are the same, so we conclude:
\begin{cor} \label{cor:gluing_does_not_change_triv}
  Changing the trivialisation of $E$ by an element of $O(n)$ does not change the trivialisation of the virtual index bundle in Corollary \ref{cor:trivialisation_orthogonal_clutching}. \qed
\end{cor}

  It is useful at this stage to explain the construction for a single moduli space: we thus fix a non-degenerate Hamiltonian  $H$ on a symplectic manifold $M$,  a pair of Hamiltonian orbits $(p,r)$, and an integer $d$ so that the component $\Mbar(p,r)^{(d)}$ of the compactified moduli space of Floer gradient trajectories does not contain any element which has a sphere bubble, as well as a factorisation of the monodromy map $\Omega M \to U(n)$ through $O(n)$, as in the statement of Proposition \ref{prop:framed-ham-flow}:
  \begin{lem} \label{lem:stably_framed_one_moduli_space}
    Any choice of thickening $X^{(d)}(p,r)$ is stably framed in the sense that there is a vector space $ W^{(d)}_{fr}(p,r)$, together with an isomorphism
  \begin{equation}
    T \bX(p,r)^{(d)}  \oplus V_p \oplus W^{(d)}_{fr}(p,r)   \cong    V_r \oplus W^{(d)}_{fr}(p,r)
  \end{equation}
  in a neighbourhood of the $0$-section.
  \end{lem}
  \begin{proof}
    The vector space $W^{(d)}_{fr}(p,r)$ is the direct sum the vector space $W^{(d)}(p,r)$ from the construction of the complex structure with the Lie algebra $\mathfrak{u}_d $, and with an additional vector space admitting an embedding of $     I^{\pi,-}(p,q)^{d} $, which is required to achieve surjectivity of the family of Cauchy-Riemann operators which we presently describe; the reader may want to go back to Figure \ref{fig:homotopy_operators_gluing}, and consult Figure \ref{fig:homotopy_operators_gluing_framing}.

\begin{figure}[h]
  \centering
  \begin{tikzpicture}
       
  \begin{scope}[shift={(-4,0)}, xscale=2]

    \begin{scope}[shift={(0,-2.5)}, yscale=-1]
              \draw  (0,-2.5) .. controls (-.5,-2.5) and  (-.5,-4) .. (-.5,-4.5) ;
         \draw  (0,-2.5)  .. controls (.5,-2.5) and  (.5,-4) .. (.5,-4.5) ;

\draw [thick]  (.5,-4.5) arc (180:0: -.5 and .1); 
         \draw[thick]  (.5,-4.5) arc (180:360:-.5 and .1);
        
         \coordinate[label= center:$\bR^n \subset  \bC^n \cong T_{p(0)} M $] (Tp) at (0,-4.75);
    \end{scope}

    \begin{scope}[shift={(0,2.5)}]
           \coordinate[label= center:$ T_{p(0)} M $] (p) at (0,-2.25);
         
          \draw  (0,-2.5) .. controls (-.5,-2.5) and  (-.5,-3.5) .. (-.5,-3.75) .. controls (-.5,-4) and  (-.5,-5) .. (0,-5) ;
         \draw  (0,-2.5)  .. controls (.5,-2.5) and  (.5,-3.5) .. (.5,-3.75) .. controls (.5,-4) and  (.5,-5) .. (0,-5) ;
   \draw [dashed]  (.5,-3.75) arc (180:0: -.5 and .1); 
         \draw  (.5,-3.75) arc (180:360:-.5 and .1);

         \coordinate[label= center:$u^{*}(TM)$] (u) at (0,-3.5);
         \coordinate[label= center:$  \bC^n \cong  T_{r(0)} M $] (r) at (0,-5.25);
         \end{scope}

         \begin{scope}[xscale=.5]
           \draw [fill,color=black] (0,0) circle (.05);
             \draw [fill,color=black] (0,-2.5) circle (.05);
          \end{scope}

       \end{scope}

       \begin{scope}[shift={(-1,0)}, xscale=2]

    \begin{scope}[shift={(0,-5)}, yscale=-1]
              \draw  (0,-2.5) .. controls (-.5,-2.5) and  (-.5,-4) .. (-.5,-4.5) --   (-.5,-7);
              \draw  (0,-2.5)  .. controls (.5,-2.5) and  (.5,-4) .. (.5,-4.5)  --  (.5,-7);
             \coordinate[label= center:$\phi_{u} $] (u) at (0,-4.25); 
\draw  (.5,-4.5) arc (180:0: -.5 and .1); 
\draw  [dashed] (.5,-4.5) arc (180:360:-.5 and .1);

\draw [thick]  (.5,-7) arc (180:0: -.5 and .1); 
         \draw[thick]  (.5,-7) arc (180:360:-.5 and .1);
          
         \coordinate[label= center:$\bR^n \subset  \bC^n $] (Tp) at (0,-7.25);
    \end{scope}

       \end{scope}

          \begin{scope}[shift={(2,0)}, xscale=2]

    \begin{scope}[shift={(0,-2.5)}, yscale=-1]
            \draw  (0,-2.5) .. controls (-.5,-2.5) and  (-.5,-4) .. (-.5,-4.5) ;
         \draw  (0,-2.5)  .. controls (.5,-2.5) and  (.5,-4) .. (.5,-4.5) ;

\draw [thick]  (.5,-4.5) arc (180:0: -.5 and .1); 
         \draw[thick]  (.5,-4.5) arc (180:360:-.5 and .1);
        
         \coordinate[label= center:$\phi_{u} \bR^n  $] (Tp) at (0,-4.75);
         
    \end{scope}

       \end{scope}

  \end{tikzpicture}
  \caption{The homotopy between the real operator associated to $D^{\bC}_{u} $ and an operator with boundary conditions the image of $\bR^n$ under a loop of unitary matrices.}
  
  \label{fig:homotopy_operators_gluing_framing}
\end{figure}
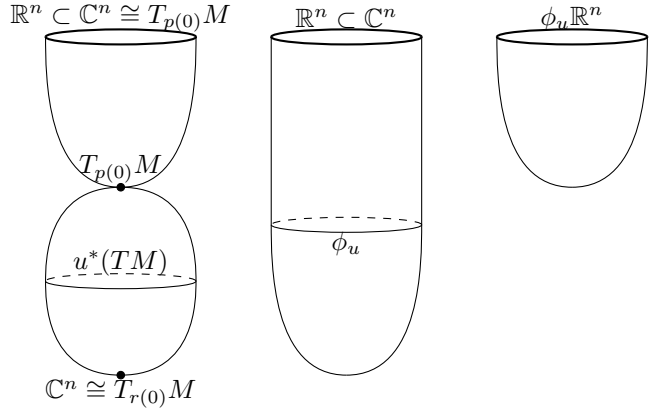

    In the construction of the stable complex structure, we obtained a virtual vector space over $X(p,r)^{(d)} $ as the index bundle of a family of complex-linear operators which we denoted $D_{u}^{\bC}$, defined on the vector bundle over $S^2$, which away from $0$ and infinity is given by $u^*(TM)$, and which is respectively extended to these two points by the trivialisations of $p^*TM$ and $r^* TM$ determined by the restriction of the connection to the Hamiltonian orbits. Using the connection to parallel transport the fibre over points near the positive end to those near the negative end, we obtain a description of this vector bundle as a clutching construction associated to a map $\phi_u \co S^1 \to U(n)$. Our assumption give a factorisation of this map, in families over $X(p,r)^{(d)}  $, through the orthogonal group. Gluing this bundle at the origin to the constant vector bundle over the disc with fibre $T_{p(0)}X \cong \bC^n$, and applying Lemma \ref{lem:trivialisation_clutching} yields the desired result, using the compactness of the zero-locus of the defining section to choose a single vector space which surjects onto the cokernel of the family of operators parametrised by the product of $[0,1]$ with the moduli space of stable Floer trajectories.
  \end{proof}

  In order to show that this construction yields a framed flow category, we need to be able to perform the operation of gluing spheres to discs at multiple points. Before we explain the general construction, we consider the case of codimension $1$ strata:

  \begin{lem} \label{lem:homotopy_codim_1}
    There is a natural homotopy between the restriction of the framing constructed in Lemma \ref{sec:BO-BU=gluing} to the boundary of $ X(p,r)^{(d)}$, and the product framing on $X(p,q)^{(d_1)} \times X(q,r)^{(d_2)}$. 
  \end{lem}
  \begin{proof}
    Let $u_1$ and $u_2$ respectively be elements of   $X(p,q)^{(d_1)} $ and $ X(q,r)^{(d_2)}$.  The product framing is associated to a pair of discs carrying sphere bubbles equipped with the bundles $u_1^* TM$ and $u_2^* TM$, and with the trivialisation of the corresponding real Cauchy-Riemann operator  obtained by gluing at the node and trivialising the bundles over the disc (see Figure \ref{fig:homotopy_operators_gluing_framing_product}). We identify the direct sum using boundary gluing with the operator on a disc with two sphere bubbles, and then consider the family of operators obtained by moving the attaching points of the spheres until they collide, which is the endpoint of another family of operators associated to gluing the two sphere at a point along the negative real line from $0$ to $\infty$. To define this second family, we have to choose an identification of the fibres at the two points being attached; we use parallel transport for the connection along the path from $0$ to the node. The end point of this family is an operator associated to a disc with two successively attached sphere, carrying the bundles $u_1^* TM$ and $u_2^* TM$, where gluing between these bundles is given by the map $\phi_{u_2}(1)$. 
\begin{figure}[h]
  \centering
  \begin{tikzpicture}

\begin{scope}[xscale=.5, yscale=.5]
      
  \begin{scope}[shift={(8,0)}, xscale=2]

    \begin{scope}[shift={(0,-2.5)}, yscale=-1]
              \draw  (0,-2.5) .. controls (-.5,-2.5) and  (-.5,-4) .. (-.5,-4.5) ;
         \draw  (0,-2.5)  .. controls (.5,-2.5) and  (.5,-4) .. (.5,-4.5) ;

\draw [thick]  (.5,-4.5) arc (180:0: -.5 and .1); 
         \draw[thick]  (.5,-4.5) arc (180:360:-.5 and .1);

    \end{scope}

    \begin{scope}[shift={(0,2.5)}, xscale=.75]
                  \draw  (0,-2.5) .. controls (-.5,-2.5) and  (-.5,-3.5) .. (-.5,-3.75) .. controls (-.5,-4) and  (-.5,-5) .. (0,-5) ;
         \draw  (0,-2.5)  .. controls (.5,-2.5) and  (.5,-3.5) .. (.5,-3.75) .. controls (.5,-4) and  (.5,-5) .. (0,-5) ;
   \draw [dashed]  (.5,-3.75) arc (180:0: -.5 and .1); 
         \draw  (.5,-3.75) arc (180:360:-.5 and .1);

         \coordinate[label= center:$u_2$] (u) at (0,-3.5);
        \coordinate[label= center:$  \phi_{u_2} $] (r) at (.75,-5);
         \end{scope}

         \begin{scope}[xscale=.5]
           \draw [fill,color=black] (0,0) circle (.05);
           \draw [fill,color=black] (0,-2.5) circle (.05);
            \draw [fill,color=black] (0,-5) circle (.05);
           \end{scope}

          \begin{scope}[shift={(0,0)}, xscale=.75]
          \draw  (0,-2.5) .. controls (-.5,-2.5) and  (-.5,-3.5) .. (-.5,-3.75) .. controls (-.5,-4) and  (-.5,-5) .. (0,-5) ;
         \draw  (0,-2.5)  .. controls (.5,-2.5) and  (.5,-3.5) .. (.5,-3.75) .. controls (.5,-4) and  (.5,-5) .. (0,-5) ;
   \draw [dashed]  (.5,-3.75) arc (180:0: -.5 and .1); 
         \draw  (.5,-3.75) arc (180:360:-.5 and .1);

         \coordinate[label= center:$u_1$] (u) at (0,-3.5);
        
         \end{scope}   

       \end{scope}

\begin{scope}[shift={(4,0)}, xscale=2]

    \begin{scope}[shift={(0,-2.5)}, yscale=-1]
              \draw  (0,-2.5) .. controls (-.5,-2.5) and  (-.5,-4) .. (-.5,-4.5) ;
         \draw  (0,-2.5)  .. controls (.5,-2.5) and  (.5,-4) .. (.5,-4.5) ;

\draw [thick]  (.5,-4.5) arc (180:0: -.5 and .1); 
         \draw[thick]  (.5,-4.5) arc (180:360:-.5 and .1);

    \end{scope}

    \begin{scope}[shift={(0,2.5)}, xscale=.75]
          \draw  (0,-2.5) .. controls (-.5,-2.5) and  (-.5,-3.5) .. (-.5,-3.75) .. controls (-.5,-4) and  (-.5,-5) .. (0,-5) ;
         \draw  (0,-2.5)  .. controls (.5,-2.5) and  (.5,-3.5) .. (.5,-3.75) .. controls (.5,-4) and  (.5,-5) .. (0,-5) ;
   \draw [dashed]  (.5,-3.75) arc (180:0: -.5 and .1); 
         \draw  (.5,-3.75) arc (180:360:-.5 and .1);

         \coordinate[label= center:$u_2$] (u) at (0,-3.5);
     
         \end{scope}

         \begin{scope}[xscale=.5]
           \draw [fill,color=black] (0,0) circle (.05);
           \draw [fill,color=black] (0,-2.5) circle (.05);
           \draw [fill,color=black] (-.75,-1.25) circle (.05);
           
           \end{scope}
          
           \begin{scope}[ shift={(-.4,1.75)},yscale=4/3, xscale=.5]
             \begin{scope}[rotate around={-90: (0.125,-2.375)}]
                \draw  (0,-2.5) .. controls (-.5,-2.5) and  (-.5,-3.5) .. (-.5,-3.75) .. controls (-.5,-4) and  (-.5,-5) .. (0,-5) ;
         \draw  (0,-2.5)  .. controls (.5,-2.5) and  (.5,-3.5) .. (.5,-3.75) .. controls (.5,-4) and  (.5,-5) .. (0,-5) ;
   \draw [dashed]  (.5,-3.75) arc (180:0: -.5 and .1); 
         \draw  (.5,-3.75) arc (180:360:-.5 and .1);

         \coordinate[label= center:$u_1$] (u) at (0,-3.5);
             \end{scope}
         \end{scope}   

       \end{scope}

          \begin{scope}[shift={(-3,0)}, xscale=2]

    \begin{scope}[shift={(0,-2.5)}, yscale=-1]
              \draw  (0,-2.5) .. controls (-.5,-2.5) and  (-.5,-4) .. (-.5,-4.5) ;
         \draw  (0,-2.5)  .. controls (.5,-2.5) and  (.5,-4) .. (.5,-4.5) ;

\draw [thick]  (.5,-4.5) arc (180:0: -.5 and .1); 
         \draw[thick]  (.5,-4.5) arc (180:360:-.5 and .1);

    \end{scope}

    \begin{scope}[shift={(-.8,1.4)}]
      \begin{scope}[xscale=.5, rotate=60]
        \begin{scope}[xscale=1.5]
          \draw  (0,-2.5) .. controls (-.5,-2.5) and  (-.5,-3.5) .. (-.5,-3.75) .. controls (-.5,-4) and  (-.5,-5) .. (0,-5) ;
         \draw  (0,-2.5)  .. controls (.5,-2.5) and  (.5,-3.5) .. (.5,-3.75) .. controls (.5,-4) and  (.5,-5) .. (0,-5) ;
   \draw [dashed]  (.5,-3.75) arc (180:0: -.5 and .1); 
         \draw  (.5,-3.75) arc (180:360:-.5 and .1);

         \coordinate[label= center:$u_2$] (u) at (0,-3.5);
      
        \end{scope}
      \end{scope}
   \end{scope}

    \begin{scope}[shift={(.8,1.4)}]
      \begin{scope}[xscale=.5, rotate=-60]
        \begin{scope}[xscale=1.5]
          \draw  (0,-2.5) .. controls (-.5,-2.5) and  (-.5,-3.5) .. (-.5,-3.75) .. controls (-.5,-4) and  (-.5,-5) .. (0,-5) ;
         \draw  (0,-2.5)  .. controls (.5,-2.5) and  (.5,-3.5) .. (.5,-3.75) .. controls (.5,-4) and  (.5,-5) .. (0,-5) ;
   \draw [dashed]  (.5,-3.75) arc (180:0: -.5 and .1); 
         \draw  (.5,-3.75) arc (180:360:-.5 and .1);

         \coordinate[label= center:$u_1$] (u) at (0,-3.5);
      
        \end{scope}
      \end{scope}
   \end{scope}

         \begin{scope}[xscale=.5]
           \draw [fill,color=black] (0.65,.3) circle (.05);
             \draw [fill,color=black] (-0.65,.3) circle (.05);
          \end{scope}

      \end{scope}

          \begin{scope}[shift={(-8,0)}, xscale=2]

    \begin{scope}[shift={(0,-2.5)}, yscale=-1]
              \draw  (0,-2.5) .. controls (-.5,-2.5) and  (-.5,-4) .. (-.5,-4.5) ;
         \draw  (0,-2.5)  .. controls (.5,-2.5) and  (.5,-4) .. (.5,-4.5) ;

\draw [thick]  (.5,-4.5) arc (180:0: -.5 and .1); 
         \draw[thick]  (.5,-4.5) arc (180:360:-.5 and .1);

    \end{scope}

    \begin{scope}[shift={(0,2.5)}, xscale=.75]
          \draw  (0,-2.5) .. controls (-.5,-2.5) and  (-.5,-3.5) .. (-.5,-3.75) .. controls (-.5,-4) and  (-.5,-5) .. (0,-5) ;
         \draw  (0,-2.5)  .. controls (.5,-2.5) and  (.5,-3.5) .. (.5,-3.75) .. controls (.5,-4) and  (.5,-5) .. (0,-5) ;
   \draw [dashed]  (.5,-3.75) arc (180:0: -.5 and .1); 
         \draw  (.5,-3.75) arc (180:360:-.5 and .1);

         \coordinate[label= center:$u_2$] (u) at (0,-3.5);

         \end{scope}

         \begin{scope}[xscale=.5]
           \draw [fill,color=black] (0,0) circle (.05);
             \draw [fill,color=black] (0,-2.5) circle (.05);
          \end{scope}

          \begin{scope}[shift={(-1,0)}]

    \begin{scope}[shift={(0,-2.5)}, yscale=-1]
              \draw  (0,-2.5) .. controls (-.5,-2.5) and  (-.5,-4) .. (-.5,-4.5) ;
         \draw  (0,-2.5)  .. controls (.5,-2.5) and  (.5,-4) .. (.5,-4.5) ;

\draw [thick]  (.5,-4.5) arc (180:0: -.5 and .1); 
         \draw[thick]  (.5,-4.5) arc (180:360:-.5 and .1);

    \end{scope}

    \begin{scope}[shift={(0,2.5)}, xscale=.75]
          \draw  (0,-2.5) .. controls (-.5,-2.5) and  (-.5,-3.5) .. (-.5,-3.75) .. controls (-.5,-4) and  (-.5,-5) .. (0,-5) ;
         \draw  (0,-2.5)  .. controls (.5,-2.5) and  (.5,-3.5) .. (.5,-3.75) .. controls (.5,-4) and  (.5,-5) .. (0,-5) ;
   \draw [dashed]  (.5,-3.75) arc (180:0: -.5 and .1); 
         \draw  (.5,-3.75) arc (180:360:-.5 and .1);

        \coordinate[label= center:$u_1$] (u) at (0,-3.5);

         \end{scope}

         \begin{scope}[xscale=.5]
           \draw [fill,color=black] (0,0) circle (.05);
             \draw [fill,color=black] (0,-2.5) circle (.05);
          \end{scope}

        \end{scope}
      \end{scope}

\end{scope}

  \end{tikzpicture}
  \caption{The endpoints as well as two representatives of the homotopy between the direct sum of real operators associated to $D^{\bC}_{u_1} $ and $D^{\bC}_{u_2} $ and the real operator associated to  $D^{\bC}_{u_1}  \# D^{\bC}_{u_2} $.}
  \label{fig:homotopy_operators_gluing_framing_product}
\end{figure}
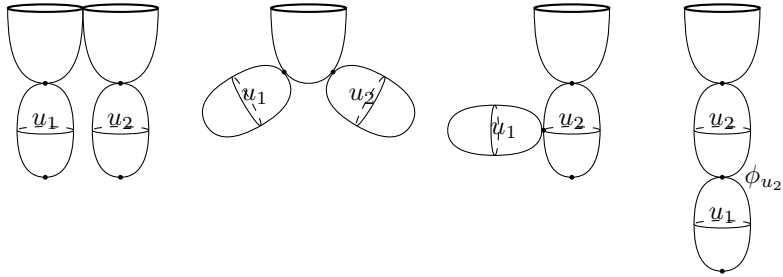

   The key point is that the index bundle associated to the pair $(u_1,u_2)$ is canonically isomorphic to the index bundle associated to the endpoint of this homotopy, even though the two vector bundles are not isomorphic; indeed, both are the direct sums of the index bundles on the two components (using again the convention that we consider sections which vanish at the origin). Moreover, using the given lift up to homotopy of  $\phi_{u_2}$ to a real matrix, the trivialisation of the associated real operator also agrees up to prescribed homotopy, as can be seen by using the clutching construction and applying Corollary \ref{cor:gluing_does_not_change_triv}. It then suffices to show that our construction of the trivialisation also extends to the family of operators constructed above; which is clear from gluing at the nodes to obtain a family of operators on the discs, which can be explicitly described using the clutching construction.
  \end{proof}

In order to generalise this construction to higher codimension, we need a method to describe the domains of the families of Cauchy-Riemann operators that appear when interpolating between different ways of gluing spheres to a disc.  While there are many ways to do this, we shall use a doubling procedured based on the geometry of moduli spaces of discs with boundary marked points: let $\Rbar_{2n+2}$ denote the Deligne-Mumford moduli space of discs with $2n +2$ boundary marked points, which we label $(z_+, z_1, \ldots, z_{2n}, z_-)$ clockwise along the boundary of the disc. The complement of the marked points $(z_+,z_-)$ is identified with  $ \bR \times [0,1]$  canonically up to the action of $\bR$ by translations, and the remaining marked points then lie on the interval $\bR \times \{0\} $. We thus describe the stable disc components connecting $(z_+,z_-)$ as strips, and all other components as disc bubbles. Let $\Sbar_{n}$ denote the subset of marked discs satisfying the following constraints:
  \begin{enumerate}
  \item Each marked point $z_i$ lies on a disc bubble.
  \item No disc bubble contains more than two marked points, in which case the points are $(z_{2i-i}, z_{2i}) $ for some integer $i$.
    \item  The node connecting the disc bubble carrying $z_{2i-1}$ to the strip is the one following $z_{2i-1}$ counterclockwise along the boundary.  
  \end{enumerate}
It is clear from these conditions that the composition map $\Rbar_{2n+2} \times \Rbar_{2m+2} \to  \Rbar_{2(n+m)+2}$ defined by gluing the marked point  $z_+$ on the first factor to the marked point $z_-$ on the second restricts to a map
    \begin{equation} \label{eq:gluing_discs_boundary_marked_points}
        \Sbar_{n} \times     \Sbar_{m} \to \Sbar_{n+m}.
    \end{equation}
  \begin{lem} \label{lem:basic_space_contractible}
    Each space $\Sbar_{n}$ is contractible.
  \end{lem}
  \begin{proof}
    The space $ \Sbar_{1}$ is a singleton, represented by the strip with a disc, carrying the marked points $(z_1,z_2)$, attached along the boundary. The result will thus follow by induction if we construct a deformation retraction of $ \Sbar_{n}$ to $ \Sbar_{1} \times \Sbar_{n-1}$, which is obtained as follows: observe that the marked points $(z_{2n-1}, z_{2n})$ must lie on the same disc bubble, which moreover contains exactly one node. To see this, consider the component carrying $z_{2n-1}$, and note that our assumptions imply that the next special point counterclockwise along the boundary of this pointing must be the node in the direction of the strip. Since all the marked points $z_i$ for $i < 2n-1$ lie counterclockwise of $z_{2n-1}$, this implies that these points do not lie on disc bubbles attached to this component, so the only possibility for this disc to be stable is for $z_{2n}$ to also lie on it.

    Removing this disc bubble (and collapsing unstable components)  defines an element of $\Sbar_{n-1}$. The retraction is then obtained by moving this disc bubble along the boundary (creating ghost components along the way), until reaching the configuration where it is attached directly to the strip, in which case it can be moved towards $z_-$.   
  \end{proof}

  \begin{figure}[h]
    \centering
    \begin{tikzpicture}

        \coordinate[label= above:$z_+$] (+) at (-1,1);
        \coordinate[label= below:$z_-$] (-) at (-1,-1);
        \draw  (-.5,-1) -- (-.5,1);
      \draw (-1.5,-1) -- (-1.5,1) ;;

\filldraw (-.5,0) circle (2pt);
\draw (0,0) circle (0.5);
\draw (90:.4) -- +(90:0.2);
 \coordinate[label= above:$z_1$] (1) at (90:.5);
 \begin{scope}[shift={(1,0)}]
 \filldraw (-.5,0) circle (2pt);
\draw (0,0) circle (0.5);
\draw (90:.4) -- +(90:0.2);
\coordinate[label= above:$z_2$] (2) at (90:.5);  
 \end{scope}
 \begin{scope}[shift={(2,0)}]
 \filldraw (-.5,0) circle (2pt);
\draw (0,0) circle (0.5);
\draw (90:.4) -- +(90:0.2);
\coordinate[label= above:$z_3$] (3) at (90:.5);  
 \end{scope}
 \begin{scope}[shift={(3,0)}]
 \filldraw (-.5,0) circle (2pt);
\draw (0,0) circle (0.5);
\draw (90:.4) -- +(90:0.2);
\coordinate[label= above:$z_4$] (4) at (90:.5);  
 \end{scope}
 \begin{scope}[shift={(4,0)}]
 \filldraw (-.5,0) circle (2pt);
\draw (0,0) circle (0.5);
\draw (60:.4) -- +(60:0.2);
\coordinate[label= above:$z_5$] (5) at (60:.5);  
 \draw (-60:.4) -- +(-60:0.2);
 \coordinate[label= below:$z_6$] (5) at (-60:.5);
\end{scope}

    \end{tikzpicture}
    \caption{The basepoint of the space $\Sbar_3$}
    \label{fig:based_sbar}
  \end{figure}
  
  In order to set up an inductive construction we shall need to explicitly describe some of the strata of $\Sbar_{n}$: to begin, we distinguish a basepoint  given by the configuration in which no disc bubble other than the one carrying $(z_{2n-1}, z_{2n})$ carries more than one marked point, and the disc bubbles lie in a chain, starting at a unique strip component, and ordered according the the labels (see Figure \ref{fig:based_sbar}).  For each (ordered) partition $P = (P_1, \ldots, P_k)$ of the set $\{1, \ldots, n\}$, we denote by $\Sbar_{P} $ the subset of $\Sbar_n$ consisiting of configurations with the property that, for each $1 \leq j \leq k$, the result of forgetting all marked points not in $P_j$ is the basepoint of $\Sbar_{|P_j|}$.  There is a redundancy in this notation, as the paritition consisting of singletons corresponds to the entire space.  We have a generalisation of Lemma \ref{lem:basic_space_contractible}, which can be proved in the same inductive way:

\begin{lem}
    Each space $\Sbar_{P}$ is contractible. \qed
  \end{lem}
  Note that all strata $\Sbar_P $ contain the distinguished basepoint. More generally, we have the following result, which reduces many constructions to the case of partitions consisting of a pair of consecutive elements, with all other elements appearing as singletons:
  \begin{lem} \label{lem:moduli_partition_intersection}
    The intersection of $\Sbar_P$ and $\Sbar_{P'}$ is the stratum associated to the minimal partition which is refined by both $P$ and $P'$. \qed
  \end{lem}

  By doubling along the boundary component carrying the marked points, we associate to each element of $\Sbar_{n} $ a stable disc with interior marked points; the disc bubbles become sphere bubble. We will call the total space of the family obtained by collapsing all disc components which do not contain odd marked points the \emph{reduced universal curve}. For the next result, we consider a collection $\{E_i\}_{i=1}^{n}$ of vector bundles on $S^2$ with connections, and identifications of the fibres at $0$ and $\infty$ with $\bC^n$. 

  \begin{lem} \label{lem:family_complex_bundles_universal_curve}
    The reduced universal curve over $ \Sbar_{n} $ carries a complex bundle whose restriction to each component that carries an odd marked point is canonically isomorphic to $E_i$. This construction is compatible with the gluing map in Equation \eqref{eq:gluing_discs_boundary_marked_points} in the sense that there is a natural isomorphism between the pullback to universal curve over $\Sbar_{n} \times     \Sbar_{m} $ of the bundle associated to $ \Sbar_{n+m}$ and the result of gluing the two bundles along the trivial component.
  \end{lem}
  \begin{proof}
  Parametrise each sphere bubble carrying an odd marked point $z_{2i-1}$ by $S^2$, with $1$ mapping to $z_{2i-1}$, with $0$ mapping to the node connecting this component to the disc component, and with $\infty$ mapping either  $z_{2i}$, or to the node separating this component from $z_{2i}$. This determines the bundle on this component, which varies smoothly in moduli since the identification of this component with $S^2$ itself varies smoothly in the limit where a bubble approches the point $z_{2i}$. Using the connection to parallel transport along the negative real axis from the strip to each node (along the shortest chain of curves connecting them), we obtain smoothly varying maps that determine the identification of fibres at the nodes, thus defining a vector bundle over the universal curve. This construction is evidently compatible with the gluing maps.
  \end{proof}

  We can now apply this construction to the proof of the main result of this subsection:
  \begin{proof}[Proof of Proposition \ref{prop:framed-ham-flow}]
    We proceed by induction on the  integer $d$. In the inductive step, the key point is to  construct maps from the collars of  $\hat{X}(p,r)^{(d)}$ to the spaces $\Sbar_k$: each collar is associated to a stratum labelled by a sequence $(q_1, \ldots, q_{k-1})$ of intermediate orbits, and a decomposition $d = d_1 + \ldots + d_k$. Starting with the collars of the highest codimension strata, we choose a map from these collars to $\Sbar_{k+1}$, so that, for each choice of element $1 \leq i < k$, the following properties hold:
  \begin{enumerate}
  \item the inclusion of the collar of $ \hat{X}(p,q_1)^{(d_1)} \times \cdots \times     \hat{X}(q_{k-1},r)^{(d_{k})}$ in $ X(p,q_{i})^{(d_1 + \ldots + d_{i})}  \times   X(q_i,r)^{(d_{i+1} + \cdots d_{k})}$ as a codimension $1$ factor of its collar in $\hat{X}(p,r)^{(d)} $   gives rise to a commutative diagram
    \begin{equation}
      \begin{tikzcd}
        {[}0,1{]}^{i} \times [0,1]^{k-i-1} \ar[r] \ar[d] &  {[}0,1{]}^{k} \ar[d] \\
        \Sbar_{i+1} \times \Sbar_{k-i} \ar[r] & \Sbar_{k+1}.
      \end{tikzcd}
    \end{equation}
  \item 
    along the boundary facet meeting the collar of
    \begin{equation}
            X(p,q_1)^{(d_1)} \times \cdots  \times X^{(d_i + d_{i+1})}(q_{i-1}, q_{i+1})  \times  \cdots \times   X(q_{k-1},r)^{(d_k)}
          \end{equation}
          the map lie in the stratum $ \Sbar_{P}$ associated to the partition consisting of singletons, except for the pair $(2i-1,2i)$. 
        \end{enumerate}

The second property allows us to require in addition the compatibility condition that the projection map $\Sbar_{P} \to \Sbar_{n-1} $ associated to forgetting the $i$\th element intertwine the choices made on the two collars. Note that these properties implies further conditions in higher codimension, which can be made explicit by Lemma \ref{lem:moduli_partition_intersection}.

        Applying Lemma \ref{lem:family_complex_bundles_universal_curve} to this situation then yields an explicit realisation of a family of operators with Lagrangian boundary condition, whose index is the realification of the complex operator defining the virtual bundle $I^{\pi}(p,r)^{d}$ from Section \ref{sec:stable-fibr-compl}. As before, there is a deformation of this family to the trivial family, which is determined up to contractible choice by the extension of the Lagrangian boundary conditions on the disc. We choose this homotopy by induction to be compatible with the composition maps in the flow category, which is possible, despite the fact that the two bundles over stable spheres labelled by a boundary stratum do not have the same gluing maps across the nodes, by the fact that they nonetheless have canonically isomorphic index bundles and trivialisations, as in Lemma  \ref{lem:homotopy_codim_1}.

        At this stage, the remainder of the argument follows by complete analogy with the construction of the stable complex structure: we inductively choose finite dimensional vector spaces $W^{\pi}_{fr}(p,r)^{(d)}$  mapping surjectively to the cokernel of all operators in the family connecting the linearisation of the Floer equation (glued to the operator $D_p$), to the direct sum of the trivial operator with $D_r$, picking as well embeddings
        \begin{equation}
                   W^{\pi}_{fr}(p,q)^{(d_1)} \oplus  W^{\pi}_{fr}(q,r)^{(d_2)}   \to W^{\pi}_{fr}(p,r)^{(d)} 
        \end{equation}
which are compatible with triple compositions.  The direct sum of this vector space  with the Lie algebra $\mathfrak{u}_d $ then yields the desired vector space and structure maps required to establish the framing.

  \end{proof}

\subsubsection{Invariance of structured flow categories}
\label{sec:invar-struct-flow}

We complete this Appendix by discussion how to incorporate the complex structure and the framing in the discussion of Sections \ref{sec:cont-equat}--~\ref{sec:comp-cont-maps}:
\begin{proof}[Proof of Proposition \ref{prop:Hamiltonian_Flow_category}]
The essential point is to construct a stable complex structure on the stack quotient $\bar{B}_{d_-,d_+}(1)$ from Section \ref{sec:doubly-framed-curves}; this takes the form of an isomorphism
  \begin{equation}
  T  \bar{B}_{d_-,d_+}(1) \oplus \mathfrak{u}_{d_-,d_+} \oplus \bR \cong  I^B_{d_-,d_+} \oplus \bR,
\end{equation}
induced by the analogue of Equation \eqref{eq:stable_complex_base} for doubly framed curves where the $\bR$ factor on the right corresponds to the fibre of the projection map to $\bar{B}_{d_-,d_+}$, and $ I^B_{d_-,d_+}$ is a complex bundle, which is essentially the tangent space of the total space of the complex line bundle over the space of rational curves with two marked points in $\bC \bP^{(d_1)} \times \bC \bP^{(d_2)} \times \bC \bP^{(d_3)}$, associated to the tangent space of one of the marked points. Going back to our conventions for orienting flow simplices in Definition \ref{def:framed_flow_simplex-elementary}, the $\bR$ factor on the left will be interpreted as $\bR^{\{r_+\}}$ (when considering a continuation map with output $\bX_+$), while the $\bR$ factor on the right is interpreted at $\bR^{\{1\}}$ (when defining a flow $1$-simplex with initial vertex $0$ and terminal vertex $1$). The compatibility of this stable complex structure, at the boundary, with that of the moduli spaces $\bar{B}_{d}$ then follows by the same arguments as in Section \ref{sec:orient-tang-space}. The fibres of the projection map $X_{-+}(p_-,r_+) \to  \bar{B}_{d_-,d_+}(1)$ is equipped with a stable complex structure in exactly the same way as in Section \ref{sec:stable-fibr-compl}, which then induces a lift of  $\hat{\bX}_{-+}$ to an edge in $\Flow^{U}$, and similarly for  $\hat{\bX}_{+-}$.

Finally, to lift $\hat{\bX}_{-+-}$ to a simplex in $\Flow^{U}$, we have to orient its base $ \bar{B}_{d_1,d_2,d_3}(2)$ which labels equivalence classes of triply framed curves with two additional marked points lying along the real axis between the two framed point. Its tangent space is isomorphic to a direct sum 
  \begin{equation}
  T  \bar{B}_{d_1,d_2,d_3}(2) \oplus \mathfrak{u}_{d_1} \oplus \mathfrak{u}_{d_2}\oplus \mathfrak{u}_{d_3}  \oplus \bR \cong  I^B_{d_1,d_2,d_3} \oplus \bR^{2},
\end{equation}
where the $\bR^{2}$ factor on the right hand side is specifically split as the direct sum of two copies, corresponding to each of the marked points. Taking the direct sum with the isomorphisms of fibrewise tangent spaces, this provides us with the isomorphism required in Definition \ref{def:framed_flow_simplex-elementary}, and completes the proof that the flow categories associated to $H_-$ and $H_+$ are equivalent in $\Flow^U$, using as well the argument where their r\^ole is reversed.
\end{proof}

\bibliographystyle{alpha}
\bibliography{Flow_Categories.bbl}

\end{document}